\documentclass[11pt]{article}

\usepackage{hyperref}
\usepackage{booktabs}
\usepackage{multirow}
\usepackage{subfigure}
\usepackage{graphicx}           
\usepackage{caption}
\usepackage{tikz}
\usepackage{tikz-qtree} 
\usetikzlibrary{matrix}
\usepackage{amsfonts}
\usepackage{mathdots}
\usepackage{amsmath}
\usepackage{amsthm}  % not compatible with siam class
\usepackage{amssymb}
\usepackage{glossaries} 
\usepackage{mathtools}
\usepackage{enumerate} 
\usepackage{stmaryrd} 
\usepackage{xcolor} 
\usepackage{algorithmic}
\usepackage{algorithm}
\usepackage[margin=1in]{geometry} 
\usepackage{eqparbox}

\DeclarePairedDelimiter\abs{\lvert}{\rvert}%
\DeclarePairedDelimiter\norm{\lVert}{\rVert}%

\newcommand{\cardn}{\text{card}}
\newcommand{\hmat}{\mathcal{H}}

\newcommand{\rank}[1]{\text{rank}( #1 )}

\makeatletter
\let\oldabs\abs
\def\abs{\@ifstar{\oldabs}{\oldabs*}}
\let\oldnorm\norm
\def\norm{\@ifstar{\oldnorm}{\oldnorm*}}
\makeatother

\newtheorem{theorem}{Theorem}[section]

\title{Data-Driven Construction of Hierarchical Matrices with Nested Bases}
\author{Difeng Cai\thanks{Department of Mathematics, Emory University, Atlanta, GA 30322 ({dcai7@emory.edu}, {yxi26@emory.edu}). The research of Difeng Cai and Yuanzhe Xi is supported by NSF award OAC 2003720 and RTG Grant DMS-2038118.}
\and Hua Huang \thanks{School of Computational Science and Engineering, Georgia Institute of Technology, Atlanta, GA 30332 ({huangh223@gatech.edu}, {echow@cc.gatech.edu}). The research of Hua Huang and Edmond Chow is supported by NSF award OAC 2003683.}
\and Edmond Chow$^{\dagger}$
\and Yuanzhe Xi$^{\ast}$}
\date{}

\begin{document}
\maketitle
\begin{abstract}
Hierarchical matrices provide a powerful representation for significantly reducing the computational complexity associated with dense kernel matrices. For general kernel functions, interpolation-based methods are widely used for the efficient construction of hierarchical matrices. In this paper, we present a fast hierarchical data reduction (HiDR) procedure with $O(n)$ complexity for the memory-efficient construction of hierarchical matrices with nested bases where $n$ is the number of data points. HiDR aims to reduce the given data in a hierarchical way so as to obtain $O(1)$ representations for all nearfield and farfield interactions.
Based on HiDR, a linear complexity $\mathcal{H}^2$ matrix construction algorithm is proposed. 
The use of data-driven methods enables {better efficiency than other general-purpose methods} and flexible computation without accessing the kernel function.
{Experiments demonstrate significantly improved memory efficiency of the proposed data-driven method compared to interpolation-based methods over a wide range of kernels.}
{Though the method is not optimized for any special kernel, benchmark experiments for the Coulomb kernel show that
the proposed general-purpose algorithm offers competitive performance for hierarchical matrix construction 
compared to several state-of-the-art algorithms for the Coulomb kernel.}
\end{abstract}

%\begin{keywords}
%Hierarchical matrix, kernel matrix, data reduction, data-driven construction, complexity analysis, Gaussian
%\end{keywords}
%
%\begin{AMS}
%15A23, 68W25, 65D40
%\end{AMS}

\section{Introduction} 
\label{sec:intro}

In various applications, the pairwise interaction between two objects is characterized by a non-local kernel function. A system of $n$ objects then gives rise to a $n$-by-$n$ dense kernel matrix. 
Such matrices arise frequently in integral equations \cite{colton1983integral,rokhlinpotential,rokhlinscattering}, astrophysics \cite{BarnesHut}, statistics \cite{multiquadric2004,eigCMAM}, machine learning \cite{bishopbook,svm2002}, etc.
A computational challenge is that the naive computational or storage cost associated with the dense kernel matrix is at least $O(n^2)$.
For the Coulomb kernel and its variants, pioneering work such as the Fast Multipole Method (FMM) \cite{rokhlinpotential,rokhlinscattering,GREENGARD1997280} and the Barnes-Hut algorithm \cite{BarnesHut}
use multilevel approximation to successfully reduce the cost to linear or quasilinear complexity.
These techniques were later generalized into the powerful algebraic framework of hierarchically low-rank matrices \cite{h2lec,hackweakadmissibility,hackintroh2app,hack2015book} for efficiently approximating general dense kernel matrices with nearly optimal computational cost.
Two widely used classes of hierarchical matrices are $\hmat$ matrices and $\hmat^2$ matrices.
The $\mathcal{H}$ matrix yields an $O(n\log n)$ representation and the $\mathcal{H}^2$ matrix yields an optimal $O(n)$ representation for approximating the kernel matrix and computing the matrix-vector multiplication.
For dense kernel matrices, the efficient construction of these hierarchical representations associated with general kernels remains a challenging problem.
In this paper, we address this issue for the $\mathcal{H}^2$ matrix representation.
For general kernels, interpolation-based methods are often adopted as a black-box tool used in hierarchical matrix construction.
However, as will be demonstrated in later sections, the use of interpolation nodes (or points outside the given dataset) may lead to  loss of accuracy.
We present a general-purpose data-driven framework for hierarchical matrix construction, that resolves these issues and offers improved efficiency.

%Disadvantages of existing methods using virtual points:
%Kernel evaluation can be expensive, or it may not be available outside given data, $\kappa(x,y)$ is undefined!
%In physics, hydrology, or statistics applications, evaluation outside given data is not interpretable. 
%For example, for stochastic modelling of subsurface flow, the correlation function $\kappa(x,y)$ is meaningless if $x$ (or $y$) is outside the computational domain.

Given a set of points and a kernel function, hierarchical matrix construction starts with an adaptive partitioning of the data.
The partitioning can be encoded by a tree structure in which each node represents a subset generated in the adaptive partitioning procedure.
For example, the root node corresponds to the entire set of $n$ points and its children nodes correspond to the subsets generated from the first partitioning of the dataset.
For hierarchical matrix construction, directly compressing the submatrix corresponding to a pair of nodes/subsets will be inefficient, since one of the subsets may contain $O(n)$ points (see Figure \ref{fig:YiYis1}) and the total cost for all such pairs will be at least $O(n^2)$.
%Existing methods circumvent this issue via various techniques, including analytic approximation (kernel expansion or interpolation), compressing an intermediate matrix (cite xxxxxx), greedy algebraic approximation (ACA).
%All those techniques require kernel function evaluation 
We propose to first process the tree-structured data to obtain a reduced representation in which each node only corresponds to a small subset with $O(1)$ points while the farfield corresponds to $O(1)$ points as well.
These $O(1)$ subsets are called \emph{representor sets} in \cite{bebendorf2012}.
See Figure \ref{fig:YiYis2}.
Using representor sets, the multilevel compression of the kernel matrix can be rapidly computed.
To achieve optimal efficiency, we design a hierarchical data reduction procedure with computational cost of $O(n)$.
Different from existing approaches, the procedure operates entirely on the given data without accessing the kernel function and no \textit{algebraic} compression is performed.
Hence the approach is termed \emph{data-driven}.
%thus can be computed rapidly.
%and can be used independently from the construction of hierarchical matrix representation.
%In practice, the hierarchical data reduction takes less time than the subsequent hierarchical matrix construction.
We show how to incorporate hierarchical data reduction in hierarchical matrix construction to obtain a fast algorithm for general kernels.
Numerical experiments demonstrate the competitive performance of the new method in terms of generality, memory use, speed, and accuracy.

\begin{figure}[htbp] 
    \centering 
    \subfigure[Left: the farfield (shown in blue) of the orange box contains $O(n)$ points. Right: the $O(n)$ points are reduced to a subset of $O(1)$ points. (For another orange box, the farfield is different and the reduced subset is generally a different set of points.)]{\label{fig:YiYis1}\includegraphics[width=65mm]{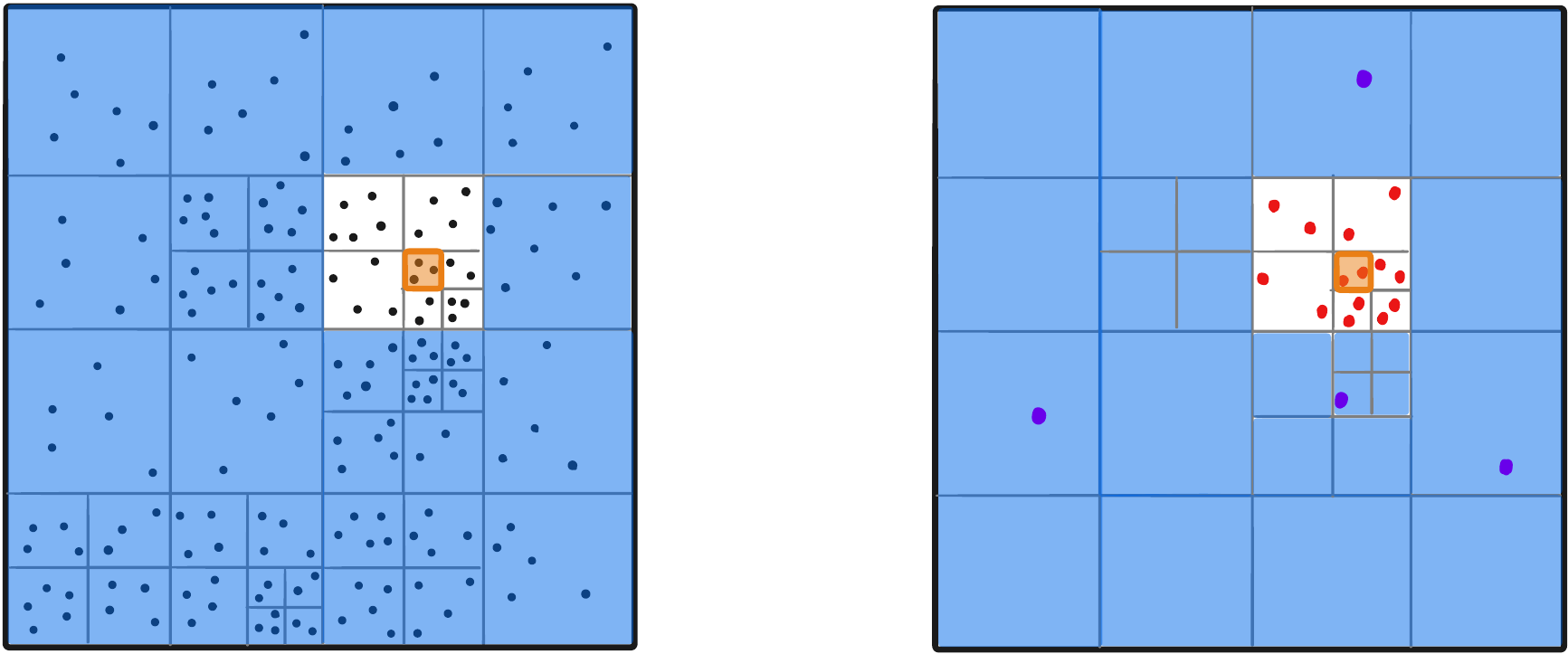}}
    \hspace{0.2in}
    \subfigure[The kernel matrix is shown, highlighting the interaction between points in the orange region and points in its farfield. The arrows point to columns corresponding to the points selected in the $O(1)$ subset. These columns form an approximate basis for the farfield.]{\label{fig:YiYis2}\includegraphics[width=65mm]{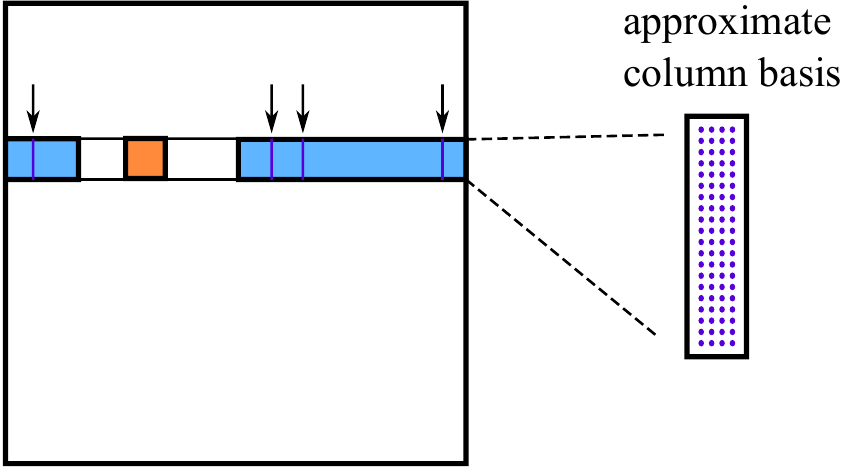}}
    \caption{Farfield data reduction as a way to construct an approximate column basis for the farfield block row in the kernel matrix}
    \label{fig:YiYis}
\end{figure}

%\begin{figure}[htbp] 
%    \centering 
%    \includegraphics[scale=.45]{./figures/YiYis.pdf}
%    \hspace{0.4in}
%    \includegraphics[scale=.6]{./figures/basisDD.pdf}
%    \caption{The original $O(n)$ farfield (left: points in blue area) and the reduced $O(1)$ representation (middle: points in blue area) associated with the points in orange area. Right: corresponding farfield block in the kernel matrix and the approximate column basis derived from the reduced farfield.}
%    \label{fig:YiYis}
%\end{figure}

%\cdf{Existing methods: Methods like FMM are extremely efficient, but are limited to certain class of kernel functions related to fundamental solutions of constant-coefficent elliptic operators such as Laplacian. General methods like interpolation-based methods can be applied to general kernels, but the compression rank is not as good as FMM. Moreover, as will be shown in Section \ref{sub:overflow}, interpolation-based methods can cause memory overuse due to the use of interpolation points.}
%
%Goal: new hierarchical data reduction algorithm. The output can be used to rapidly derive hierarchical matrix representations. Controlled approximation, input is tolerance

The rest of the manuscript is organized as follows.
Section \ref{sec:review} reviews hierarchical matrix structures and existing methods for constructing hierarchical matrices.
%The motivation of the proposed data-driven method is illustrated in Section \ref{sub:overflow}.
Section \ref{sec:HiDR} introduces the hierarchical data reduction algorithm.
Based on this hierarchical data reduction,
the complete hierarchical matrix construction is presented in Section \ref{sec:ddH}.
{Section \ref{sec:dr1} presents numerical experiments to investigate different data reduction techniques.
Numerical results for the proposed data-driven hierarchical matrix construction are given in Section \ref{sec:numerical}. 
}
Concluding remarks are drawn in Section \ref{sec:conclusion}.

The notation used in this paper is listed below.
\begin{itemize}
    \item $|x-y|$ denotes the Euclidean distance between $x$ and $y$ in $\mathbb{R}^d$;
    \item diam($X$) denotes the diameter of set $X$, i.e., $\max\limits_{x,y\in X} |x-y|$;
    \item dist($X_1,X_2$) denotes the distance between sets $X_1$ and $X_2$, i.e. $\min\limits_{x\in X_1,y\in X_2} |x-y|$;
    \item card($X$) denotes the cardinality of set $X$.
    \item {$X^*$ denotes a representor set of $X$.}
\end{itemize}
%%%%%%%%%%%% section  (end) %%%%%%%%%%%%%%%%

\section{Review of Hierarchical Matrix Representations}
\label{sec:review}

Given a kernel function $\kappa(x,y)$ and a dataset $X=\{x_1,\dots,x_n\}$, the associated kernel matrix is defined by 
\[
    K = [\kappa(x_i,x_j)]_{i,j=1}^n.
\]
Dense kernel matrices are ubiquitous and arise in various applications, where the kernel function measures the interaction between objects.
In Coulombic N-body simulations, $\kappa(x,y)=\frac{1}{|x-y|}$.
In boundary integral equations, $\kappa(x,y)$ can take very different forms, including $\Phi(x,y)$, $\nabla_{v_y} \Phi(x,y)$, where $\Phi(x,y)$ denotes the fundamental solution of the underlying differential operator and $v_y$ denotes the unit outer normal at $y$ on the boundary. {In certain structured matrix computations such as those involving Cauchy and Cauchy-like matrices, $\kappa(x,y)$ is taken as $\frac{1}{x-y}$ with $x,y\in\mathbb{C}$.}
In statistics and machine learning, $\kappa(x,y)$ is often taken as the Gaussian kernel $e^{-|x-y|^2}$ or the Laplace kernel $e^{-|x-y|}$.
A major computational bottleneck in these applications lies in the $O(n^2)$ cost in storing the dense matrix $K$ and performing operations such as matrix-vector multiplication.

To avoid the high cost in forming $K$ explicitly, hierarchically low-rank matrix representations are used to approximate $K$.
The hierarchical representation is based on the fact that $K$ can be partitioned into blocks in a hierarchical fashion and many blocks are numerically low rank.
Low-rank factors are computed for the blocks and are stored in the hierarchical representation to replace original dense blocks.
Two widely used hierarchical representations are $\hmat$ and $\hmat^2$ \cite{h2lec,hackintroh2app,hackhss2002,hack2015book}.
The $\hmat$ matrix representation in general has a complexity of $O(n\log n)$ in space while $\hmat^2$ has the optimal complexity of $O(n)$.
The matrix-vector multiplication can be computed in $O(n\log n)$ complexity for $\hmat$ matrices and in $O(n)$ complexity for $\hmat^2$ matrices, which is much more efficient than directly multiplying $K$ by a vector.

We review the mathematical description of hierarchical matrices in Section \ref{sub:Hmat}.
Existing methods for constructing the hierarchical matrices are discussed in Section \ref{sub:methods}.

\subsection{Hierarchical matrix representations} 
\label{sub:Hmat}
In the following, we review the general algebraic framework of $\hmat$ and $\hmat^2$ matrices.

Given a dataset $X=\{x_i\}_{i=1}^n$ in $\mathbb{R}^d$,
one builds a tree structure by recursively partitioning $X$ spatially until no more than {$m=O(1)$ points are contained in each partitioned subset}.
The tree encodes the subsets of $X$ generated by the adaptive partitions.
Namely, the root node is associated with $X$ and its children nodes are associated with subsets of $X$ created by the first partition. 
Inductively, each node is associated with a non-empty subset of $X$.
For node $i$, we denote by $X_i$ the subset associated with that node. Hence $X_{\text{root}}=X$. An illustration is given in Figure \ref{fig:partition}.

\begin{figure}[htbp]
    \centering 
    \includegraphics[scale=.3]{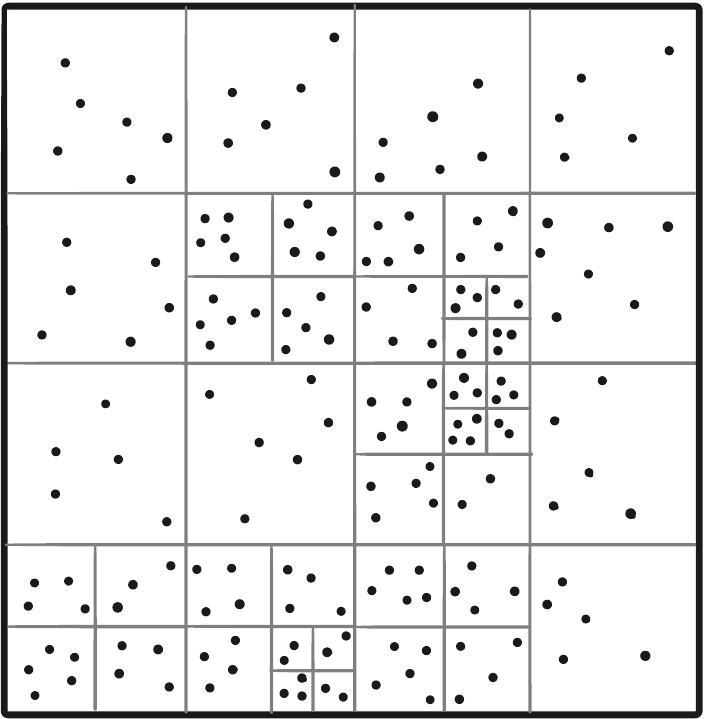} 
    \hspace{0.3in}
    \includegraphics[scale=.3]{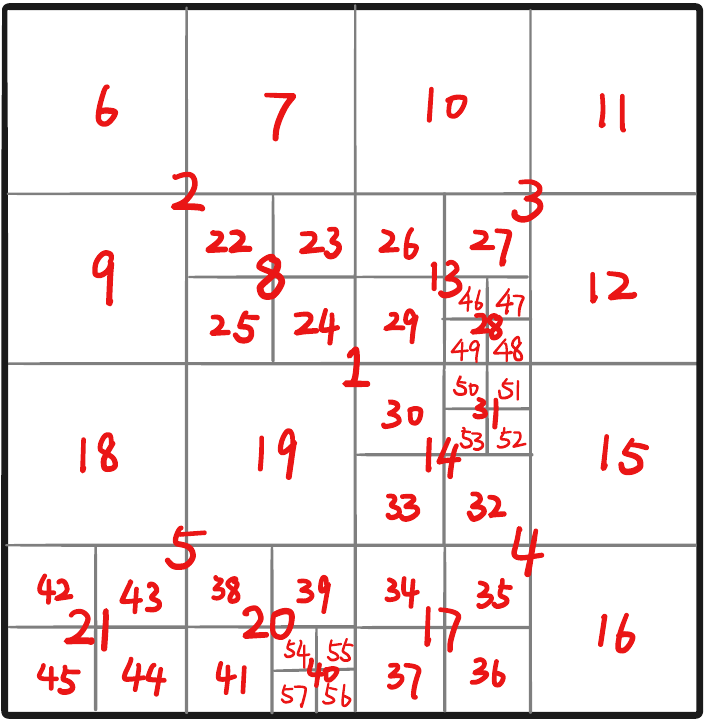} 
    \hspace{0.3in}
    \includegraphics[scale=.3]{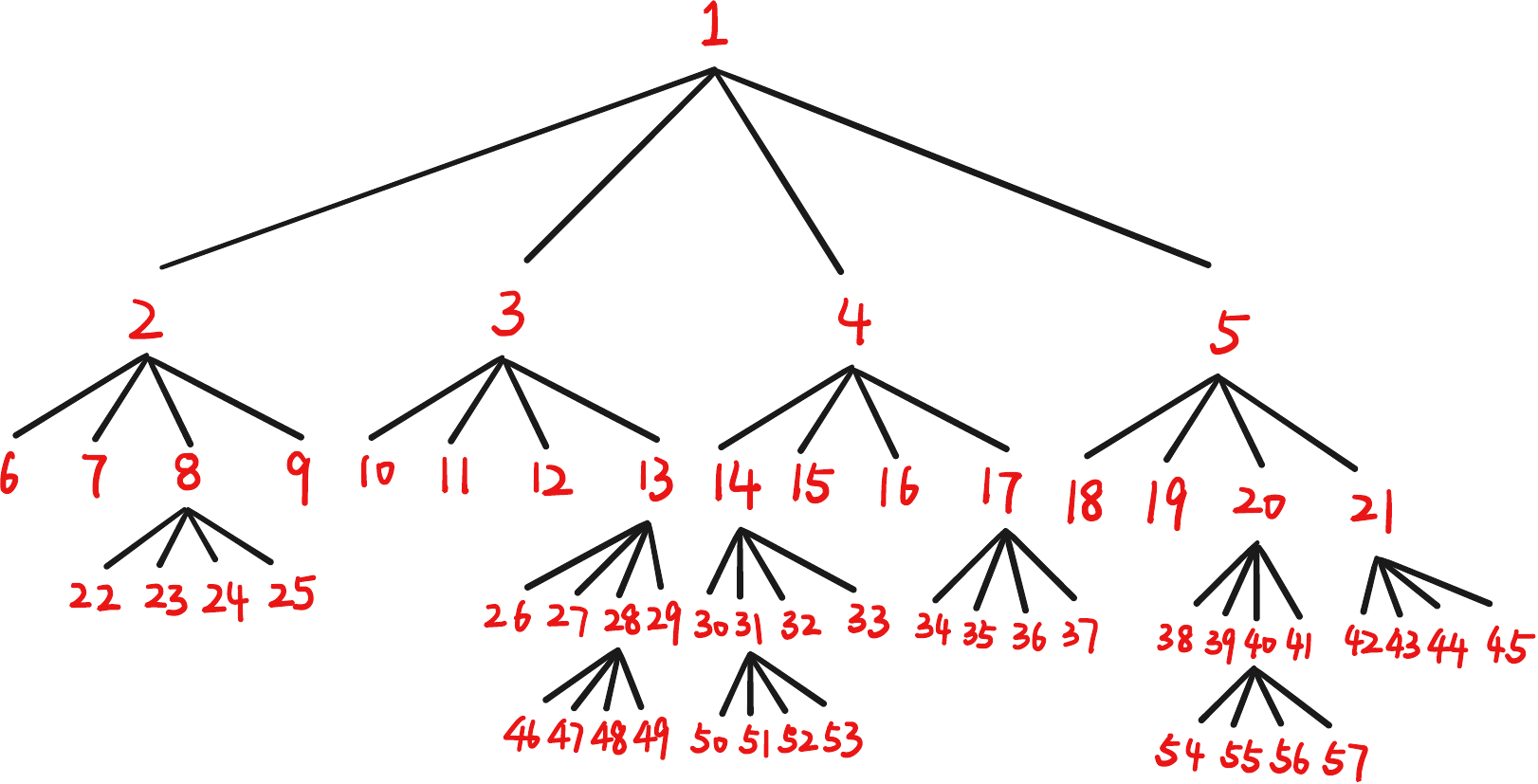} 
    \caption{Adaptive partition of the dataset (left), label (at center) for each subset (middle) and the associated partition tree (right).}
    \label{fig:partition}
\end{figure}

The tree structure automatically yields a blockwise partition of the matrix $K$.
See Figure \ref{fig:hmat1d} for an example with data points lying inside an interval.
We define $K_{i,j} = [\kappa(x,y)]_{\substack{x\in X_i\\ y\in X_j}}$.
Block $K_{i,j}$ is approximated by a low-rank factorization
if $(i,j)$ satisfies an admissiblity condition that requires $X_i$ and $X_j$ to be separated from each other to some extent.
For example, the pair $(i,j)$ is considered admissible if 
\begin{equation}
\label{eq:separation}
\text{diam}(X_i)+\text{diam}(X_j)\leq 2\tau |a_i-a_j|
\end{equation}
for some $\tau\in [0,0.7]$, where $a_i$ denotes the center of the box associated with $X_i$ in the partition.
The admissiblity condition in \eqref{eq:separation} is used in \cite{xiaobaiFMM,smash}.
The parameter $\tau$, often called the \emph{separation ratio} \cite{xiaobaiFMM}, controls how separated the two subsets are. Smaller $\tau$ implies better separation and the value $0.7$ can be replaced by other values less than 1 (but not close to one). {For any admissible $(i,j)$, the corresponding submatrix $K_{i,j}$ is called an \emph{admissible block}.}
An admissible block is also referred as a \emph{farfield} block.
%If $i$, $j$ are both leaf nodes and $(i,j)$ is not admissible, then the corresponding submatrix $K_{i,j}$ is called a \emph{nearfield} block.
Non-admissible blocks are referred as \emph{nearfield} blocks.
For a subset $X_i\subset X$, the union of all $X_j$ that are separated from $X_i$ in the sense of \eqref{eq:separation} is called the \emph{farfield} of $X_i$.
Points that are not in the farfield of $X_i$ constitute the \emph{nearfield} of $X_i$. 
In hierarchical matrices, each admissible block is approximated by a low rank factorization,
$$K_{i,j}\approx U_i B_{i,j} V_j^T$$
where $U_i$, $V_j$ are column and row basis matrices, and $B_{i,j}$ is called a  \emph{coupling matrix}. 
The maximum column size of all basis matrices $U_i$ and $V_j$ is the approximation rank for $K$.
A larger approximation rank yields a more accurate approximation.
For each node $i$, the \emph{interaction list} of $i$ consists of nodes $j$ such that $X_j$ is well-separated from $X_i$ but for the parent $p$ of $j$, $X_p$ is not well-separated from $X_i$.
The interaction list specifies the blockwise partition of the matrix in which each admissible block has a low-rank approximation.
For example, in Figure \ref{fig:hmat1d}, the interaction list of node $4$ is $\{6,7\}$;
the interaction list of node $10$ is $\{8,12,13\}$;
The number of elements in any interaction list is bounded by the so-called sparsity constant $c_{\text{sp}}$, which is independent of $n$.
The blockwise low-rank representation yields the $\hmat$ matrix structure, which stores all nearfield blocks and low-rank factors $U_i,B_{i,j}, V_j$.
$\hmat$ matrices generally admit $O(n\log n)$ storage complexity.
The more refined $\hmat^2$ structure requires the basis matrices to be \emph{nested}. 
The nested bases property states that, if node $p$ has children $c_1,\dots,c_k$, then there exist \emph{transfer matrices} $R_{c_i}$, $W_{c_i}$ such that 
\[
    U_p = \begin{bmatrix}
        U_{c_1}R_{c_1}\\
        \vdots\\
        U_{c_k}R_{c_k}\\
    \end{bmatrix},\quad 
    V_p = \begin{bmatrix}
        V_{c_1}W_{c_1}\\
        \vdots\\
        V_{c_k}W_{c_k}\\
    \end{bmatrix}.
\]
See Figure \ref{fig:nestedU} for an illustration.
The sizes of the transfer matrices are equal to $r$ if rank-$r$ factorizations are used for approximations to admissible blocks. {In practice, one uses $r=O(1)$ independent of $n$.}
Due to the nested bases property, only basis matrices $U_i$, $V_i$ associated with leaf nodes need to be stored in an $\hmat^2$ representation, in addition to all transfer matrices (whose row and column sizes are $O(1)$) and nearfield blocks. This results in $O(n)$ storage cost.

\begin{figure}[htbp] 
    \centering 
    \includegraphics[scale=.45]{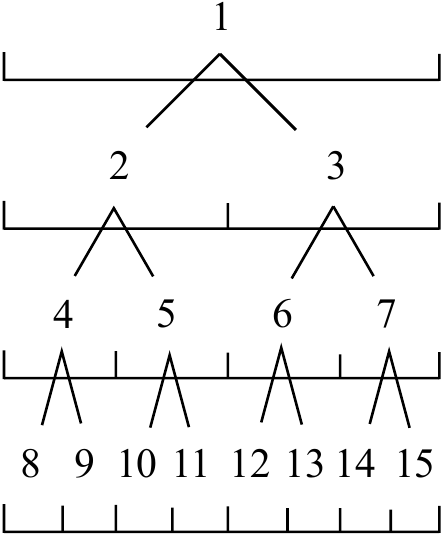}
    \hspace{0.2in}
    \includegraphics[scale=.45]{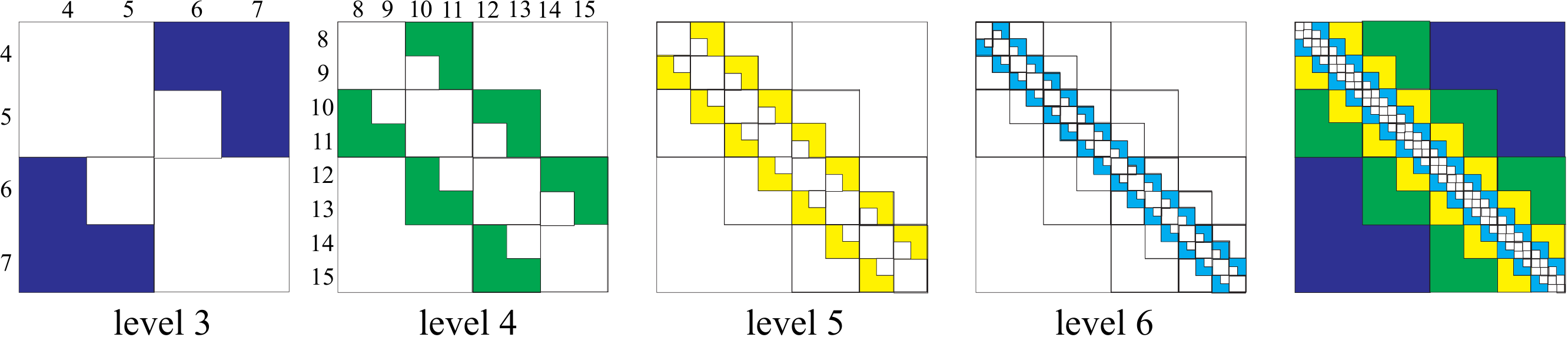}
    \caption{Hierarchical matrix structure for 1D problem (left to right): tree (only top 4 levels are plotted), admissible blocks (colored) at levels 3,4,5,6 and all admissible blocks (colored) in the kernel matrix} 
    \label{fig:hmat1d} 
\end{figure}

\begin{figure}[htbp] 
    \centering 
    \includegraphics[scale=.6]{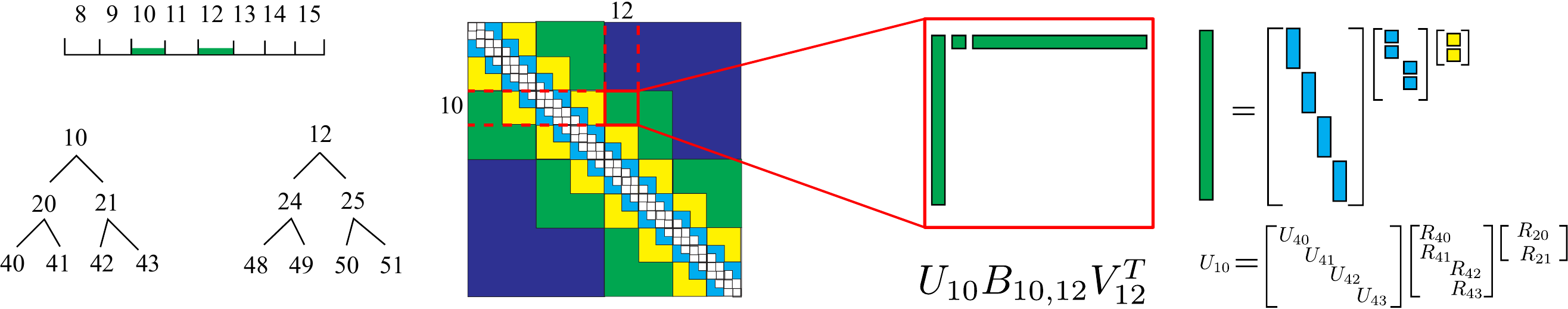}
    \caption{Admissible block (10,12) and nested bases}
    \label{fig:nestedU}
\end{figure}

Note that the above complexity estimates for storing hierarchical representations assume that the low-rank factors have already been computed.
%Computing all those hierarchical low-rank factors is usually the most costly step in hierarchical matrix techniques.
In practice, computing these hierarchical low-rank factors is usually the most costly step, compared to applying the hierarchical representation to a vector.
Extensive research has focused on the efficient computation of the hierarchical representation.
We review several state-of-the-art methods in Section \ref{sub:methods}.

% subsection Hierarchical matrix representations (end)

\subsection{General-purpose methods}
\label{sub:methods}
For an $n$-by-$n$ kernel matrix $K$ associated with a general kernel function $\kappa(x,y)$,
the hierarchical matrix representation can be computed in linear or quasilinear time with a variety of techniques, including interpolation \cite{h2lec,hack2015book}, adaptive cross approximation (ACA) \cite{bebendorf2000,bebendorf2012}, hybrid cross approximation (HCA) \cite{HCA2005}, SMASH \cite{smash}, etc.
These methods work for general kernels (for example, non-symmetric, non-translationally-invariant).
For special kernel functions, such as fundamental solutions of certain elliptic partial differential equations (PDEs), layer potentials in integral equations, Coulomb interactions in electrostatics,
analytic methods such as the fast multipole method \cite{rokhlinpotential,rokhlinscattering,GREENGARD1997280,xiaobaiFMM,fmm1d} and its variants \cite{anderson92,Ying04} can be used to construct a hierarchical matrix representation efficiently.
%The most dominant cost in computing hierarchical low-rank representations is the computation of low-rank approximations.

When constructing the hierarchical matrix with \emph{nested bases}, 
one needs to compute basis matrices $U_i$, $V_i$ for submatrices $K_{X_i Y_i}$ that account for the interaction between $X_i$ with $O(1)$ points and its entire farfield $Y_i$ with $O(n)$ points (see Figure \ref{fig:YiYis1}).
In order to achieve $O(n)$ optimal complexity for building the hierarchical representation, each basis matrix $U_i$ (as well as $V_i$) must be computed with $O(1)$ complexity. 
This requires that the matrix $K_{X_i Y_i}$, which is $O(1)$-by-$O(n)$ in size, must not be formed.
%{Since $K_{X_i Y_i}$ is $O(1)$-by-$O(n)$ in size, algebraic methods such as singular value decomposition (SVD), QR or LU factorizations require $O(n)$ complexity to compute all entries in $K_{X_i Y_i}$ and thus at least $O(n^2)$ complexity for the overall hierarchical matrix construction (with $O(n)$ such nodes $i$).}
For general kernels, a widely used technique to construct a column basis matrix $U_i$ in $O(1)$ complexity is based on interpolation.
For $K_{X_i Y_i}$, interpolating the kernel function $\kappa(x,y)$ at $r$ nodes $Q=\{q_1,\dots,q_r\}$:
$$
    \kappa(x,y)\approx \sum_{i=1}^r \kappa(q_i,y)L_i(x)
$$
yields a rank-$r$ approximation:
\begin{equation}
\label{eq:interpolation}
K_{X_i Y_i} \approx U_i K_{Q Y_i},
\end{equation}
where $L_i$ is the Lagrange polynomial corresponding to node $q_i$ and $U_i = [L_k(x)]_{\substack{x\in X_i\\ k=1:r}}$.
%$$
%U_i = \begin{bmatrix}
%    L_1(x_{n_1}) & \dots & L_r(x_{n_1})\\
%    \vdots & \dots & \vdots\\
%    L_1(x_{n_i}) & \dots & L_r(x_{n_i})\\
%\end{bmatrix}.
%$$
Due to its generality and efficiency in computing $U_i$, this interpolation method is used in a number of general-purpose hierarchical matrix algorithms, e.g. \cite{h2lec,hackintroh2app,HCA2005,smash}.

Another way of finding column basis matrix for $K_{X_i Y_i}$ is through subset selection \cite{bebendorf2012,smashIPDPS}. 
The column basis matrix is chosen as the submatrix corresponding to a judiciously chosen $O(1)$ subset $Y_i^*$ from $Y_i$.
Similar to interpolation, subset selection can be applied to general kernel functions.
The reference \cite{bebendorf2012} presents an efficient hierarchical scheme to select representor sets for all nodes using two steps: top-down and bottom-up.
The cost of the algorithm in \cite{bebendorf2012} is dominated by computing farfield representor sets in the top-down step.
%for all $X_i$ in the top-down step. 
%If $X_i$ contains $O(n)$ points, then the total farfield points from its interaction list generally also contain $O(n)$ points, and computing representor sets for the farfield of $X_i$ entails $O(n)$ complexity using the algorithm in \cite{bebendorf2012}.
Computing representor sets for all nodes $i$ leads to $O(n\log n)$ complexity for a balanced tree with $O(\log n)$ levels.
The total complexity of the resulting hierarchical matrix construction is $O(n\log n)$, instead of the optimal $O(n)$ complexity achieved by interpolation-based methods.
However, compared to using interpolation nodes (which are generally \emph{outside} the given dataset), selecting subsets directly from the dataset is more memory-efficient for low-rank approximation.
%Furthermore, in some cases, using submatrices in the low-rank approximation can be much more accurate than using a kernel matrix with out-of-data points.
See Section \ref{sec:numerical} for a detailed discussion.

\section{Fast Hierarchical Data Reduction (HiDR)}
\label{sec:HiDR}

To facilitate the fast construction of hierarchically low rank representations, we propose an efficient preprocessing scheme to reduce the tree-structured data so that each node in the partition tree induces $O(1)$  cost in the subsequent hierarchical matrix construction process. Specifically, let $X_i$ be the set of points corresponding to node $i$ and $Y_i$ be the farfield of $X_i$. The data reduction aims to find representor sets $X_i^*\subset X_i$ with $O(1)$ points and $Y_i^*\subset Y_i$ with $O(1)$ points for each node $i$. 
Note that a naive data reduction for $Y_i$ with $O(n)$ points {into a subset of evenly spaced points} as shown in Figure \ref{fig:YiYis} will lead to $O(n)$ computational complexity. 
The cost can be reduced to $O(1)$ with a carefully designed hierarchical procedure presented in Section \ref{sub:HiDR}.
%Since certain nodes are associated with subsets with $O(n)$ points (for example, those close to the root node), a low-rank block $K_{X_i Y_i}$ can have size $O(n)$-by-$O(n)$, if $X_i$ has $O(n)$ points and its entire farfield $Y_i$ has $O(n)$ points.

%\cdf{
%In \cite{bebendorf2012}, an efficient hierarchical procedure is proposed with a top-down reduction for the farfield followed by a bottom-up reduction for the nearfield. 
%The overall complexity of the algorithm is $O(n \log n)$ since the cost for reducing $Y_i$ to $Y_i^*$ is $O(\cardn(Y_i))$, instead of $O(1)$, for each node $i$. As a result, processing each level of the partition tree requires $O(n)$ complexity, which sums to $O(n\log n)$ for processing all nodes.
%Moreover, in \cite{bebendorf2012}, algebraic compression (ACA) is used in the data reduction, which requires computing submatrices of the kernel matrix.}

We present the hierarchical data reduction algorithm in Section \ref{sub:HiDR} and verify that it scales linearly with the size of the data in Section \ref{sub:Complexity}. Several algorithms for performing {data reduction} are discussed in Section \ref{sub:DR}.
%Since HiDR is independent of the kernel function, the resulting hierarchical matrix construction algorithm is applicable to a wide range of kernel functions.
%It yields a rapid $\hmat$ matrix construction without any analytic or algebraic compression involved.
%It can be seamlessly incorporated into existing $\hmat^2$ matrix construction framework such as SMASH \cite{smash} and offers significant speed-up.
%Empirically, numerical experiments demonstrate that the preprocessing results in improved accuracy of the hierarchical matrix approximation.

\subsection{Linear complexity hierarchical data reduction}
\label{sub:HiDR}
%We present a fast hierarchical scheme with $O(n)$ complexity to reduce the sizes of $X_i$ and $Y_i$ to $O(1)$ for each node $i$ in the adaptive partition tree.
The fast HiDR consists of two traversals of the tree: bottom-up and top-down. The $O(1)$ representor set $X_i^*$ for $X_i$ is computed in the bottom-up pass and the $O(1)$ farfield representor set $Y_i^*$ for $Y_i$ is computed in the top-down pass. 
{A building block for the hierarchical scheme is a {\tt DataReduct} subroutine that takes the form:
\begin{equation*}
   {\tt DataReduct}(X,k)\to X^*,  
\end{equation*}
where $X^*$ is a subset (representor set) of the input $X$ and $k$ is a parameter that specifies the size of $X^*$.
There are several options for the subroutine {\tt DataReduct} to obtain $X^*$ from $X$. A detailed discussion is presented in Section \ref{sub:DR}.}
Given a set of points, the subroutine selects a subset of evenly spaced points whose size is bounded by a prescribed constant and scales linearly with the size of the input data. The HiDR is designed such that the input dataset for {\tt DataReduct} is always $O(1)$ in size.

%The first pass is a \emph{bottom-up} traversal of the tree and the second pass is a \emph{top-down} traversal of the tree.

\begin{figure}[htbp]
    \centering 
    \includegraphics[scale=.3]{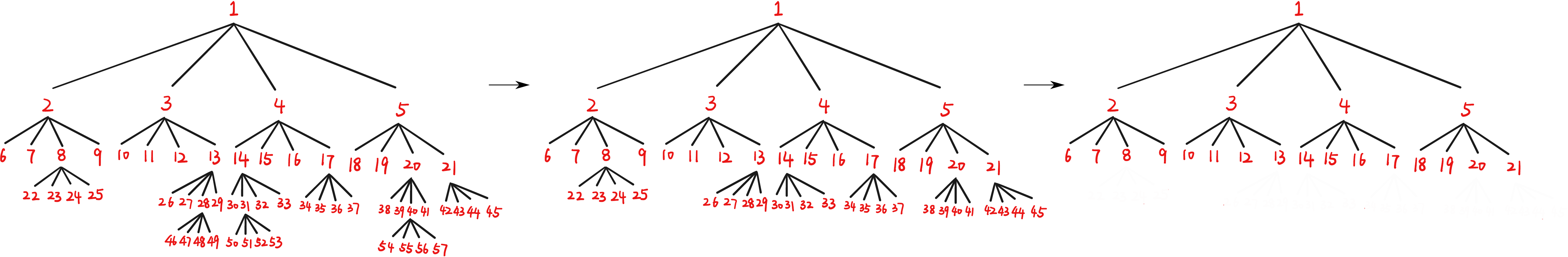} 
    \caption{Bottom-up pass of the tree: {representor sets $X_i^*$ are first computed for nodes $i$ at the deepest level and then upper levels. Nodes with $X_i^*$ computed are wiped out. 
The three trees from left to right correspond to the 1st, 3rd, 5th configurations in Figure \ref{fig:HiDRup}}.}
    \label{fig:treeUp}
\end{figure}

In the bottom-up sweep, starting from leaf nodes $i$, each $X_i$ contains $O(1)$ points and thus the data reduction from $X_i$ to $X_i^*$ induces $O(1)$ cost only.
After children nodes have been processed, we define for each parent $p$ an intermediate set $S_p$ as the union of reduced sets $X_i^*$ for all its children $i$.
The representor set $X_p^*$ is obtained by applying data reduction to the intermediate set $S_p$.
Since each parent has at most $C$ children ($C=2,4,8$ for binary tree, quadtree, octree, respectively) and each $X_i^*$ is $O(1)$ in size, the intermediate set $S_p$ is always $O(1)$ in size.
Thus the cost of computing the representor set for $p$ is always $O(1)$.
Recursively, all non-root nodes can be processed in $O(1)$ complexity in the bottom-up procedure {(see Figure \ref{fig:treeUp})}. Figure \ref{fig:HiDRup} shows the data reduction to generate $X_i^*$ from the lowest level to upper levels of the tree constructed in Figure \ref{fig:partition}  where each $X_i^*$ contains 2 points.

\begin{figure}[htbp] 
    \centering 
    \includegraphics[scale=.3]{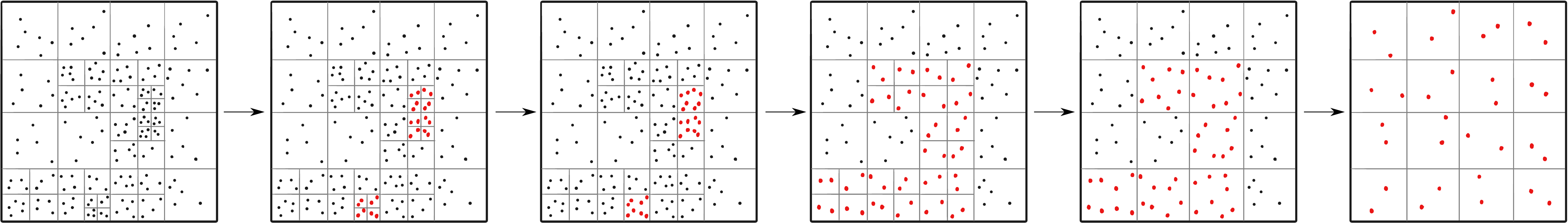} 
    \caption{Fast hierarchical data reduction: bottom-up pass. Each $X_i$ is reduced to $X_i^*$ such that $X_i^*$ contains at most 2 points. {The 1st, 3rd, 5th arrows correspond to the subroutine {\tt DataReduct} for the boxes at the bottom level of the tree, where points in red are output of {\tt DataReduct}. The 2nd, 4th arrows correspond to merging children boxes and going up in the tree.}}
    \label{fig:HiDRup}
\end{figure}

In the top-down sweep, starting from the top nodes $i$ in the tree with non-empty interaction list, we define the intermediate set $T_i$ as the union of $Y_p$ ($p$ denotes the parent of $i$) and $X_j^*$ for all nodes $j$ in the interaction list of $i$.
%$T_i$ is used as an initial representor set of $Y_i$.
The representor set $Y_i^*$ for the farfield $Y_i$ of $X_i$ is computed by applying data reduction to the intermediate set $T_i$.
Since the cardinality of the interaction list of $i$ is bounded by the sparsity constant $c_{\text{sp}}$, which is $O(1)$, and each $X_j^*$ is already $O(1)$ in size, we see that $T_i$ only contains $O(1)$ points and consequently computing $Y_i^*$ has $O(1)$ cost.
Once parent nodes have been processed, we define the intermediate set $T_i$ for each child $i$ as the union of all $X_j^*$ from its interaction list and $Y_p^*$ from its parent $p$.
From this definition, $T_i$ is also $O(1)$ in size.
Similar to the above, the representor set $Y_i^*$ is then obtained by applying data reduction to $T_i$. 
Recursively, each reduced farfield representation $Y_i^*$ can be computed in $O(1)$ complexity via the top-down procedure.
Figure \ref{fig:HiDRdown} shows the top-down data reduction for farfield to generate $Y_i^*$ (circled), as $i$ goes from a node near the root node to a leaf node. In Figure \ref{fig:HiDRdown}, each $Y_i^*$ contains at most 4 points only.

\begin{figure}[htbp] 
    \centering 
    \includegraphics[scale=.3]{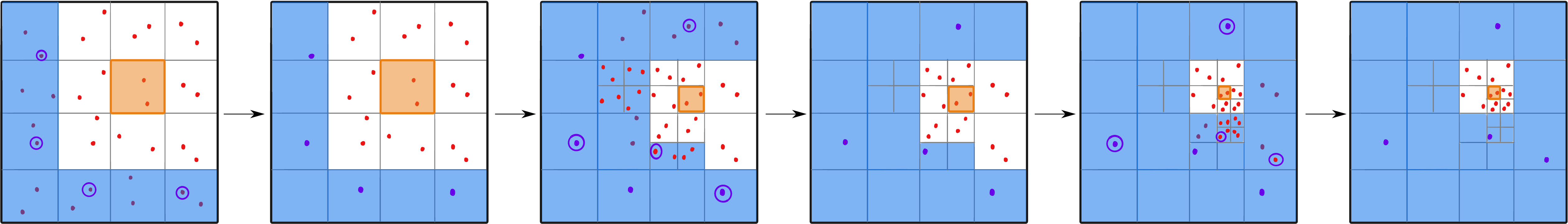} 
    \caption{Fast hierarchical data reduction: top-down pass. For each $X_i^*$ in an orange box, the blue region contains the entire farfield, where each representor set $Y_i^*$ (circled) contains at most 4 points.}
    \label{fig:HiDRdown} 
\end{figure}
%It can be seen from Figure \ref{fig:HiDRdown} that the number of points in the farfield (blue) data reduction for any node (orange) is always $O(1)$. 

The HiDR algorithm is summarized in Algorithm \ref{alg:data}. In practice, the cost of HiDR is  lower than the hierarchical low-rank compression.
%Hence we can use it to sparsify the dataset, which in turn will accelerate the subsequent hierarchical matrix construction.
HiDR reduces the computational cost of the subsequent hierarchical matrix construction.
\begin{algorithm}[h]
    \caption{\it Hierarchical data reduction (HiDR)}
    \label{alg:data}
    \emph{Input:} The adaptive partition tree $\mathcal{T}$ for dataset $X$, the collection of subsets $X_i$ for all leaf nodes $i$, prescribed maximum size $r_1$ of $X_i^*$ and maximum size $r_2$ of $Y_i^*$ \\
    \emph{Output:} Reduced representations $X_i^*$ and $Y_i^*$ for all nodes $i$
        \begin{algorithmic}[1]
        \FORALL{$i\in\mathcal{T}$}
            \STATE $Y_i^* = \emptyset$, $S_i = \emptyset$
            \IF{$i$ is a leaf node}
                \STATE $S_i = X_i$
            \ENDIF
        \ENDFOR
        \FOR{each level (from bottom to top)}
            \FORALL{$i$ at this level}
                \IF{$i$ is a parent}
                    \STATE $S_i=\bigcup\limits_{c\in \text{ch}(i)} X_c^*$ with ch($i$) the set of children of node $i$
                \ENDIF
                \STATE $X_i^*$ = {\tt DataReduct}($S_i$,$r_1$)
            \ENDFOR
        \ENDFOR
        \FOR{each level (from top to bottom)}
            \FORALL{$i$ at this level}
                \IF{$i$ has non-empty farfield}
                    \STATE $T_i=Y_p^* \bigcup X_j^*$ over all $j$ in the interaction list of $i$ ($p$ denotes the parent of $i$)
                    \STATE $Y_i^*$={\tt DataReduct}($T_i$,$r_2$)
                \ENDIF
            \ENDFOR
        \ENDFOR
        
        \RETURN $X_i^*$, $Y_i^*$ for all nodes $i\in\mathcal{T}$
    \end{algorithmic}
\end{algorithm}

\subsection{Complexity analysis} 
\label{sub:Complexity}
In this section, we show that Algorithm \ref{alg:data} (HiDR) has linear complexity with respect to the number of points in $X$.

\begin{theorem}
\label{thm:HiDRcost}
    For the given dataset $X$ that contains $n$ points, 
    let $\mathcal{T}$ be a partition tree for $X$, in which each leaf node corresponds to a subset of $X$ with $O(1)$ points and the sparsity constant $c_{\text{sp}}=O(1)$.
    Then the complexity of Algorithm \ref{alg:data} is $O(n)$.
\end{theorem}

\begin{proof}
    We first analyze the complexity in the bottom-up pass.
    Since for each leaf node $i$, $X_i$ contains $O(1)$ points, the {\tt DataReduct} for $X_i$ has $(1)$ cost.
    We next prove by induction that the cost for each nonleaf node is also $O(1)$.
    For a nonleaf node $i$, assume that for each child $c$, 
    $X_c^*$ contains $O(1)$ points.
    Then the intermediate set $S_i$ in Line 10 of Algorithm \ref{alg:data} contains $O(1)$ points because $i$ has $O(1)$ children and $X_c^*$ contains $O(1)$ points for each child $c$.
    Consequently, the cost of {\tt DataReduct} applied to $S_i$ is $O(1)$.
    This shows that the cost to obtain $X_i^*$ is also $O(1)$.
%    The updated $X_p^*$ after {\tt DataReduct} will only contain a small number of points determined by {\tt DataReduct} (independent of the level that node $p$ lies at), similar to its child nodes.
%    Hence the induction assumption holds for the update in the next level, for which $p$ serves as a child node.
%    Since the base case of the induction assumption is already true, i.e. all leaf nodes correspond to subsets with $O(1)$ points, we conclude that the cost for \emph{each} node is $O(1)$ and each reduced subset $X_i^*$ has $O(1)$ points.
    From the induction, we conclude that the cost for \emph{each} node is $O(1)$ and each reduced subset $X_i^*$ has $O(1)$ points.

    Now we analyze the complexity for the top-down pass.
    First note that each $Y_i^*$ has at most $O(1)$ points according to the construction.
    Consequently, each $T_i$ in Line 18 of Algorithm \ref{alg:data} contains at most $O(1)$ points because the interaction list of $i$ contains at most $c_{\text{sp}}=O(1)$ nodes and each $X_j^*$ has $O(1)$ points.
    This implies that the {\tt DataReduct} of $T_i$ to generate $Y_i^*$ has $O(1)$ complexity.
%    After {\tt DataReduct}, the updated $Y_i^*$ contains a small number of points (independent of the level of node $i$).
% Inductively, similar to the bottom-up pass, it can be seen that the update step at Line 21 always has $O(1)$ cost since $i$ has $O(1)$ nodes and $Y_i^*$ has $O(1)$ points.
%    Overall, before {\tt DataReduct} is applied, $Y_i^*$ has $O(1)$ points contributed from $X_j^*$ with $j$ in the interaction list and from the parent $Y_p^*$.
%    Hence the {\tt DataReduct} for each subset $Y_i^*$ has $O(1)$ complexity.

    Overall, we see that for each node $i$, the associated total cost to compute $X_i^*$ and $Y_i^*$ is $O(1)$.
Since there are $O(n)$ nodes in the tree, the total cost for all nodes is $O(n)$.
This completes the proof of the theorem.
\end{proof}

% subsection Complexity an (end)

%There are two challenges:
%(1). how to sample points such that the associated approximation accuracy will not be marred?
%(2). how to achieve $O(n)$ complexity?
%For the first challenge,
%we rely on a new sampling method proposed in \cite{anchornet}.
%The second challenge has been the main computational concern for algebraic approaches in constructing hierarchical matrices.
%This is due to the fact that, in order to maintain the nested structure of basis matrices,
%all farfield interactions have to be considered,
%which will lead to at least $O(n^2)$ cost using a naive compression.
%To resolve this difficulty, 
%we propose a novel two-step sampling of the dataset: first a bottom-up sampling, then a top-down sampling.
%The bottom-up step ``refines" the dataset so that each box (generated in adaptive partition) contains only $O(1)$ points.
%The top-down step samples for each box the points that account for farfield interactions.
%The total cost is $O(n)$ since the cost associated with each box is $O(1)$.

\subsection{Data reduction methods}
\label{sub:DR}

In this section, we provide several algorithms for performing {data reduction}.
%One is free to choose a method to use for the data reduction subroutine in hierarchical matrix construction.
The goal is to select a subset of $X_i$ and of $Y_i$ 
%in Algorithm \ref{alg:data} 
such that the selected subsets preserve the geometry of $X_i$ and $Y_i$.
{For an input set $X$, the subroutine takes the simple form:
{\tt DataReduct}($X$,$k$) $\to$ $X^*$ with $X^*\subset X$ the selected subset whose size is controlled by the parameter $k=O(1)$. Note that {\tt DataReduct} only depends on the input data and is independent of any kernel function.}
It has been shown in \cite{anchornet} that the choice of the subset is essential for the accuracy and robustness of the low rank approximation.
According to the results in \cite{anchornet,ddblock}, 
a subset evenly distributed over the containing set can offer an improved approximation robustness and accuracy over one that is not.
The methods below can be used to generate such a subset efficiently. 
Many of them rely on a reference set with good uniformity, such as a uniform tensor grid.
In addition to a tensor grid, it was shown recently that deep neural networks can be used to generate distributions with good uniformity \cite{deepMC}.

We briefly review some of the existing data reduction methods below. 
An empirical comparison of these methods for the low-rank approximation of kernel matrices is presented in Section \ref{sub:DRtest}.

%\cdf{(Add scatter plots to show selected pts (which dataset?) from different methods: 1. original data; 2. FPS; 3. tensor (used in Bebendorf, Venn 2012 \cite{bebendorf2012}); 4. surface; 5. anchor net?)}

\paragraph{Farthest point sampling.}
Given $X$ and a target size $k$ for the reduced subset $S$ of $X$, farthest point sampling (FPS) constructs $S$ in a sequential manner.
$S$ is initialized with one point only.
Then FPS searches for a point in $X\backslash S$ that is farthest from $S$ and adds the point to $S$.
This procedure is repeated until $S$ reaches size $k$.
FPS generates evenly distributed subsets and has been widely used in computational geometry \cite{FPS97,FPS06,FPS11}. FPS was recently proposed for computing low rank approximations \cite{ddblock}.

\paragraph{Volume-based data reduction.}
Volume based data reductions choose a subset $S$ of $X$ via a reference grid $Q$ with $O(1)$ points inside the computational domain.
%We first construct a grid inside a rectangular domain that encloses $X$.
%Such a reference grid $Q$ must be easy to construct.
For example, $Q$ can be chosen as a tensor grid (cf. \cite{bebendorf2012}) inside the rectangular domain that encloses the data. 
$S$ is chosen to be the collection of points in $X$ that are closest to each point in $Q$.
%For example, inside the smallest rectangular domain $\Omega$ that encloses $X\subset \mathbb{R}^3$,
%we construct ellipsoids centered at the center of $\Omega$ with principal semi-axes equal to $\gamma$ times the respective dimension of the rectangular domain $\Omega$, where $\gamma>0$ is a parameter.
%In practice, we can choose one or several ellipsoids corresponding to different values of $\gamma$.
%For example, we can choose three ellipsoids with $\gamma=0.2,0.5,0.7$.
%All those options for volume and surface-based methods can be computed in at most $O(n)$ complexity for $X$.
%Once a reference grid $Q$ with $O(1)$ points is constructed,
%we choose $S$ as the points in $X$ that are closest to each point in the reference grid.
%Namely, for each point $q$ in the reference grid $Q$,
%we choose a point $s\in X$ such that $|q-s|=\dist{q,X}$.

\paragraph{Surface-based data reduction.}
Following the same idea as the volume-based method, we can also use a reference set $Q$ based on surfaces constructed from the given data $X$.
%Surface-based data reduction but with a difference choice of the reference set $Q$.
In the surface-based method, we define $Q$ as the union of points distributed on surfaces near the boundary of the computational domain.
For example, we can construct ellipsoids centered at the center of a rectangular domain that encloses $X$.
The principal semi-axes of the ellipsoids are chosen to be equal to $\gamma$ times the width of the rectangular domain in each dimension, where $\gamma>0$ is a hyperparameter. 
In Section \ref{sub:DRtest}, we use three ellipsoids with $\gamma=0.3,0.6,1.2$.

%In practice, we can choose one or several ellipsoids corresponding to different values of $\gamma$.
%For example, we can choose three ellipsoids with $\gamma=0.2,0.5,0.7$.

%Compared to volume-based data reduction, the surface-based approach is less sensitive to the increase of dimension of the data points.

\paragraph{Anchor net method.}
The anchor net method \cite{anchornet} is a newly proposed subset selection method based on approximating the geometry of the given dataset with low discrepancy subsets.
For low-rank approximation, the anchor net method is shown to achieve a good time-accuracy trade-off in practice and is particularly efficient for high-dimensional data (cf. \cite{anchornet,ddblock}).

\section{Data-driven Hierarchical Matrix Construction} 
\label{sec:ddH}
%given tolerance, auto-select rank, parameter automation
In this section, we first show how to extract a low-rank factorization instantly for an admissible block after the prepossessing procedure HiDR in Section \ref{sub:block} and then present an algorithm with $O(n)$ complexity (Algorithm \ref{alg:H2}) for constructing an $\hmat^2$ matrix representation in Section \ref{sub:H2}. We analyze the computational complexity of Algorithm \ref{alg:H2} in Section \ref{sub:H2Complexity}. The proposed method enjoys the following features:
\begin{enumerate}[a)]
    \item black-box general-purpose (kernel independent) construction of the hierarchical low rank format;
    \item optimal $O(n)$ complexity, where $n$ is the number of points in the given dataset;
%    \item valid for high dimensional problems (more than 3 dimensions);
    \item better efficiency for data from complex geometry compared to general-purpose approaches as well as specialized kernel-dependent techniques.
%than interpolation-based methods.
%lower rank than interpolation-based method for achieving the same accuracy.
\end{enumerate}

\begin{algorithm}[h]
    \caption{\it Data-driven (DD) HiDR-based $\hmat^2$ matrix construction}
    \label{alg:H2}
    \emph{Input:} Dataset $X$, kernel function $\kappa(x,y)$, approximation tolerance $\epsilon$, separation ratio $\tau$, maximum number of points $q$ for a leaf node\\
%The adaptive partition tree $\mathcal{T}$ for dataset $X$, the collection of subsets $X_i$ for all leaf nodes $i$\\
    \emph{Output:} $\hmat^2$ matrix representation
        \begin{algorithmic}[1]
        \STATE Apply adaptive partitioning to $X$ to generate the partition tree $\mathcal{T}$ with at most $q$ points for each leaf node and obtain subsets $X_i$ for all \emph{leaf} nodes $i$
        \STATE Determine approximation parameters $r_1, r_2$ from $\epsilon$
        \STATE Apply HiDR in Algorithm \ref{alg:data} with approximation parameters $r_1,r_2$ to obtain $O(1)$ representor sets $X_i^*$ and $Y_i^*$ for all nodes $i\in\mathcal{T}$
%        \STATE Apply SMASH $\hmat^2$ construction (bottom-up traversal): 
        \FORALL{$i\in\mathcal{T}$}
            \STATE Define $\bar{X}_i^{\text{(row)}}=\bar{X}_i^{\text{(col)}}=\emptyset$
            \IF{$i$ is a leaf node}
                \STATE define $\bar{X}_i^{\text{(row)}}=\bar{X}_i^{\text{(col)}}=X_i$ 
            \ENDIF
        \ENDFOR
        \FORALL{non-root $i\in\mathcal{T}$ from bottom level to top level} 
            \STATE Apply getBasis($K_{\bar{X}_i^{\text{(row)}} Y_i^*}$) to obtain $U_i,\hat{X}_i^{\text{(row)}}$ and getBasis($K_{Y_i^* \bar{X}_i^{\text{(col)}}}^T$) to obtain $V_i,\hat{X}_i^{\text{(col)}}$
            \STATE Update $\bar{X}_p^{\text{(row)}} = \bar{X}_p^{\text{(row)}}\cup \hat{X}_i^{\text{(row)}}$ and $\bar{X}_p^{\text{(col)}} = \bar{X}_p^{\text{(col)}}\cup \hat{X}_i^{\text{(col)}}$, where $p$ is the parent of $i$
        \ENDFOR
        \STATE Define the $\hmat^2$ column and row basis matrices:
    $U=\{U_i\}_{\text{leaf } i}$,
    $V=\{V_i\}_{\text{leaf } i}$,
transfer matrices $R=\{R_i\}, W=\{W_i\}$:
        \[
         \begin{aligned}
%             U_i &= &\text{if $i$ is a leaf node},
             \begin{bmatrix}
                 R_{c_1}\\
                \vdots\\
                R_{c_k}
             \end{bmatrix} = U_i, \quad
             \begin{bmatrix}
                 W_{c_1}\\
                \vdots\\
                W_{c_k}
             \end{bmatrix} = V_i, \quad \text{if $i$ has children $c_1,\dots,c_k$},
         \end{aligned}
        \]
coupling matrices $B=\{B_{i,j}\}$:
\[
            B_{i,j} = 
\begin{cases}
    K_{\hat{X}_i^{\text{(row)}} \hat{X}_j^{\text{(col)}}}, &\text{if $(i,j)$ is admissible}\\
    K_{X_i X_j} & \text{otherwise}
\end{cases}
\]
%with basis matrices $\{U_i, V_i\}_{\text{leaf nodes } i}$, transfer matrices $\{R_i, W_i\}_{i}$ and coupling matrices $\{B_{i,j}\}_{\text{admissible} (i,j)}$ and dense nearfield blocks $\{K_{X_i X_j}\}_{\text{non-admissible leaf pair} (i,j)}$
%        \FOR{each leaf node $i\in\mathcal{T}$}
%            \STATE compute the initial basis matrices $\hat{U}_i = K_{X_i Y_i^*}$, $V_i = K_{Y_i^* X_i}^T$
%        \ENDFOR
%        \FOR{each pair of non-admissible leaf nodes $i,j$}
%            \STATE compute the nearfield block $K_{X_i X_j}$
%        \ENDFOR
        \RETURN $\hmat^2$ representation: $U,V,R,W,B$
    \end{algorithmic}
\end{algorithm}

\subsection{Approximating the entire farfield $K_{X_i Y_i}$}
\label{sub:block}
In this section, we show how to derive a low-rank approximation for the entire farfield $K_{X_i Y_i}$ based on the representor sets $X_i^*, Y_i^*$ returned by Algorithm \ref{alg:data}.

When computing an approximate column basis for $K_{X_i Y_i}$, $X_i$ contains $O(1)$ points and its entire farfield $Y_i$ contains $O(n)$ points for the leaf node $i$ (see Figure \ref{fig:YiYis}).
In order for the entire algorithm to have linear complexity, the column basis matrix of $K_{X_i Y_i}$ must be computed in $O(1)$ complexity. 

We first apply strong rank-revealing QR (SRRQR) factorization \cite{rrqr96} to the submatrix $K_{X_i Y_i^*}$:
\begin{equation}
\label{eq:SRRQR1}
    K_{X_i Y_i^*} = P \begin{bmatrix}
        I\\
        G
    \end{bmatrix}
    K_{\hat{X}_i Y_i^*},
\end{equation}
where $P$ is a permutation matrix, $||G||_{\max}$ is bounded by a prescribed constant, and $\hat{X}_i$ is a subset of $X_i$ with $O(1)$ points.
Then the column basis matrix is chosen as $U_i=P[I;G^T]^T$ and the low-rank approximation for the entire farfield reads
\begin{equation}
\label{eq:SRRQR2}
   K_{X_i Y_i}\approx P \begin{bmatrix}
        I\\
        G
    \end{bmatrix}
    K_{\hat{X}_i Y_i}.
\end{equation}

%In hierarchical matrices, the factorization \eqref{eq:SRRQR2} is widely used to construct nested bases (cf. \cite{smash,xing2020interpolative}).
For notational convenience, we denote the procedure in \eqref{eq:SRRQR1}-\eqref{eq:SRRQR2} for computing an approximate column basis for $K_{X_i Y_i}$ by 
\begin{equation}
\label{eq:QRoutput1} 
    \text{getBasis}(K_{X_i Y_i}) = (U_i, \hat{X}_i),
\end{equation}
where $U_i:= P[I;G^T]^T$ is the computed column basis and $\hat{X}_i\subset X_i$. 
Note that the kernel matrix $K_{X_i Y_i}$ is never formed
because the input of ``getBasis" is the kernel function and the subsets $X_i, Y_i^*$.

The cost to obtain $U_i$ and $\hat{X}_i$ from \eqref{eq:QRoutput1} is $O(1)$ as the matrix $K_{X_i Y_i^*}$ is $O(1)$-by-$O(1)$. 
Also notice that 
\begin{equation}
\label{eq:sizeXhat}
    \cardn(\hat{X}_i)\leq \rank{K_{X_i Y_i^*}}\leq \cardn(Y_i^*) = O(1).
\end{equation}
Thus, $\hat{X}_i$ always contains $O(1)$ points.

As we shall see in Section \ref{sec:numerical}, the column basis $U_i$ derived from $K_{X_i Y_i^*}$ can yield better accuracy than analytic methods such as interpolation.

\subsection{Computing hierarchical matrices with nested bases using HiDR} 
\label{sub:H2}
Hierarchical matrices with nested bases, e.g. $\hmat^2$ matrices, can offer optimal $O(n)$ complexity in time and space when approximating an $n$-by-$n$ kernel matrix.
A black-box hierarchical matrix construction proposed in \cite{smash} works for general kernel functions and allows for arbitrary low-rank compression techniques.
In this section, we show that the hierarchical data reduction (HiDR) can be incorporated naturally into the construction of $\hmat^2$ matrix representations via the general framework proposed in SMASH \cite{smash}.
%SMASH \cite{smash} employs a bottom-up procedure to construct the hierarchical matrix representation with nested bases.
SMASH employs a bottom-up procedure that recursively applies rank-revealing factorization to the initial basis matrix (with $O(1)$ entries) for each node in the tree. 
In \cite{smash}, the initial basis matrices are constructed via either interpolation or analytic expansion of the kernel function.
In this section, we leverage representor sets produced by HiDR in Algorithm \ref{alg:data} to construct the initial basis matrices.
%The framework allows for any method to be used to construct the initial basis matrices. 
%As long as the complexity of computing the initial basis matrix is $O(1)$ for each node, the whole SMASH procedure for constructing an $\hmat^2$ representation has $O(n)$ complexity for a dataset with $n$ points.
%With HiDR, we follow the SMASH framework by using the submatrix $K_{X_i Y_i^*}$ as the initial basis matrix for each leaf node $i$ to build the $\hmat^2$ representation.

The full data-driven construction is presented in Algorithm \ref{alg:H2}.
The algorithm automatically determines the approximation parameters $r_1,r_2$ for HiDR in Algorithm \ref{alg:data} according to the approximation tolerance $\epsilon$ prescribed by the user.
The idea here is to apply the low-rank approximation in \eqref{eq:SRRQR2} to the artificial kernel matrix $K_{Z_1 Z_2}$ where $Z_1,Z_2\subset \mathbb{R}^d$ are well-separated subsets (in the sense of \eqref{eq:separation}) of $O(1)$ random points.
The parameters $r_1,r_2$ are chosen adaptively by increasing from $r_1=r_2=1$ to a point such that the approximation error to $K_{Z_1 Z_2}$ is smaller than $10^{-2}\epsilon$.
%Here the constant $10^{-2}$ is to avoid an aggressive estimate of the approximation error for the later admissible block in hierarchical matrix construction.
More sophisticated techniques like a posteriori error estimation (cf. \cite{localL2,cai2020equi,cai2020JSC,confusion}) can also be studied to estimate the approximation error.
After the parameters are determined, Algorithm \ref{alg:data} first applies hierarchical data reduction to $X$ associated with tree $\mathcal{T}$ to obtain representor sets. 
Then the hierarchical matrix representation can be computed rapidly by following the SMASH $\hmat^2$ construction and using $K_{X_i Y_i^*}$ as the initial basis matrix for each leaf node $i$.

We perform numerical experiments in Section \ref{sec:numerical} to demonstrate that the new data-driven method improves the matrix approximation accuracy of interpolation-based SMASH algorithm \cite{smash}.
Moreover, the cost of hierarchical data reduction is smaller than the subsequent hierarchical matrix compression.

\subsection{Complexity analysis} 
\label{sub:H2Complexity}
%We show in the following that, to approximate an $n$-by-$n$ kernel matrix, Algorithm \ref{alg:H2} has complexity $O(n)$.

\begin{theorem}
\label{thm:H2Cost}
    For the given dataset $X$ with $n$ points,
let $\mathcal{T}$ be a partition tree for $X$, in which each leaf node corresponds to a subset of $X$ with $O(1)$ points and the sparsity constant $c_{\text{sp}}=O(1)$. 
%Let $c_{sp}$ denote the sparsity constant.
Then the complexity of Algorithm \ref{alg:H2} is $O(n)$.
\end{theorem}

\begin{proof}
Algorithm \ref{alg:H2} follows the $\hmat^2$ construction in SMASH \cite{smash} with an additional hierarchical data reduction (HiDR) in Line 3.
According to Theorem \ref{thm:HiDRcost}, HiDR has $O(n)$ complexity.
%Hence the complexity analysis is same to the $\hmat^2$ construction in SMASH \cite{smash} plus the additional cost of the data reduction preprocessing, which has $O(n)$ complexity according to Theorem \ref{thm:HiDRcost}.
Thus to prove the $O(n)$ complexity of Algorithm \ref{alg:H2}, it suffices to show that the cost per node is $O(1)$ in Lines 11--12.
%The SMASH construction after Line 3 also has $O(n)$ complexity for the following reasons.

We first show that $\bar{X}_i^{\text{(row)}}$ and $\bar{X}_i^{\text{(col)}}$ contain at most $O(1)$ points.
If $i$ is a leaf node, then according to the definition in Line 7, 
$\bar{X}_i^{\text{(row)}} = \bar{X}_i^{\text{(col)}}= X_i$ contain $O(1)$ points. 
If $i$ is a parent node, then after all children of $i$ have been updated, $\bar{X}_i^{\text{(row)}}$ and $\bar{X}_i^{\text{(col)}}$ can be written as 
\[
    \bar{X}_i^{\text{(row)}}=\bigcup\limits_{c \text{ is a child of } i} \hat{X}_c^{\text{(row)}},\quad 
    \bar{X}_i^{\text{(col)}}=\bigcup\limits_{c \text{ is a child of } i} \hat{X}_c^{\text{(col)}}.
\]
Since the number of children for every node is bounded from above by a constant, and every subset $\hat{X}_c$ contains $O(1)$ points according to \eqref{eq:sizeXhat},
we see that $\bar{X}_i^{\text{(row)}}$ and $\bar{X}_i^{\text{(col)}}$ contain $O(1)$ points.
This implies that the complexity in Line 12 is at most $O(1)$ for all $i$.

Next we analyze the complexity in Line 11.
Since $Y_i^*$ contains $O(1)$ points only, it follows that in Line 11, the input matrices $K_{\bar{X}_i^{\text{(row)}} Y_i^*}$ and $K_{Y_i^*\bar{X}_i^{\text{(col)}}}^T$ have $O(1)$ rows and columns.
Consequently, performing ``getBasis" in Line 11 only takes $O(1)$ time.

Now we conclude that the total complexity in Lines 11--12 is $O(1)$.
Therefore, the total complexity for Algorithm \ref{alg:H2} is $O(n)$.
\end{proof}

\section{Data Reduction and Low Rank Approximation} 
\label{sec:dr1}
{
In this section, we investigate different data reduction techniques for low rank approximation.
%Numerical experiments are presented to illustrate the performance of different techniques.
Section \ref{sub:failure of virtual points} presents an example to reveal a drawback of methods that rely on points outside the given dataset (such as interpolation nodes or random points) for computing the low-rank approximation to kernel matrices.
In Section \ref{sub:DRtest}, we compare the performance of the data reduction methods of Section \ref{sub:DR} for low-rank approximation. 
%Section \ref{sub:failure of virtual points} illustrates the drawback of choosing points outside the given data when constructing low rank approximations.
%Section \ref{sub:DRtest} presents empirical comparison of the four data reduction methods in Section \ref{sub:DR}.
}

\subsection{Drawback of using points outside the given data for irregular datasets}
\label{sub:failure of virtual points}

%Given a kernel matrix that is numerically low-rank and that is associated with a
%kernel function $\kappa$ and sets of points $X$ and $Y$, the rectangular
%bounding boxes $\Omega_X$ and $\Omega_Y$, containing $X$ and $Y$
%may not be well-separated. In other words, the analytic and hybrid
%methods described in Section \ref{sub:review} are not effective
%at approximating the low-rank matrix, because the low-rank property
%of the matrix is a consequence of the configuration of the points,
%$X$ and $Y$. In high-dimensional data analysis problems in particular, 
%the low-rank property of a kernel matrix may arise from the data,
%i.e., the data lies on a manifold which could be called a ``pattern''
%to be discovered in the data.
%
%%We use a model problem to illustrate the potential inefficiency of methods that use extrinsic points, i.e., that ignore geometry.

Interpolation nodes, random points, or in general points outside the given dataset are commonly used in low-rank approximation to obtain approximate column basis efficiently without forming the original kernel matrix.
Since these points are created artificially (not part of the given data), we call these points \emph{virtual points}.
For given data $X$ and $Y$, virtual points are constructed in rectangular domains 
$\Omega_X$ and $\Omega_Y$ that cover $X$ and $Y$, respectively, and the kernel matrix associated with these virtual points is used to compute the low-rank approximation.

{One issue of of using virtual points is that it may lead to an incorrect approximation to the kernel matrix.
This is because the virtual points lie outside the original data, and the kernel matrix involving these virtual points may have a very different spectrum from that of the kernel matrix to be approximated.
Thus methods using virtual points may not be robust for approximating general low-rank kernel matrices.}
To illustrate the issue, we use the ``nJ'' dataset as illustrated in Figure \ref{fig:2Dset2Data}, where $X$ contains 120 points in $\Omega_X=[-2.5,2.5]\times [0,3.75]$ and $Y$ contains 150 points in $\Omega_Y=[0,5.13]\times [16.25,23.84]$.
%, which potentially resembles more complicated distributions in real applications.
We consider the smooth kernel function
\begin{equation*}
 \kappa(x,y) = \sqrt{1+100|x-y+a|^2}
 %,\quad  \kappa(x,y) = \exp(-|x-y+a|^2),
 %\quad \frac{1}{|x-y+a|}, \quad (x,y)\in X\times Y,
 \end{equation*}
with $a=[0,20]^T$.
{The same issue also arises for other kernels such as Gaussians.}
The corresponding kernel matrix $K_{XY}$ (120-by-150) has rapidly decaying singular values as shown in Figure \ref{fig:2Dset2svd} (dashed line).
Consequently, $K_{XY}$ can be approximated very well by a low-rank matrix.

Now consider $K_{X_1 Y_1}$, with $X_1$ being 2000 random points
selected in $\Omega_X$ and $Y_1$ being 2000 random points
selected in $\Omega_Y$.  The singular values of $K_{X_1 Y_1}$ are an
approximation to the continuous singular values of the problem in \eqref{eq:T}. The computed singular values of 
$K_{X_1 Y_1}$ are also plotted in Figure \ref{fig:2Dset2svd}.
{It can be seen that the singular values of $K_{X_1 Y_1}$ do not decay rapidly, compared to $K_{XY}$.}

Now consider $K_{X_2 Y_2}$, where $X_2$ and $Y_2$ are 
$10 \times 10$ Chebyshev points in $\Omega_X$ and $\Omega_Y$, respectively.
%Such points are used for interpolation in analytic mehtod \cite{hackintroh2app} and for hybrid cross approximation \cite{HCA2005}. 
Like $K_{X_1 Y_1}$, this matrix does not have singular values that decay as rapidly as those of $K_{XY}$, and therefore an algebraic compression of $K_{X_2 Y_2}$ will not be an effective approximation for $K_{XY}$.

\begin{figure}[htbp]
    \centering
    \subfigure[Dataset $X$, $Y$]{\label{fig:2Dset2Data}\includegraphics[width=25mm]{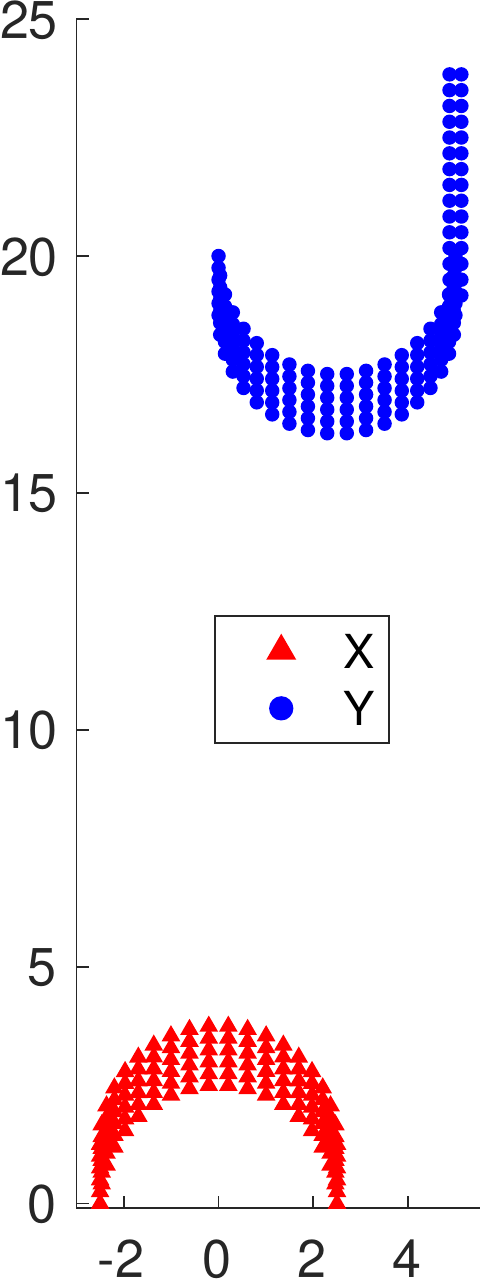}}
     \hspace{2mm}
    \subfigure[2000 random points in $\Omega_X$, $\Omega_Y$]{\label{fig:2Dset2PPT}\includegraphics[width=25mm]{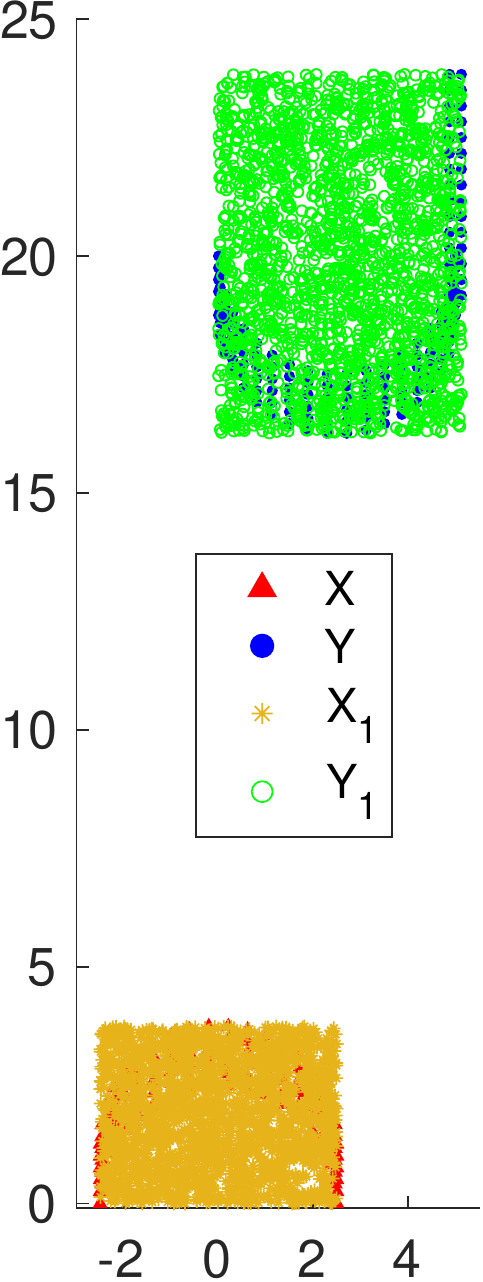}}
     \hspace{2mm}
    \subfigure[$10 \times 10$ Chebyshev points in $\Omega_X$, $\Omega_Y$]{\label{fig:2Dset2Cheb}\includegraphics[width=25mm]{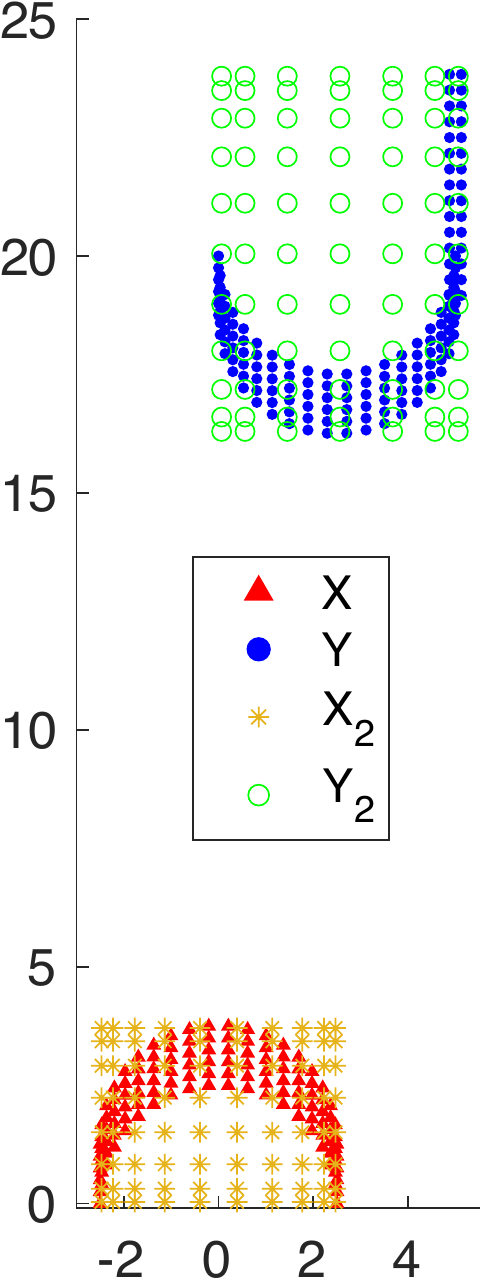}}
    \caption{Dataset $X,Y$ and different types of ``virtual points" within 
$\Omega_X$ and $\Omega_Y$.}
    \label{fig:2Dset2}
\end{figure}

\begin{figure}[htbp]
    \centering
    %\subfigure[Error-rank plot]{\label{fig:2Dset2errorMQ}\includegraphics[scale=0.3]{./figures/2DerrorMultiQ.pdf}}
    %\hspace{0.2in}
    % \subfigure[$\exp(-|x-y+a|^2)$]{\label{fig:2Dset2errorGS}\includegraphics[scale=0.35]{./figures/2DerrorGS.pdf}}
    %\subfigure[Singular values]{\label{fig:2Dset2svd}\includegraphics[width=34mm]{./figures/2DsvdMultiQ.pdf}}
    \includegraphics[width=50mm]{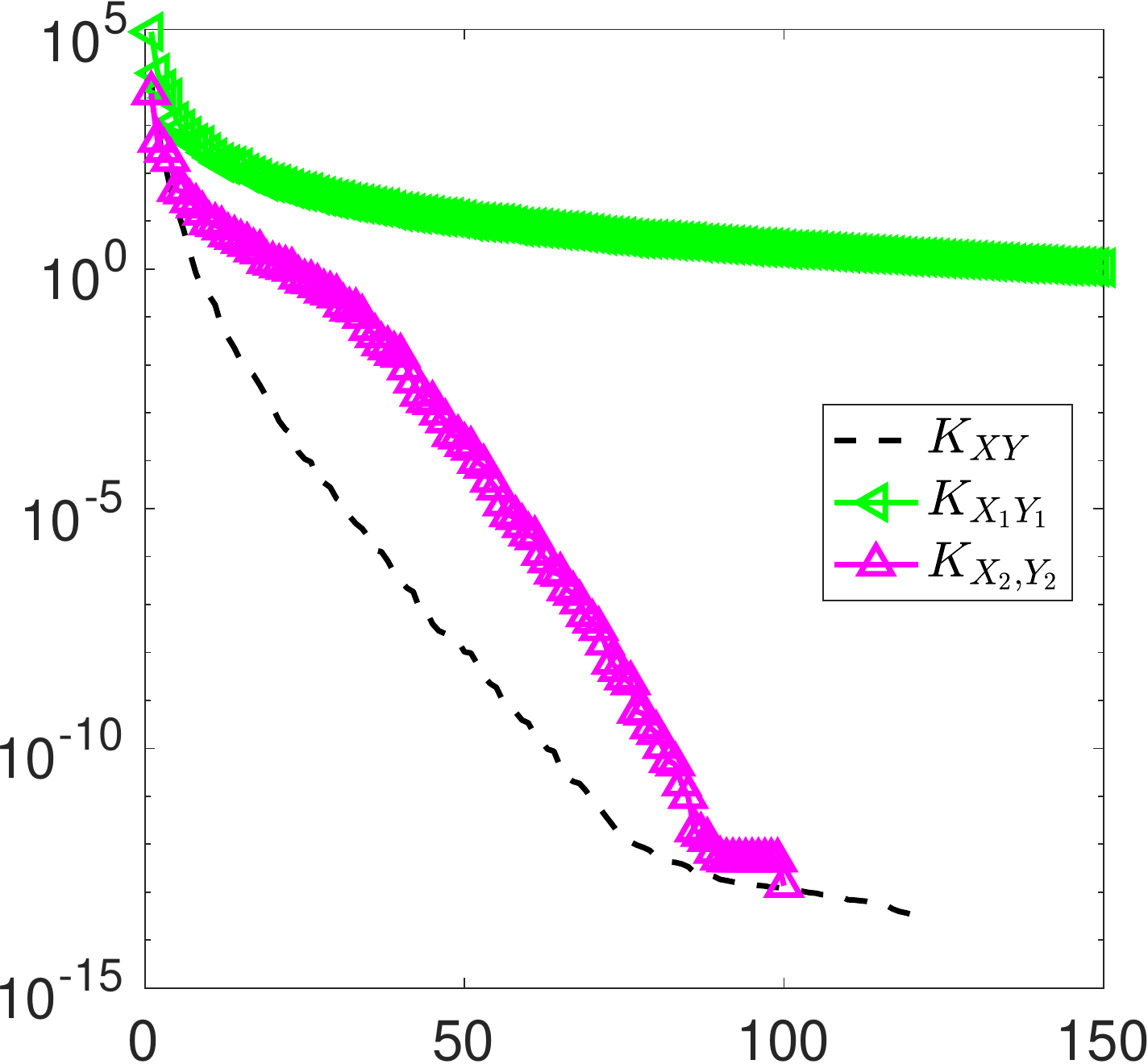}
    \caption{Distinct singular value patterns for kernel matrices with original data $X\times Y$ and virtual points $X_i\times Y_i$ ($i=1,2$) of Figure \ref{fig:2Dset2}: singular values of
$K_{XY}$ (120-by-150), 
$K_{X_1 Y_1}$ (2000-by-2000),
$K_{X_2 Y_2}$ (100-by-100).
For $K_{X_1 Y_1}$, the largest 150 singular values are plotted.}
    \label{fig:2Dset2svd}
% (2000-by-2000) for Figure \ref{fig:2Dset2PPT}, $K_{X_2 Y_2}$ (100-by-100) for Figure \ref{fig:2Dset2PPT100}, $K_{X_3 Y_3}$ (100-by-100) for Figure \ref{fig:2Dset2Cheb}.}
    % \label{fig:2Dset2error} 
\end{figure}

%The analytic and hybrid methods described in Section \ref{sub:review} approach the discrete matrix approximation problem using a continuous approximation and utilize sets of points in a pair of domains that are independent of the dataset $X,Y$. We call these points \emph{extrinsic points} (extrinsic to the data in general) and the corresponding kernel matrix formed from these points an \emph{extrinsic matrix}.  In these methods, the low-rank factorization of the original kernel matrix is then obtained from this extrinsic matrix or an algebraic compression of it.

%Another issue is that the use of extrinsic points and the associated extrinsic kernel matrix may not correctly reflect the spectrum of the original kernel matrix, which will lead to an inaccurate approximation. A numerical example is presented in Section \ref{sub:extrinsicEx} to illustrate this issue.

\paragraph{Mathematical explanation.}
Employing a continuous treatment of the matrix approximation problem ignores the geometry of the discrete dataset.
This can be problematic in general as the continuous problem may have entirely different spectral properties compared to the matrix.
It is even possible that the kernel function is \emph{undefined} at virtual points.
For the model problem in Figure \ref{fig:2Dset2}, 
the matrix $K_{X_1 Y_1}$ with virtual points $X_1, Y_1$ is related to the following integral operator:
\begin{equation}
\label{eq:T}
    T: L^2(\Omega_X):\to L^2(\Omega_Y),\quad (Tf)(x) := \int_{\Omega_X} \kappa(x,y) f(y) dy.
\end{equation}
The singular values of $K_{X_1 Y_1}$ (with $X_1$ and $Y_1$ chosen as 
described above) approximate the singular values (up to a scaling constant)
of the integral operator \cite{atkinson1967eig,keller1965eig,spence1978eig,osborn1975,atkinson1975,eigCMAM}.
These singular values do not decay rapidly like those of $K_{XY}$.
In essence, we see that virtual points methods treat the matrix approximation as a \emph{continuous} problem and thus ignore the geometry of the discrete data. 
When the continuous problem differs substantially from the original \emph{discrete} problem (kernel matrix approximation), the performance of {methods that utilize virtual points} can be very unsatisfactory.

% subsection Failure of virtual points (end)

%\subsection{Single block low-rank compression (Compare different Data reductions)}
%\label{sub:blocktest}
%%Compare to ACA, interpolation, hy-int, etc. for approximating $K_{X_i Y_i}$
%
%compare different data reductions and interpolation
%
% subsection singleblock (end)

%Use $\hmat$ for scaling-dependent kernels e.g. Gaussian, exp.
%$\hmat^2$/nested bases structure is not good for scaling-dependent kernels
%
%Use $\hmat^2$ for scaling-independent kernels

%General kernel: Gaussian, Coulomb, some RBF kernels, double layer potential kernels, gradient kernels/gradient of RBF

%Cost: Compare pre-computation cost and hmat construction cost

%\begin{itemize}
%%    \item Part 1. low rank compression: we illustrate the advantages of data driven methods over existing techniques, general accurate, applicable in high dimensions;
%%        Part 2. How to choose column basis in $O(1)$ cost? use hierarchical sampling combined with SMASH framework;
%%    \item state of the art development on general data driven methods
%%    \item Benefits of hierarchical matrix representations
%%    \item Optimal $O(n)$ nested bases, existing methods, column basis construction, how to achieve $O(n)$ complexity
%%    \item Do not comment on pros and cons, just introduce another very different framework
%    \item Numerical tests: show rank/storage comparison to spherical harmonics basis, interpolation
%    \item Numerical tests: show largest matrix size during intermediate compression (RRQR) ? Error vs largest intermediate matrix size ?
%\end{itemize}

\subsection{Comparison of data reduction methods} 
\label{sub:DRtest}
In this section, we perform experiments to compare the performance of the four data reduction methods in Section \ref{sub:DR}: farthest point sampling (`FPS'), volume-based reduction (`Volume'), surface-based reduction (`Surface'), anchor net method (`AnchorNet').
{These methods operate on the dataset and do not require any kernel function.}

\paragraph{Experiment setup}
We consider low-rank approximation to the kernel matrix $K_{XY}$, where $X$ (198 points) and $Y$ (1577 points) are well-separated subsets from a dinosaur manifold as shown in Figure \ref{fig:DinoXY}.
The diameter of $X$ is 58.21, and the distance between $X$ and $Y$ is 21.275.
We test three different kernel matrices $K_{XY}$ corresponding to the kernel functions below:
$$\frac{1}{|x-y|},\quad e^{-\frac{|x-y|^2}{900}},\quad |x-y|^{11}.$$
To obtain the low-rank approximation, we first perform data reduction for $Y$, and then build the factorization as described in Section \ref{sub:block} using \eqref{eq:SRRQR1} and \eqref{eq:SRRQR2}.
The low-rank approximation error is measured by the relative matrix approximation error in the 2-norm.

%We increase the size of reduced subset, namely, the number of points selected by each data reduction method, and see how fast the low-rank approximation error decays for different methods.
The error plots for three different kernels are shown in Figure \ref{fig:DRerrors}.
For each plot, the horizontal axis denotes the number of points selected by the data reduction method, namely the size of the subset $Y^*\subset Y$.
Each curve shows how the low-rank approximation error for a specific data reduction method decays as we increase the size of $Y^*$.
We see that `Volume' and `AnchorNet' offer the best performance and are almost indistinguishable {from each other in performance} across all three kernels.
`Surface' achieves similar performance for the first two kernels but is slightly worse than `Volume' and `AnchorNet' for the third kernel.
The farthest point sampling `FPS' performs well but is not as accurate as the other three methods for the same number of selected points for all kernels tested.
%Overall, all these data reduction methods provide excellent performances for the low-rank approximation without requiring any information from the kernel function or accessing the kernel matrix.
%From the experiments, we see that volume-based data reduction provides an excellent option for low dimensional data.

\begin{figure}[htbp] 
    \centering 
    \includegraphics[scale=.33]{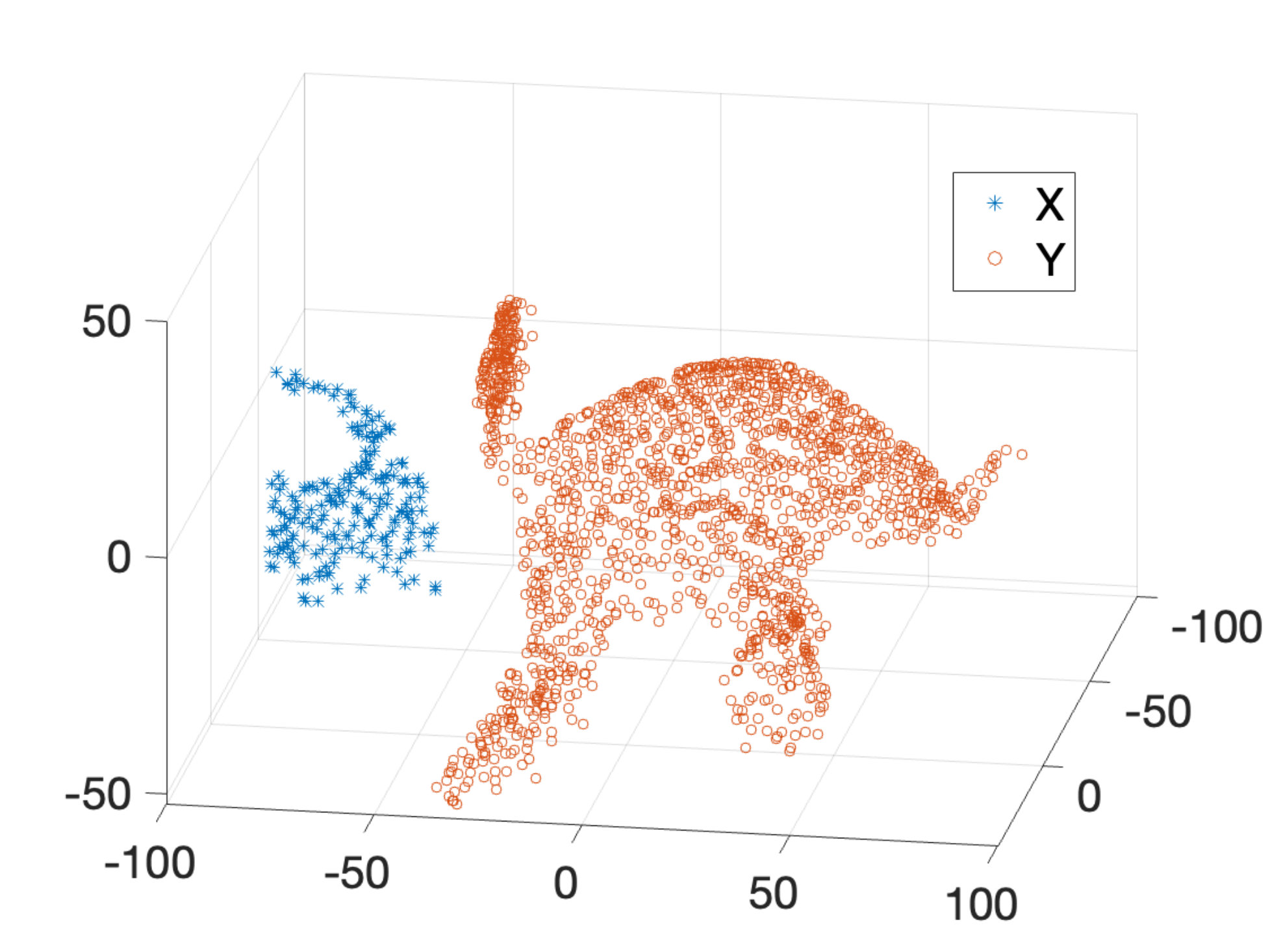} 
    \caption{Section \ref{sub:DRtest} dataset: $X$ and $Y$ well-separated}
    \label{fig:DinoXY}
\end{figure}

\begin{figure}[htbp] 
    \centering 
    \includegraphics[scale=.27]{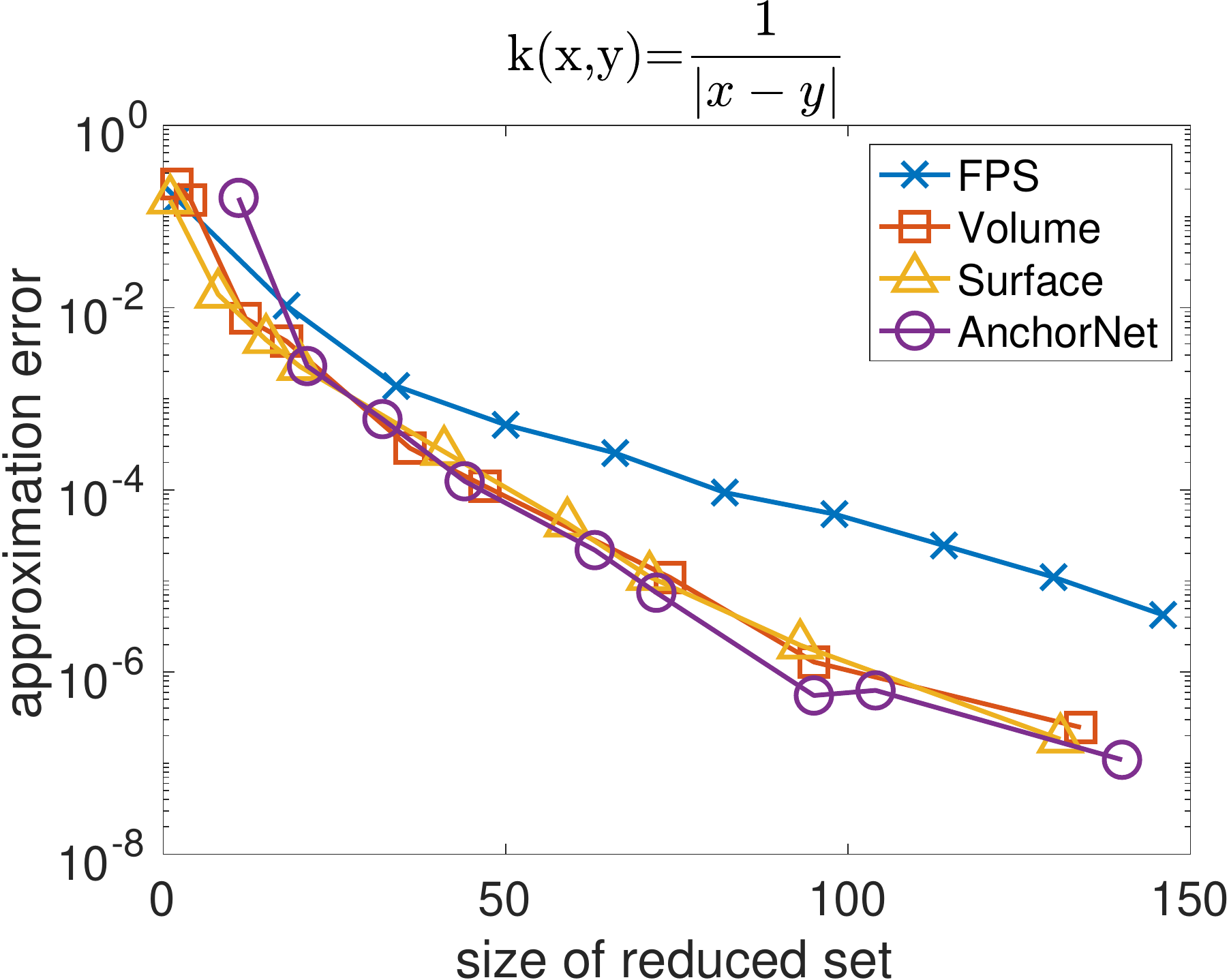}
    \includegraphics[scale=.27]{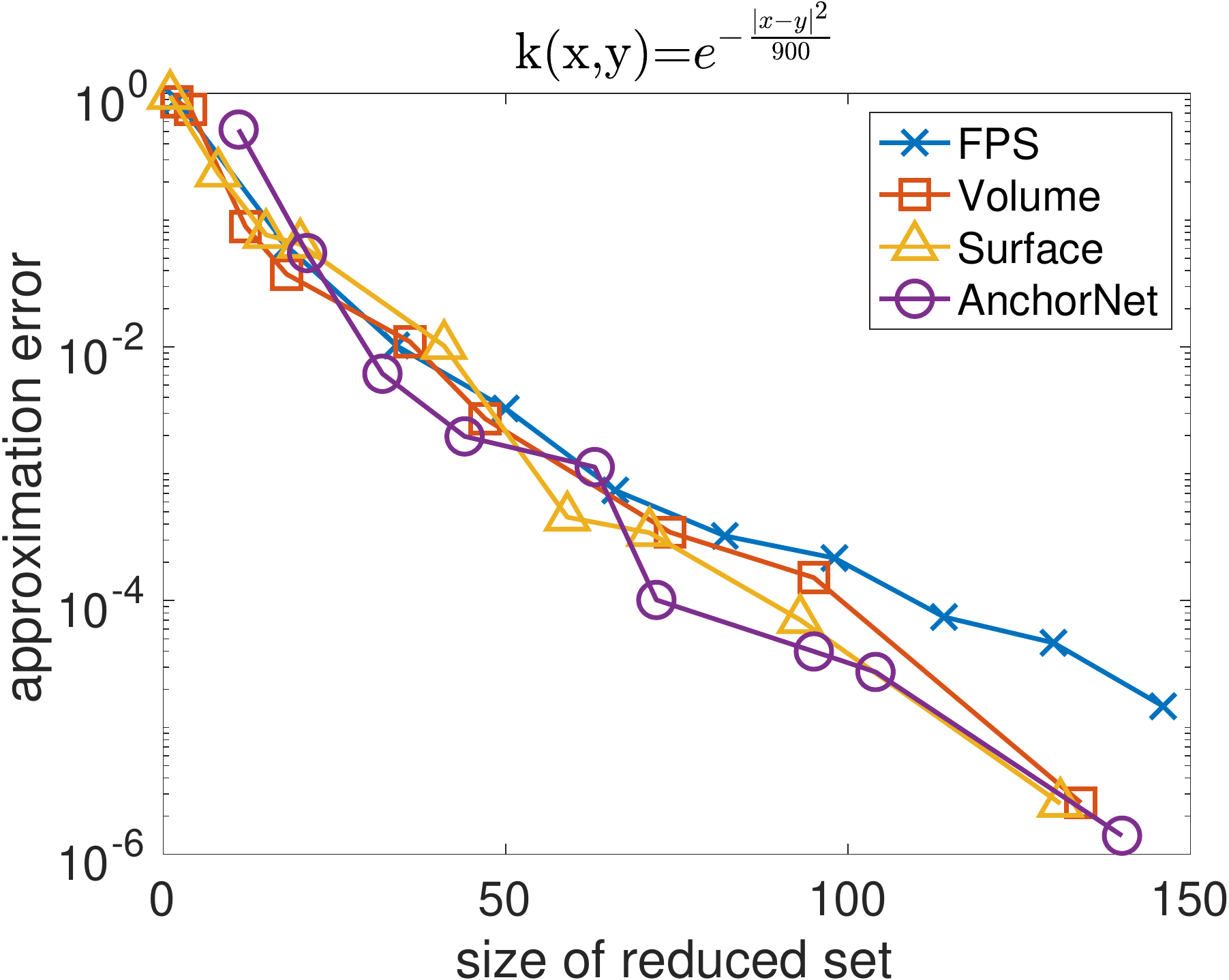}
    \includegraphics[scale=.27]{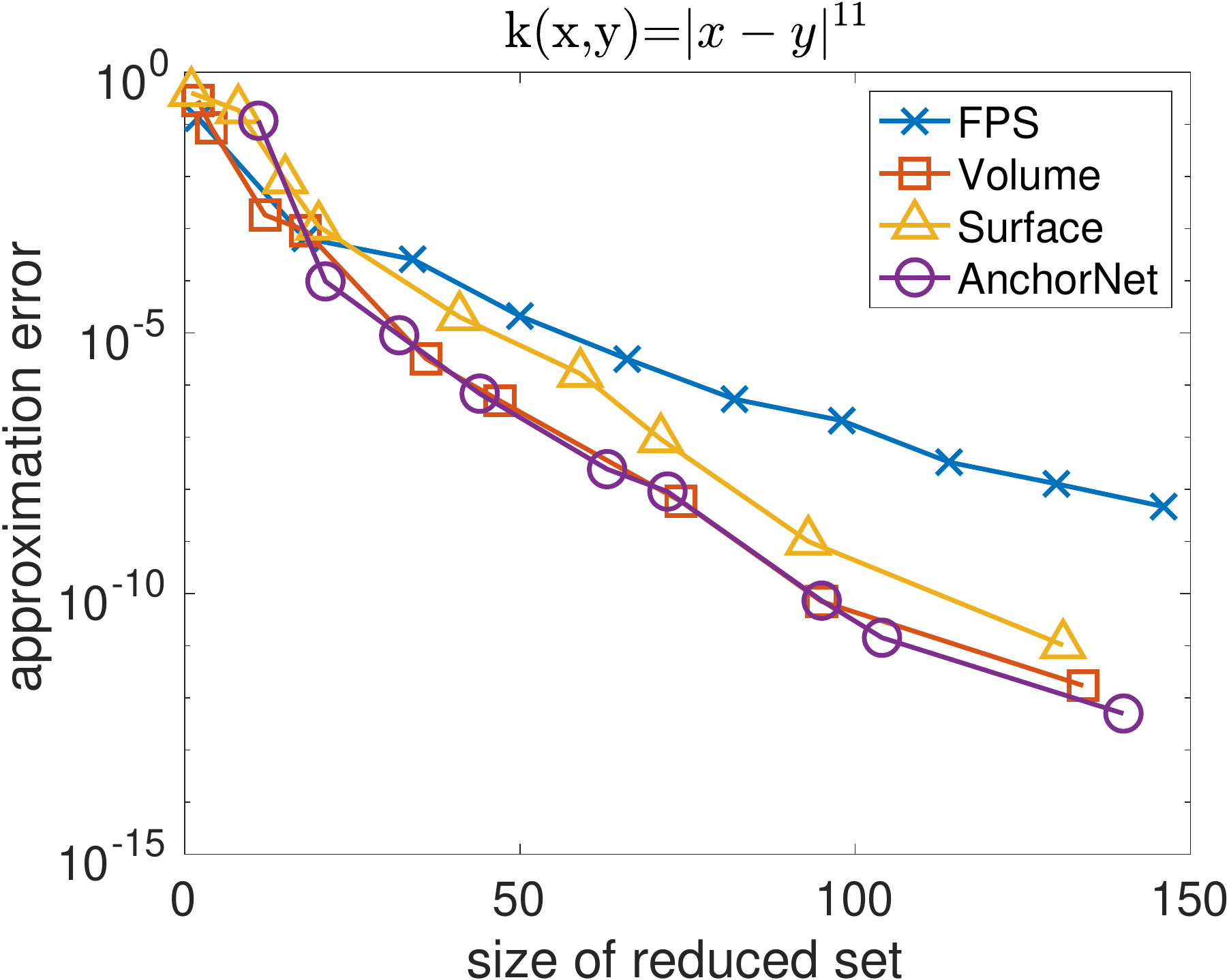}
    \caption{Comparison of data reduction methods for low-rank approximation to $K_{XY}$ with dataset $X\times Y$ in Figure \ref{fig:DinoXY} and different kernels $k(x,y)$ on top of each plot}
    \label{fig:DRerrors}
\end{figure}

% subsection Data reduction schmes (end)

%%%%%%%%%%%% section  (end) %%%%%%%%%%%%%%%%

\section{Numerical experiments} 
\label{sec:numerical}
We present a series of numerical experiments in this section to illustrate the performance of the proposed data-driven construction in Algorithm \ref{alg:H2}.
The code for the algorithm is available on GitHub\footnote{\url{https://github.com/scalable-matrix/H2Pack/tree/sample-pt-algo}}.
The performance of the proposed data-driven hierarchical matrix construction is shown in Section \ref{sub:Scaling}, inlcuding linear scaling, generality for various kinds of kernels, and the efficiency of hierarchical data reduction for varying kernel parameters.
Comparison to the state-of-the-art special-purpose methods for the Coulomb kernel is presented in Section \ref{sub:Comparison special}.
Comparison to the widely used general-purpose method (interpolation) for various kernels is presented in Section \ref{sub:Comparison general}.
%A scaling test for different types of kernel functions is presented in Section \ref{ssub:kernels}.
For the data-driven hierarchical construction, volume-based data reduction is used.
For experiments in Section \ref{sub:Scaling} and Section \ref{sub:Comparison special}, we use one compute node on the Georgia Tech PACE-Hive cluster. This node has two sockets and 192 GB DDR4 memory. Each socket has an Intel Xeon Gold 6226 12-core processor. 

The approximation error is measured by the relative matrix-vector product error: 
$\frac{||Kz-\tilde{K}z||}{||z||}$, where $\tilde{K}$ denotes the hierarchical approximation to the kernel matrix $K$ and $z$ is a standard normal random vector.
$||\cdot||$ denotes the 2-norm.
%In every experiment, the same vector is used for all cases.

%\cdf{add other kernels to show that ACA will fail for certain kernels? or ignore ACA?}

\subsection{Data-driven construction: scaling, generality, once-for-all HiDR}
\label{sub:Scaling}
{
This section has three objectives: (1) test the complexity of the proposed data-driven approach in Section \ref{ssub:Scaling test}, including the hierarchical data reduction (HiDR) and the resulting hierarchical matrix construction;
(2) illustrate the generality of the data-driven approach by testing different kernels in Section \ref{ssub:kernels};
(3) apply HiDR once and use the representor sets to construct hierarchical matrices for various types of kernels, including Gaussian kernels with different bandwidths in Section \ref{ssub:HiDR once for all}.
}

\subsubsection{Scaling test for different datasets} 
\label{ssub:Scaling test}

\paragraph{Datasets} 
Three datasets are used: cube, 3-sphere, Dino.
The ``cube" dataset contains random samples from the uniform distribution in the unit cube $[0,1]^3$.
%The ``sphere" dataset contains random samples from the uniform distribution on the unit sphere.
The ``Dino" dataset is used in \cite{smash} and is illustrated in Figure \ref{fig:datasets}.
It consists of points distributed on a dinosaur-shaped surface in three dimensions.
The ``3-sphere" dataset  (see Figure \ref{fig:datasets}) consists of random points distributed on the surface of three intersecting unit spheres whose centers form an equilateral triangle with side length close to 1.
Roughly the same number of points is sampled from each sphere.
Let $n$ denote the number of points in the dataset.
For the first three synthetic datasets, we test for $n$ from $10^5$ to $1.6\times 10^7$.
For the Dino dataset, since the size of the original data is fixed, we sample $n$ points randomly and vary $n$ from $10^4$ to $1.5\times 10^5$.

\begin{figure}[htbp] 
    \centering 
    \includegraphics[scale=.35]{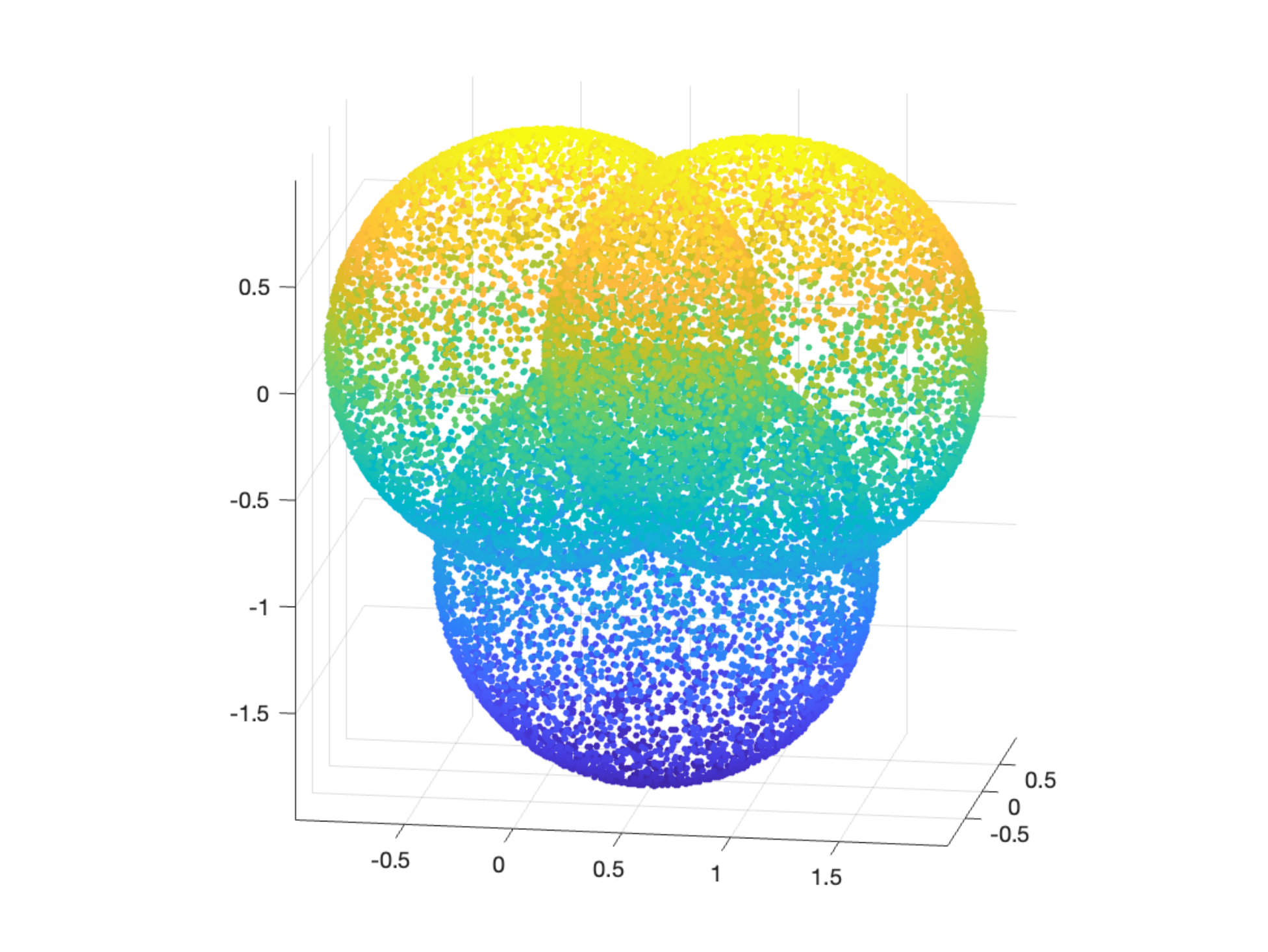} 
    \includegraphics[scale=.35]{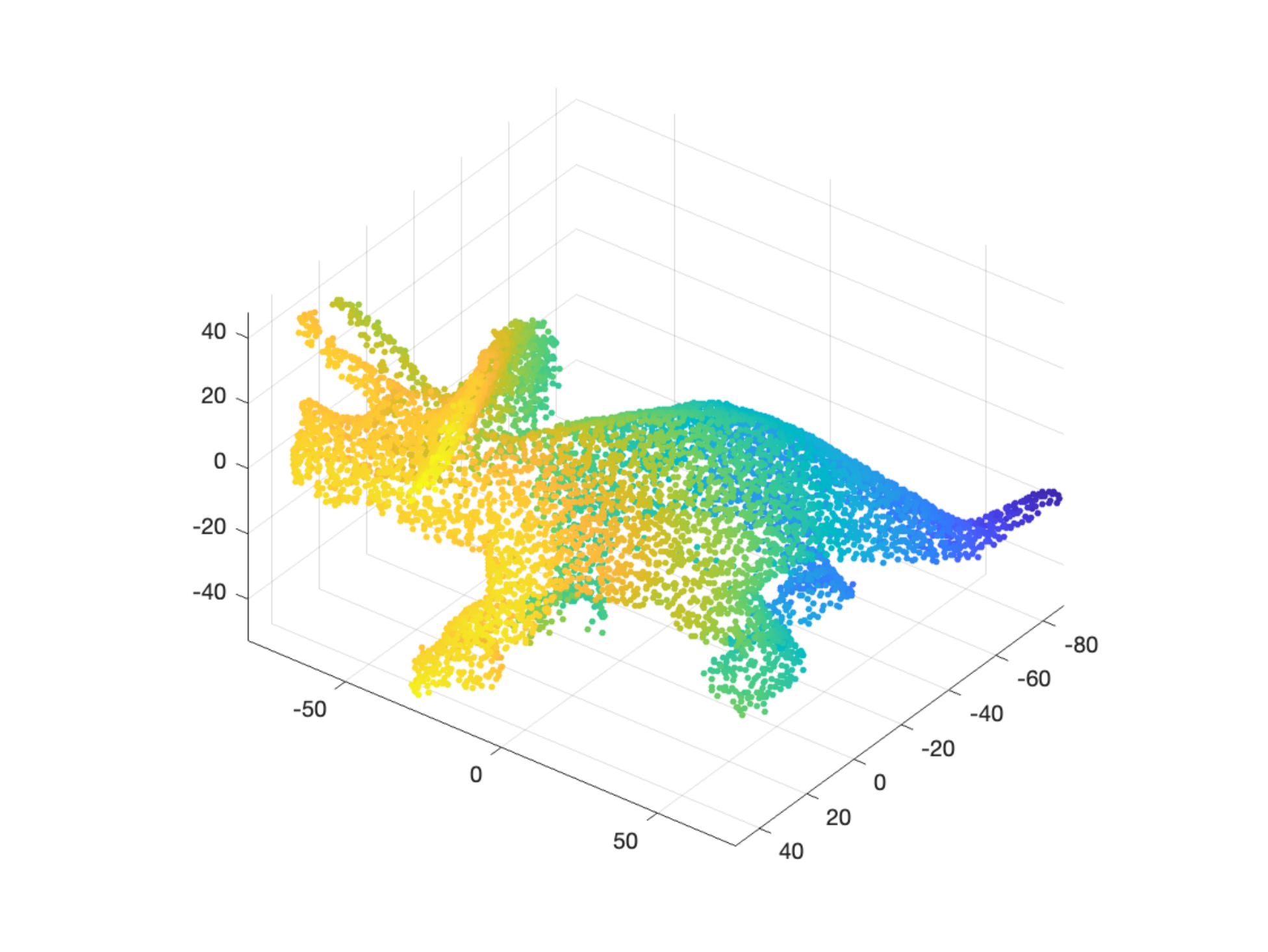} 
    \caption{3-sphere (left) and Dino (right) datasets} 
    \label{fig:datasets} 
\end{figure}

The kernel function is chosen to be the Coulomb kernel $\frac{1}{|x-y|}$ and the approximation error is $10^{-6}$ for each test.
Figure \ref{fig:Geo} shows the timings for hierarchical data reduction (`HiDR'), $\mathcal{H}^2$ matrix construction (`build'), the resulting matrix-vector multiplication (`matvec') with respect to $n$, respectively.
All timings scale linearly with $n$ and the hierarchical data reduction (`HiDR') has a much lower cost than the subsequent hierarchical matrix construction (`build').
The low cost of data reduction and the kernel independence make the data-driven approach suitable for the case when the kernel matrix changes frequently due to changes in the data or kernel function.
\begin{figure}[htbp]
    \centering
    \includegraphics[scale=.3]{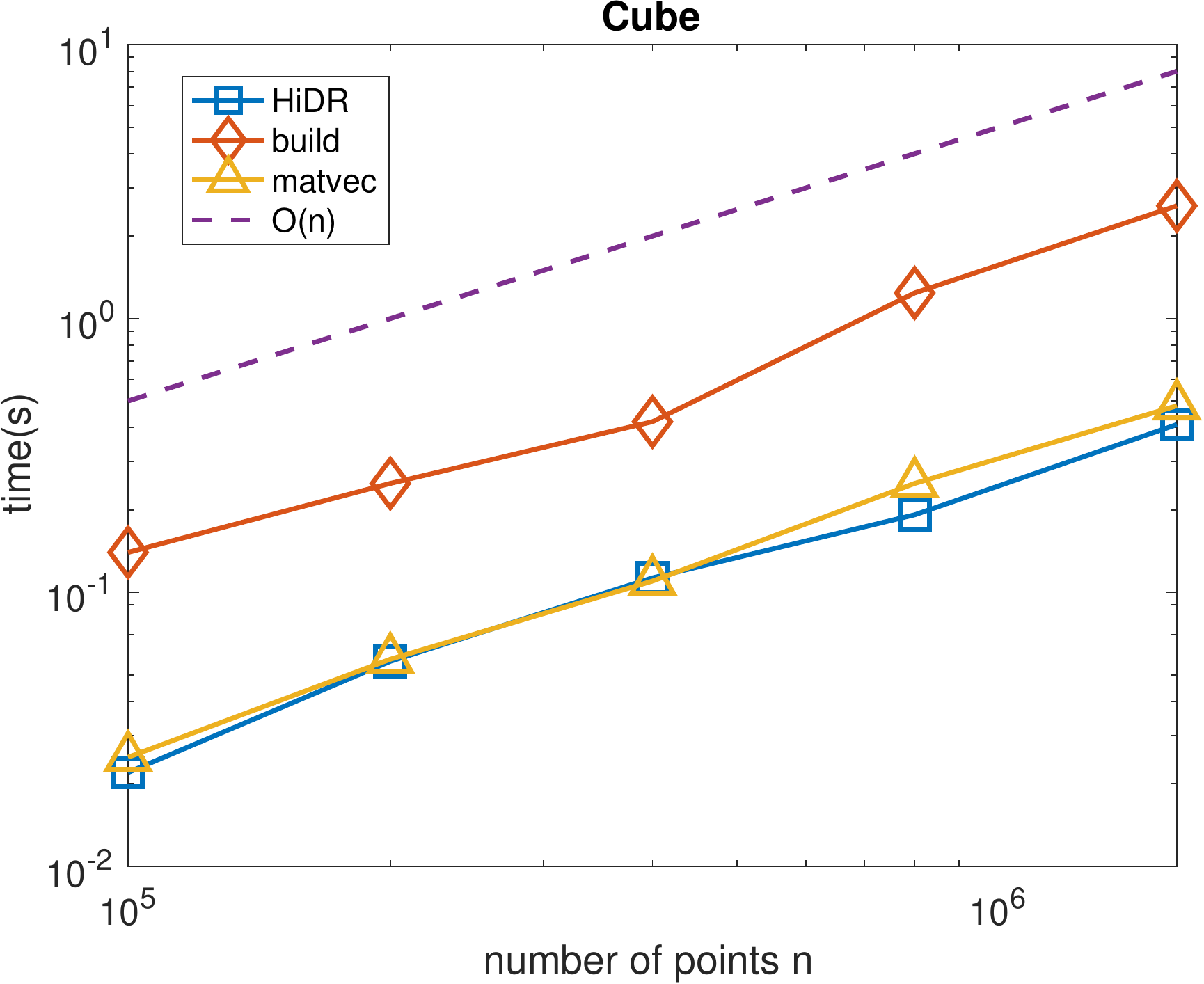}
    \includegraphics[scale=.3]{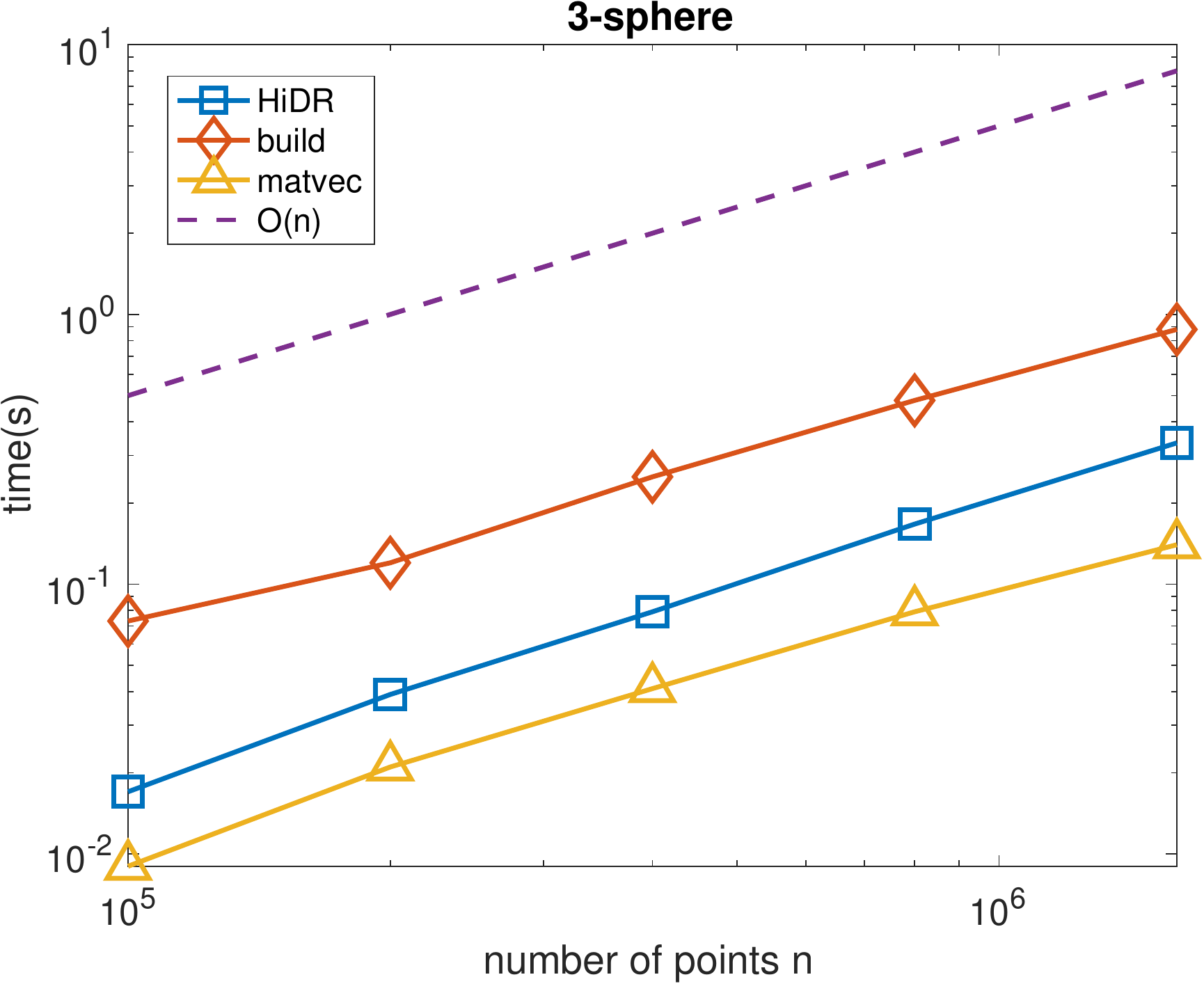}
    \includegraphics[scale=.3]{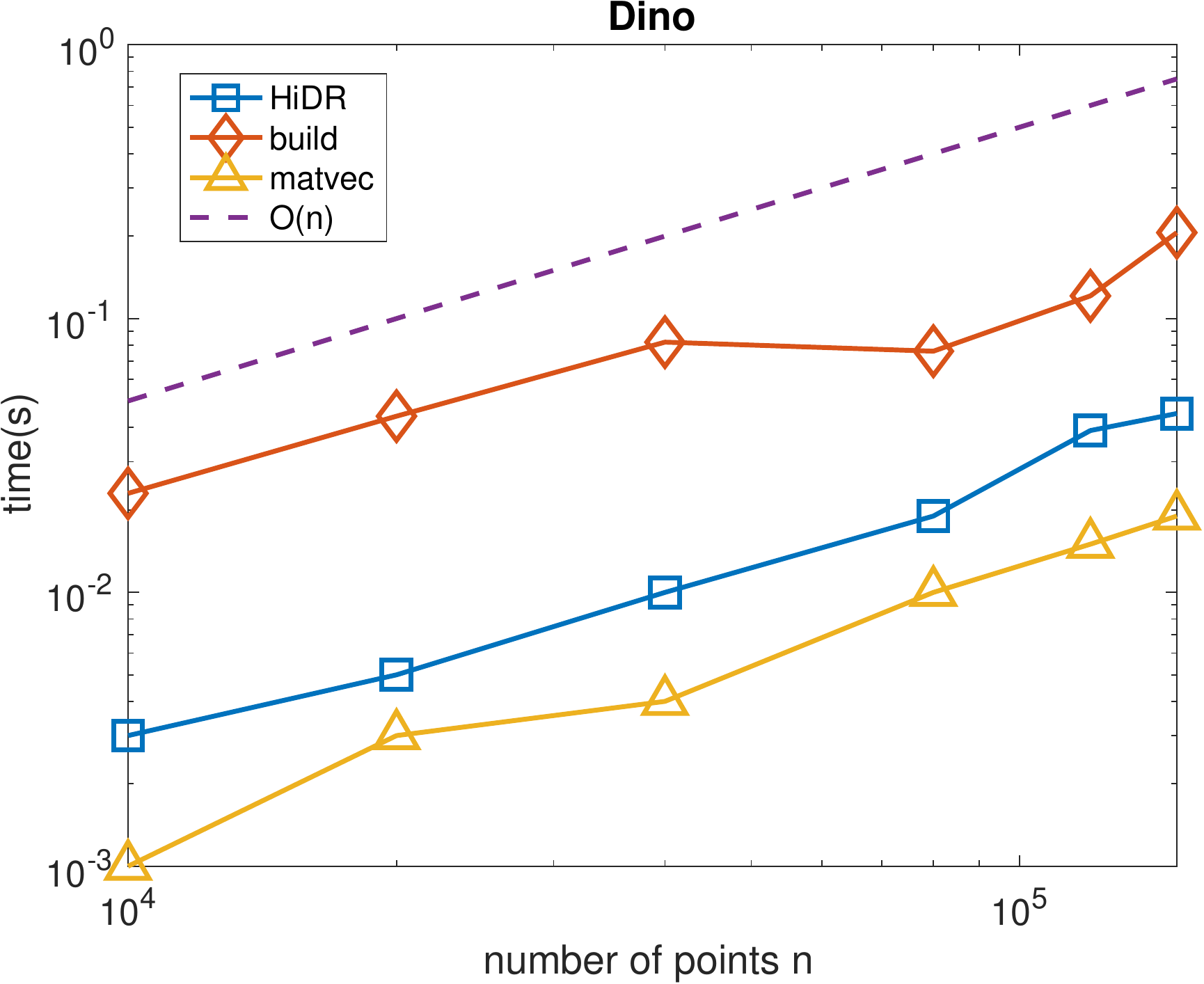}
    \caption{Section \ref{ssub:Scaling test} experiment: Timings of HiDR, $\mathcal{H}^2$ build, matvec for $n$-by-$n$ Coulomb kernel matrices with three datasets in $\mathbb{R}^3$}
% (left to right): cube, sphere, 3-sphere, Dino}
    \label{fig:Geo}
\end{figure}

\subsubsection{Scaling test for different kernels}
\label{ssub:kernels}
%Kernel-independence: use one data reduction for different kernels, show computational gain 
In this section, we test the proposed data-driven algorithm for the kernel functions in Table \ref{tab:kernels}.
{We show that the general data-driven algorithm is scalable for different types of kernel functions for the same approximation accuracy.
}
The ``3-sphere" dataset is used.

%the following four kernel functions:
%\[
%    \kappa_1(x,y)=\frac{x_1 y_1}{|x-y|^3},\quad
%    \kappa_2(x,y)=\exp(-|x-y|^2),\quad
%    \kappa_3(x,y)=\cos(x\cdot y),\quad
%    \kappa_4(x,y)=\exp{\left(-\frac{1}{1-0.1|x-y|^2}\right)},
%\]

For each kernel function, we measure the time cost of the proposed algorithm as the size of data $n$ increases.
The three types of costs - hierarchical data reduction, hierarchical matrix construction, matrix-vector multiplication - correspond to the three plots in Figure \ref{fig:kernelsTime}.

In Figure \ref{fig:kernelsTime}, each plot shows the timing for all four kernels.
The relative error for each case is $10^{-6}$.
It is easily seen from Figure \ref{fig:kernelsTime} that each cost scales linearly with data size $n$.
The algorithm is able to maintain accuracy for different types of kernel functions.
%Meanwhile, for the fixed approximation accuracy, the costs of HiDR and the subsequent hierarchical matrix construction  hardly change for different types of kernel functions tested.

\begin{table}
\caption{Kernel functions used in the experiments in Section \ref{sub:Scaling}. Here $\kappa_1(x,x)=0$.}
\label{tab:kernels}
\begin{center}
\begin{tabular}{cccc}
\hline
$\kappa_1(x,y)$ & $\kappa_2(x,y)$ & $\kappa_3(x,y)$ & $\kappa_4(x,y)$ \\
\hline 
$\frac{1}{|x-y|}$ & $\exp(-|x-y|^2)$ & $\cos(x\cdot y)$ & $\exp{\left(-\frac{1}{1-0.1|x-y|^2}\right)}$\\
\hline 
\end{tabular}
\end{center}
\end{table}

% subsubsection subsubsection name (end)
\subsubsection{HiDR once for all} 
\label{ssub:HiDR once for all}
{
In this experiment - called ``HiDR once for all" - we perform HiDR on the dataset to obtain representor sets and then use the representor sets to construct hierarchical matrix representation for the different kernel functions in Table \ref{tab:kernels} and Gaussian kernels with different bandwidths.
The key here is that, for a fixed compression level, HiDR is only performed \emph{once} and the same representor sets are used for \emph{all} kernels.
}

%After applying HiDR for the dataset to obtain representor sets (independent of the kernel functions),
%we construct the hierarchical matrix formats for the four kernels in Table \ref{tab:kernels} as well as the Gaussian kernel $\exp(-|x-y|/L^2)$ with different choices of the bandwidth parameter $L$.
%By varying the average size of the representor sets, different approximation accuracy can be reached.
%For a fixed size, HiDR is only performed \emph{once}.

Figure \ref{fig:kernels_Yi} shows the approximation error of the hierarchical matrix with respect to the average size of farfield representor sets $Y_i^*$.
Different error curves correspond to approximations to different kernels.
We see that by increasing the size of the farfield representor sets, the matrix approximation error is reduced effectively for all kinds of kernel functions.
Since HiDR is only applied once, the precomputation cost is almost negligible when amortized over multiple kernels.
The accuracy as seen from Figure \ref{fig:kernels_Yi} justifies the data-driven construction with the efficient kernel-independent HiDR.
We see that the data-driven approach is particularly useful when hierarchical matrices for different kernel functions or kernel parameters need to be computed.

{In applications like Gaussian processes, the bandwidth parameter for the Gaussian kernel is unknown and is determined by an iterative algorithm. 
Therefore, the bandwidth changes constantly, thus the kernel function.
This makes existing hierarchical matrix constructions inefficient because the entire hierarchical algorithm needs to be run from scratch every time the bandwidth changes. 
%For this reason, hierarchical matrix constructions with a high precomputation cost can become quite inefficient in those situations since the entire hierarchical algorithm needs to be run from scratch every time the bandwidth changes. 
The proposed data-driven approach, however, performs data reduction only \emph{once} and no matter what the bandwidth is, the hierarchical matrix representation can be constructed rapidly based on the computed representor sets.
It can be seen from Figure \ref{fig:kernels_Yi}(right) that for a wide range of bandwidth values, the approximation error decays effectively as more points are used in the representor sets. The nearly zero approximation error for the bandwidth $L=0.01$ is due to the fact that the admissible block is almost a zero matrix as $\exp(-|x-y|^2/0.01^2)\approx 0$ when $x$ and $y$ are away from each other.}

{Overall, it can be seen from the experiments that the data-driven approach serves as a black-box tool for rapidly computing hierarchical matrices for general kernel functions. It is especially efficient in the situation when multiple kernel functions need to be approximated.}

\begin{figure}[htbp]
    \centering
    \includegraphics[scale=.3]{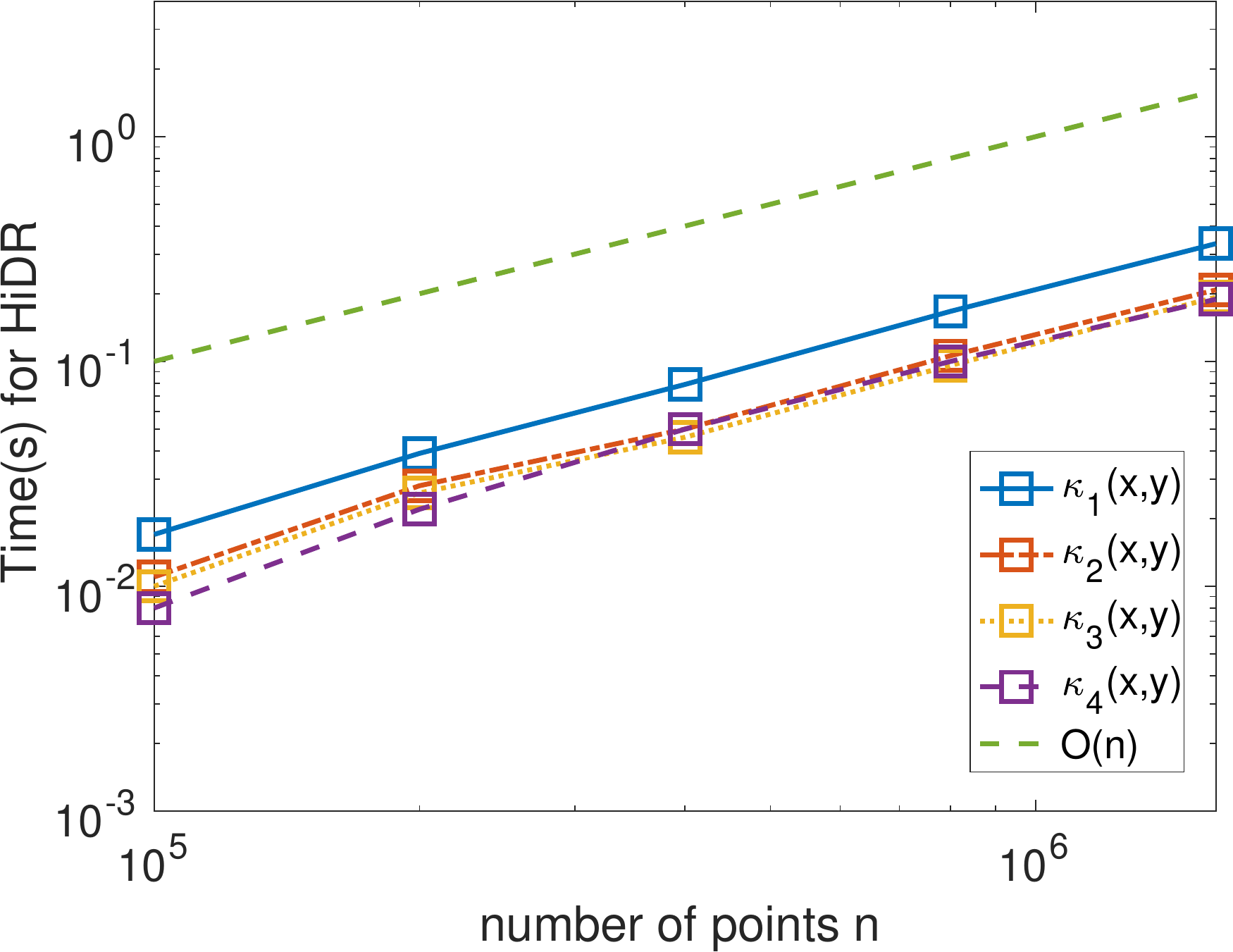}
    \includegraphics[scale=.3]{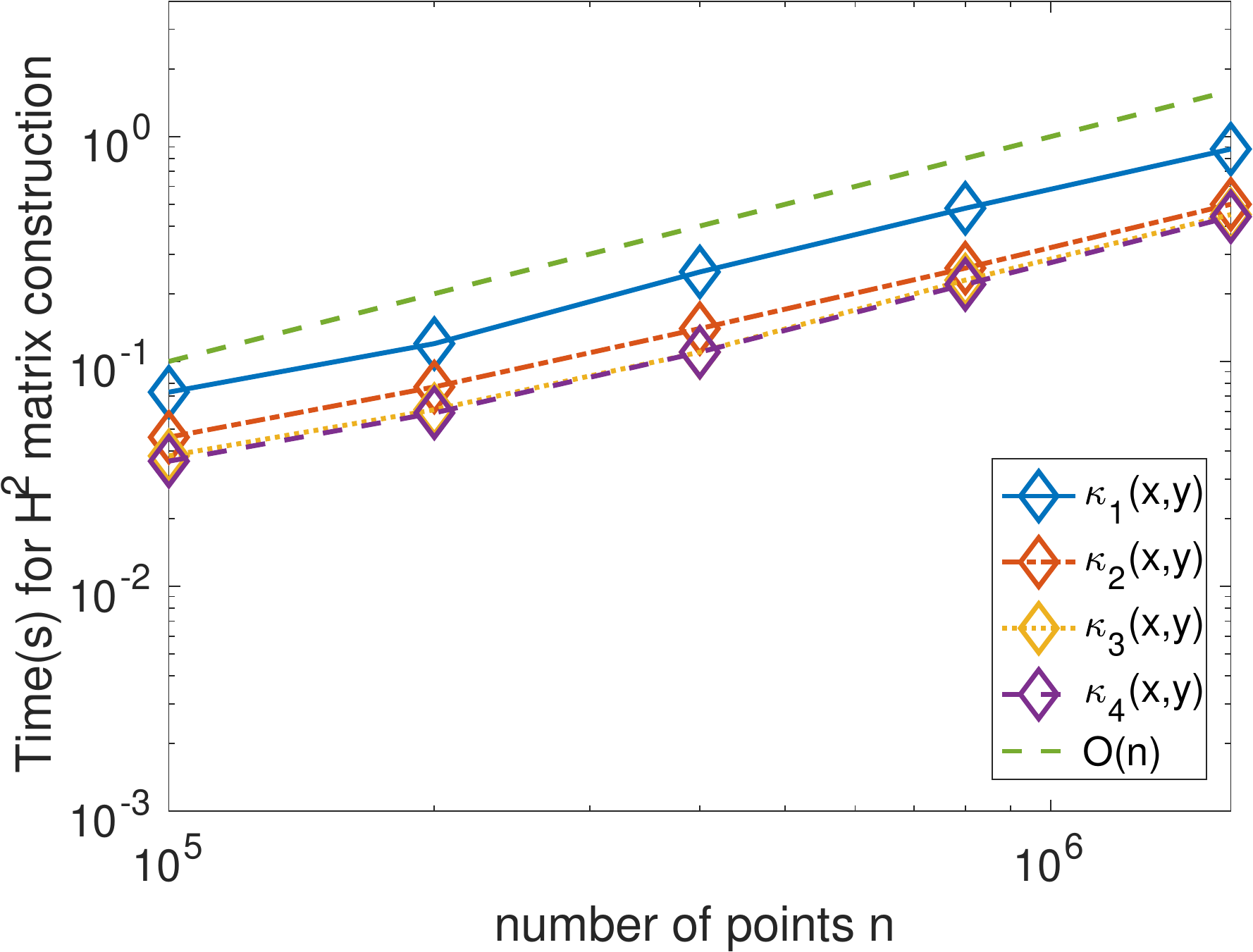}
    \includegraphics[scale=.3]{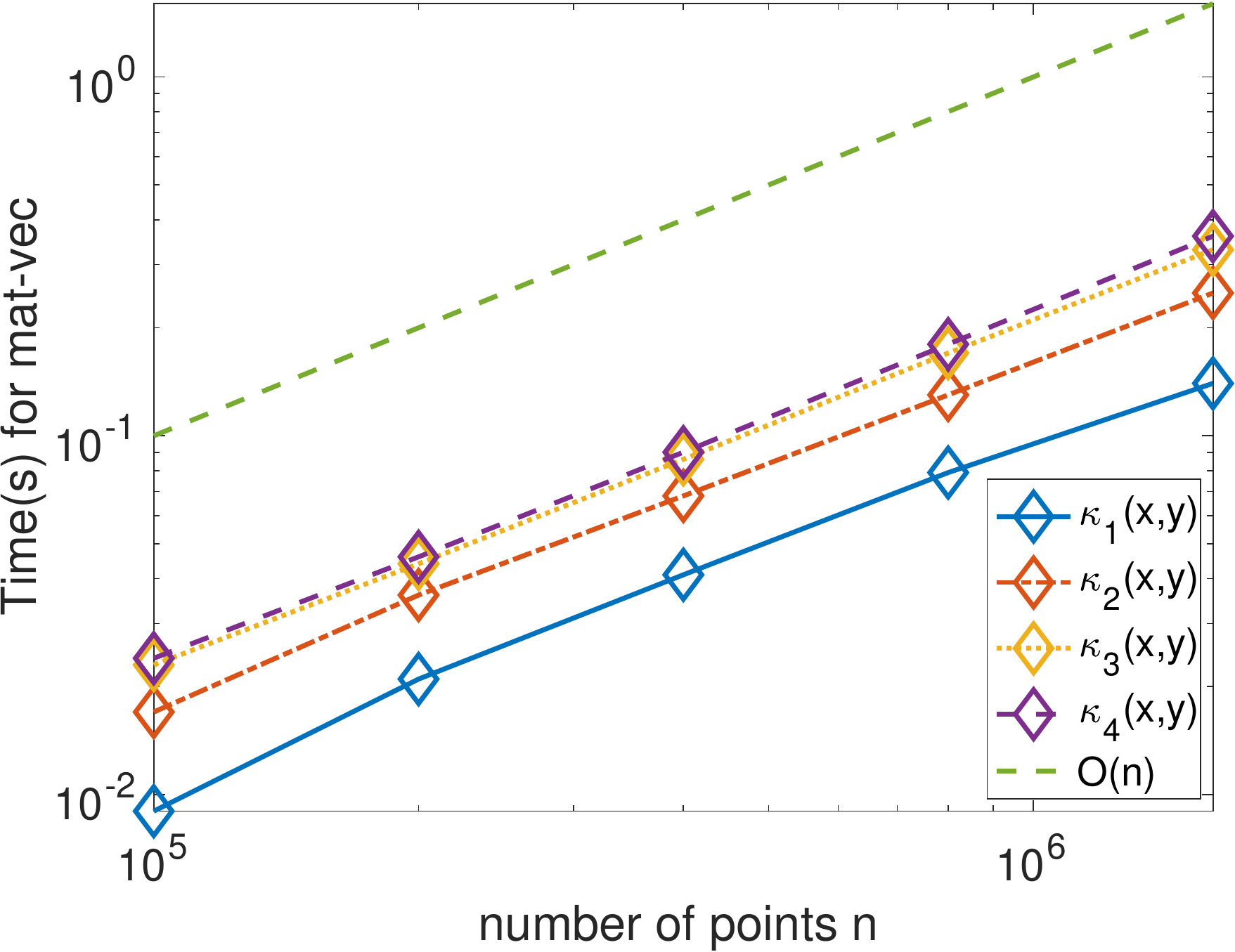}
    \caption{Section \ref{ssub:kernels} scaling test for kernels in Table \ref{tab:kernels}: CPU time for HiDR (left), hierarchical matrix construction (middle), matrix-vector multiplication (right)}
    \label{fig:kernelsTime}
\end{figure}

\begin{figure}[htbp] 
    \centering 
    \includegraphics[scale=.4]{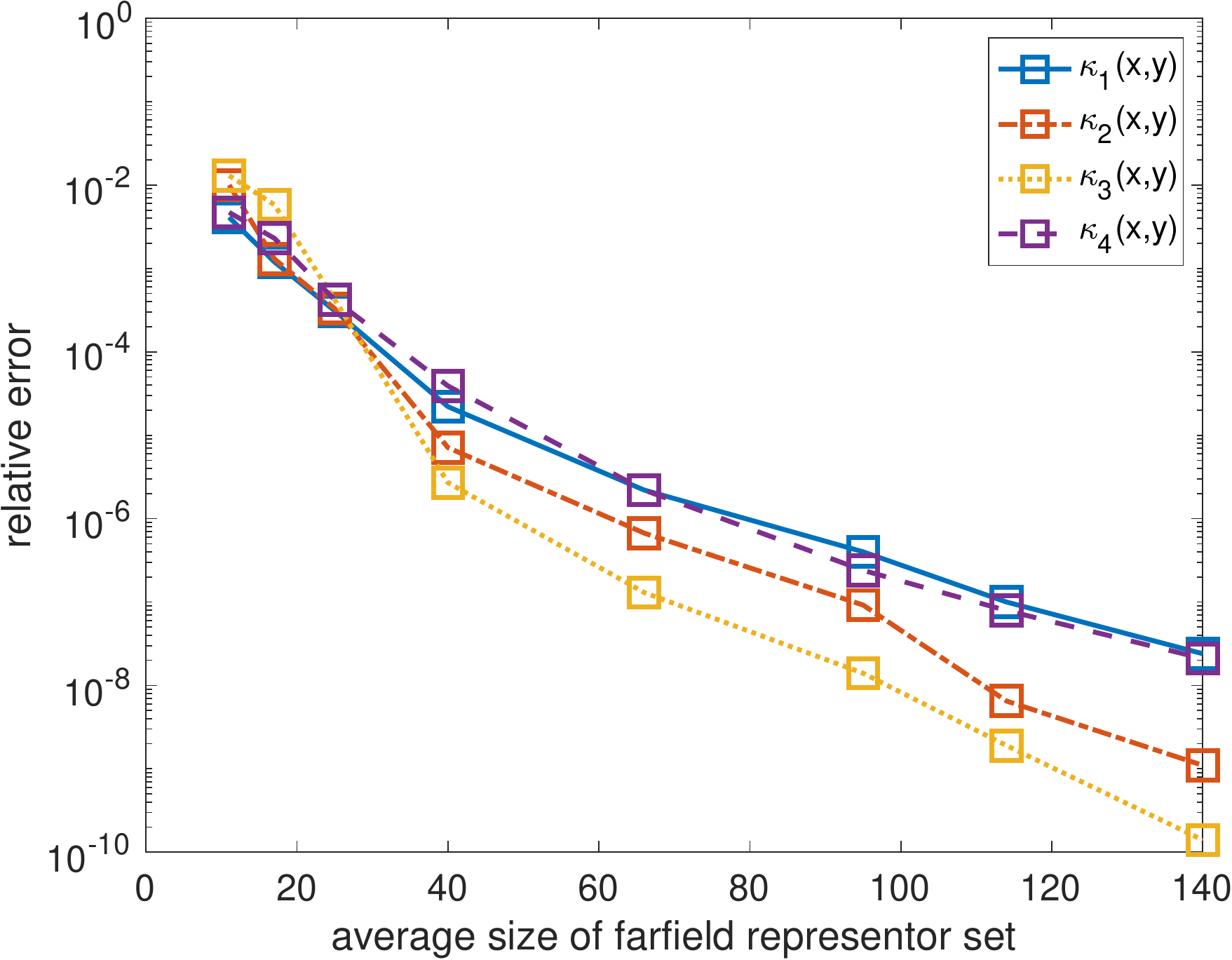}
    \includegraphics[scale=.4]{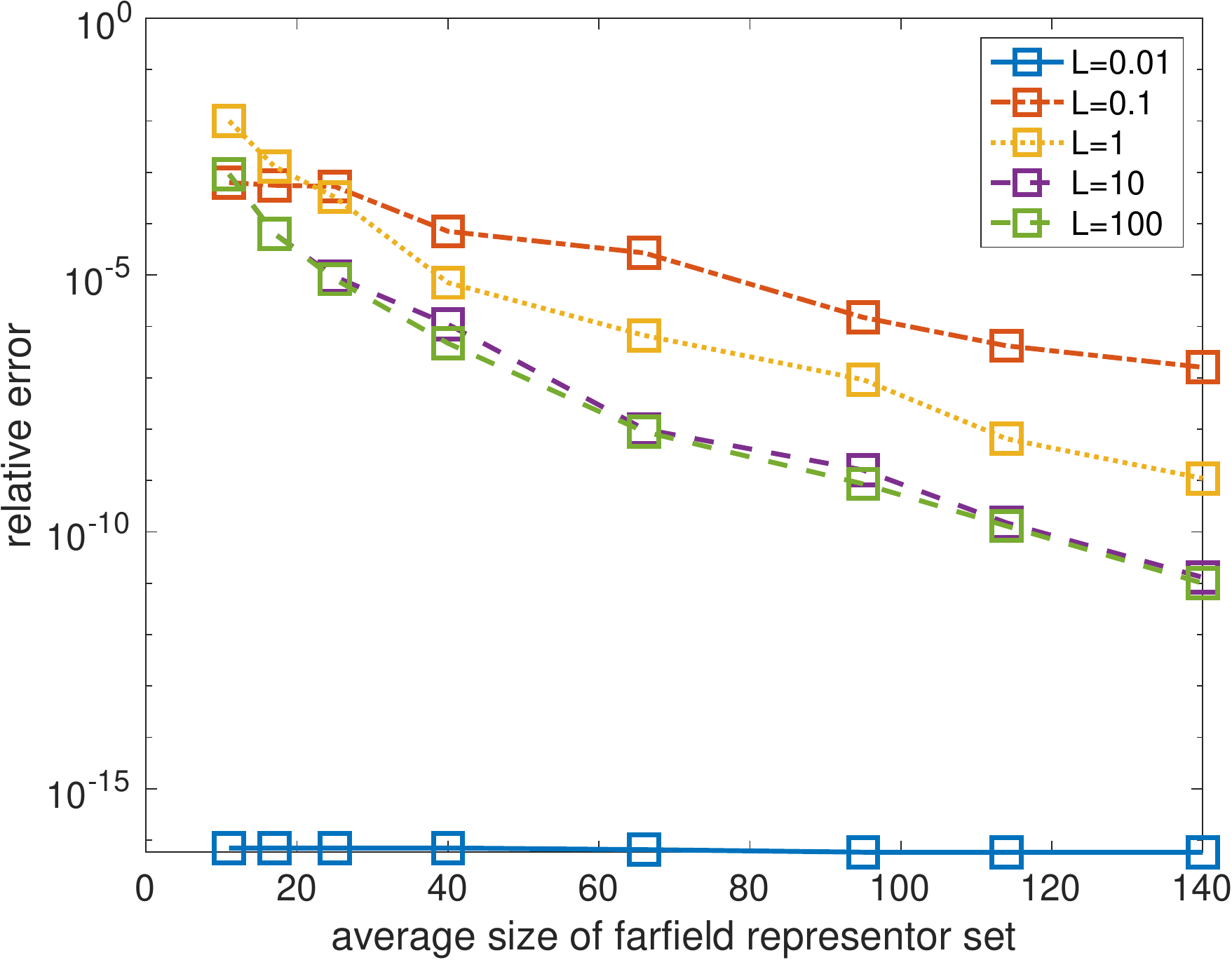}
    \caption{Section \ref{ssub:HiDR once for all}: perform HiDR only once to obtain representor sets and then use them to construct hierarchical matrices for multiple kernels: matrix approximation error vs average size of farfield representor sets $Y_i^*$. Left: kernels in Table \ref{tab:kernels}; Right: Gaussian kernel $\exp(-|x-y|^2/L^2)$ with different bandwidth $L=10^{k}$ with $k=-2,-1,0,1,2$}
    \label{fig:kernels_Yi}
\end{figure}

\subsection{Comparison to special-purpose methods for the Coulomb kernel}
\label{sub:Comparison special}
In this section, we compare the new general-purpose data-driven (`DD') construction to several optimized packages for the Coulomb kernel $\kappa(x,y)=\frac{1}{|x-y|}$, for example,
FMM3D\footnote{https://fmm3d.readthedocs.io/en/latest/}\cite{FMM3D}, PVFMM \cite{pvfmm}, 
%H2Pack with proxy point method \cite{h2pack} 
and Proxy Surface method (cf. \cite{Martinsson20051,ID2005,Gillman2012}) implemented in H2Pack \cite{h2pack}.
%All these methods have linear scaling 
These methods are specialized for the Coulomb kernel to offer better efficiency in practice than the interpolation-based methods for constructing $\mathcal{H}^2$ matrices.

DD, FMM3D, PVFMM and H2Pack are compiled using Intel C/C++/Fortran compiler v19.0.5 with optimization flags ``-xHost -O3''. Intel MKL 19.0.5 is used in all tested libraries to perform general matrix-vector and matrix-matrix multiplications. 
DD, H2Pack, and FMM3D 
{use one thread per CPU core and 24 cores on one computing node. PVFMM uses MVAPICH2 2.3.2 as the MPI backend and uses one MPI process with 24 cores on one computing node.}

We use the same datasets as in Section \ref{sub:Scaling}.
For every method, the total time is computed as
$$\text{total time = precomputation + $\mathcal{H}^2$ construction + matrix-vector multiplication}.$$
%the total time includes the precomputation, $\mathcal{H}^2$ matrix construction, and matrix-vector multiplication.
{FMM3D and proxy surface method do not have precomputation, while PVFMM and DD require precomputation.}
% since it computes the hierarchical matrix representation on-the-fly during the matrix-vector multiplication.
For DD, the precomputation refers to HiDR.

%Hence it is easier to compare the total computation time for evaluating the performance of different types of methods.

Timings for precomputation, hierarchical build, matvec, are shown in  
Figure \ref{fig:Geo-precomp}, Figure \ref{fig:Geo-build}, Figure \ref{fig:Geo-mv}, respectively.
The total time is shown in Figure \ref{fig:Geo-total}.
%The precomputation time comparison is shown in Figure \ref{fig:Geo-precomp}, \cdf{the hierarchical matrix construction time after precomputation is shown in Figure \ref{fig:Geo-build}, the matrix-vector (matvec) multiplication time comparison is shown in Figure \ref{fig:Geo-mv},} and the total time comparison is shown in Figure \ref{fig:Geo-total}.
The relative error (in 2-norm) for each test is $10^{-6}$.

From Figure \ref{fig:Geo-precomp}, we see that the proposed hierarchical data-reduction (HiDR) requires significantly lower precomputation cost compared to PVFMM.
This is due to the fact that the kernel matrix is \emph{never} accessed in HiDR and \emph{no} algebraic compression is computed.
Moreover, we see from Figure \ref{fig:Geo-precomp} that the advantage of the data-driven method becomes more obvious for irregular data from a manifold, such as 3-sphere and Dino.
It can be seen that PVFMM has almost constant cost independent of the size of the dataset $n$. 
This is because PVFMM treats matrix compression as a continuous problem (thus independent of the size of data) and the precomputation involves solving integral equations on spheres to facilitate the farfield compression.

From Figure \ref{fig:Geo-build}, 
we see that FMM3D outperforms other methods in hierarchical construction.
This is because, unlike other methods, FMM3D does \emph{not} compute and store a hierarchical matrix representation. Instead, it computes the hierarchical representation when performing matrix-vector multiplication.
The other methods have similar performance, where no single method performs significantly better than others across all datasets.
It should be noted that, for methods with precomputation, the hierarchical construction time can be further reduced at the expense of more precomputation time.
In principle, more time spent in precomputation could yield faster hierarchical construction.

%The data-driven method has a higher cost than other methods for the cube dataset. 
%There are two possible reasons. Firstly, when there are sufficiently many quasi-uniform points in the unit cube, the property of the kernel matrix becomes similar to that of the corresponding integral operator and thus matrix approximations based on continuous formulations become quite suitable.
%Secondly, the data-driven method, developed for general kernels, is not specifically optimized for the Coulomb kernel.
%Compared to other three methods, \emph{no} particular property of the Coulomb kernel is ever used in the data-driven construction.
%It should be noted that, 
%the fast hierarchical construction for PVFMM and ProxyPoint
%come at the expense of high precomputation costs (see Figure \ref{fig:Geo-precomp}) that dominate the hierarchical construction costs.
%The data-driven method, on the contrary, does \emph{not} sacrifice precomputation for hierarchical matrix construction.

For matvec, Figure \ref{fig:Geo-mv} shows that the data-driven method and proxy surface method achieve similar performance that is in general better than FMM3D and PVFMM.
FMM3D is significantly slower than other methods due to the on-the-fly hierarchical construction.
%Three methods - DD, ProxyPoint and PVFMM - achieve similar matvec performance for the unit cube dataset, while FMM3D is significantly slower than other methods due to the on-the-fly hierarchical construction.
For data sampled from a manifold, PVFMM is outperformed by DD and Proxy Surface.
The results in Figure \ref{fig:Geo-mv} justify the efficiency of the hierarchical matrix representation built from the fast HiDR in Figure \ref{fig:Geo-precomp}.

For the total computation time, as can be seen from Figure \ref{fig:Geo-total}, 
we see that Proxy Surface provides the best performance overall, followed by DD.
The advantage of DD is more evident for data from a manifold, e.g. 3-sphere, Dino.
%PVFMM and ProxyPoint perform well when $n$ is sufficiently large because the matrix problem becomes more like a continuous problem.
In general, we see that the data-driven method leads to a lot more computational savings when data is sampled from a low-dimensional manifold.

It should be emphasized that the Proxy Surface method is a specialized method optimized for the Coulomb kernel to offer superior efficiency, while DD is a general-purpose approach that can be applied to a variety of kernel functions (cf. Table \ref{tab:kernels}).
%Note that unlike the special-purpose methods, the proposed data-driven method does \emph{not} require any analytic information from the kernel function.
Unlike the special-purpose methods, no analytic property of the kernel function is used in the data-driven hierarchical construction.
It is nonetheless possible to design specialized data-driven algorithms for the kernel function of interest to improve efficiency.
{Overall, we conclude from the results in Figure \ref{fig:Geo-total} that the data-driven method, as a black-box tool for hierarchical matrix computations, also offers excellent efficiency for special kernels without utilizing any specific property of the kernel.}

\begin{figure}[htbp]
    \centering
    \includegraphics[scale=.3]{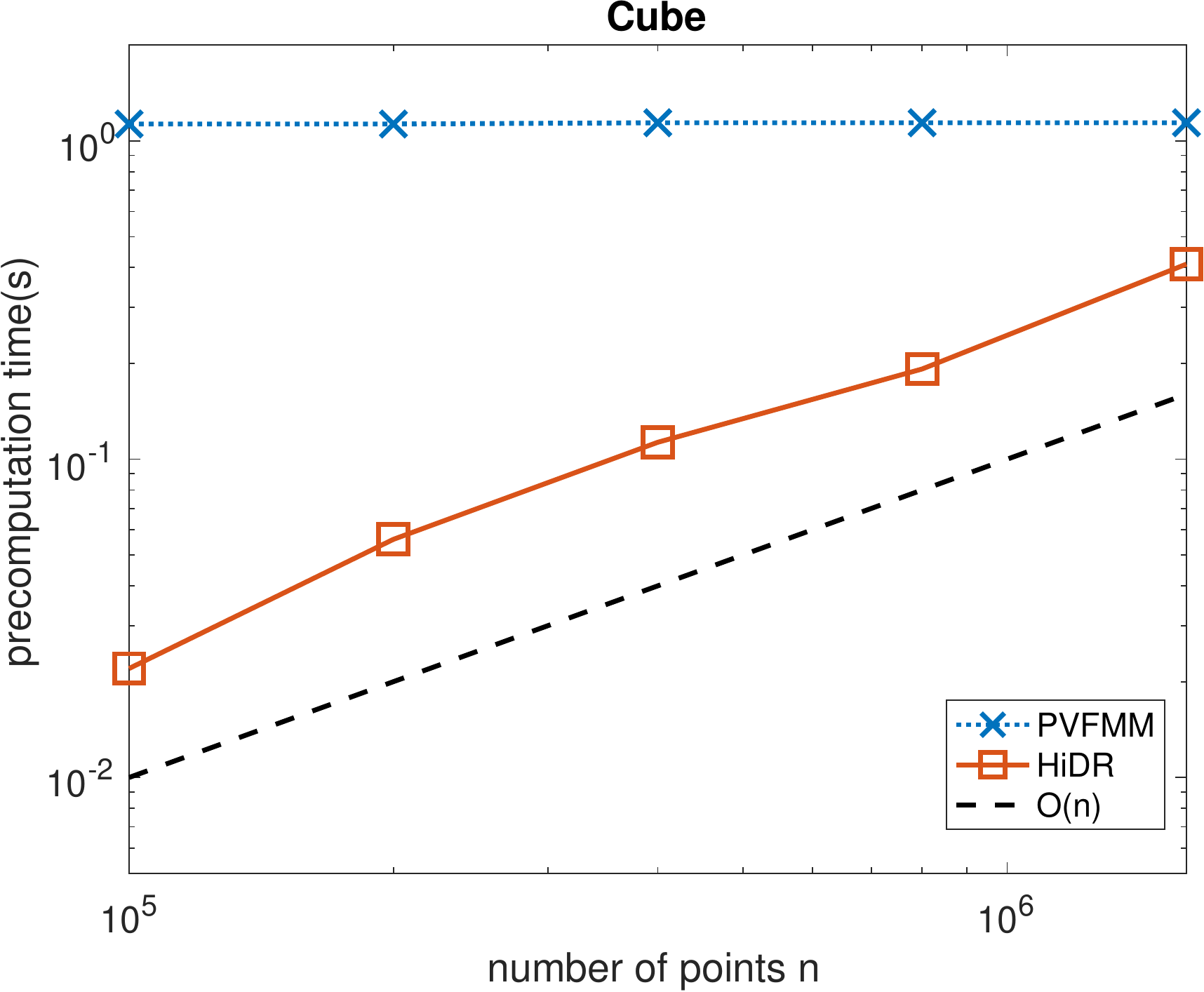}
    \includegraphics[scale=.3]{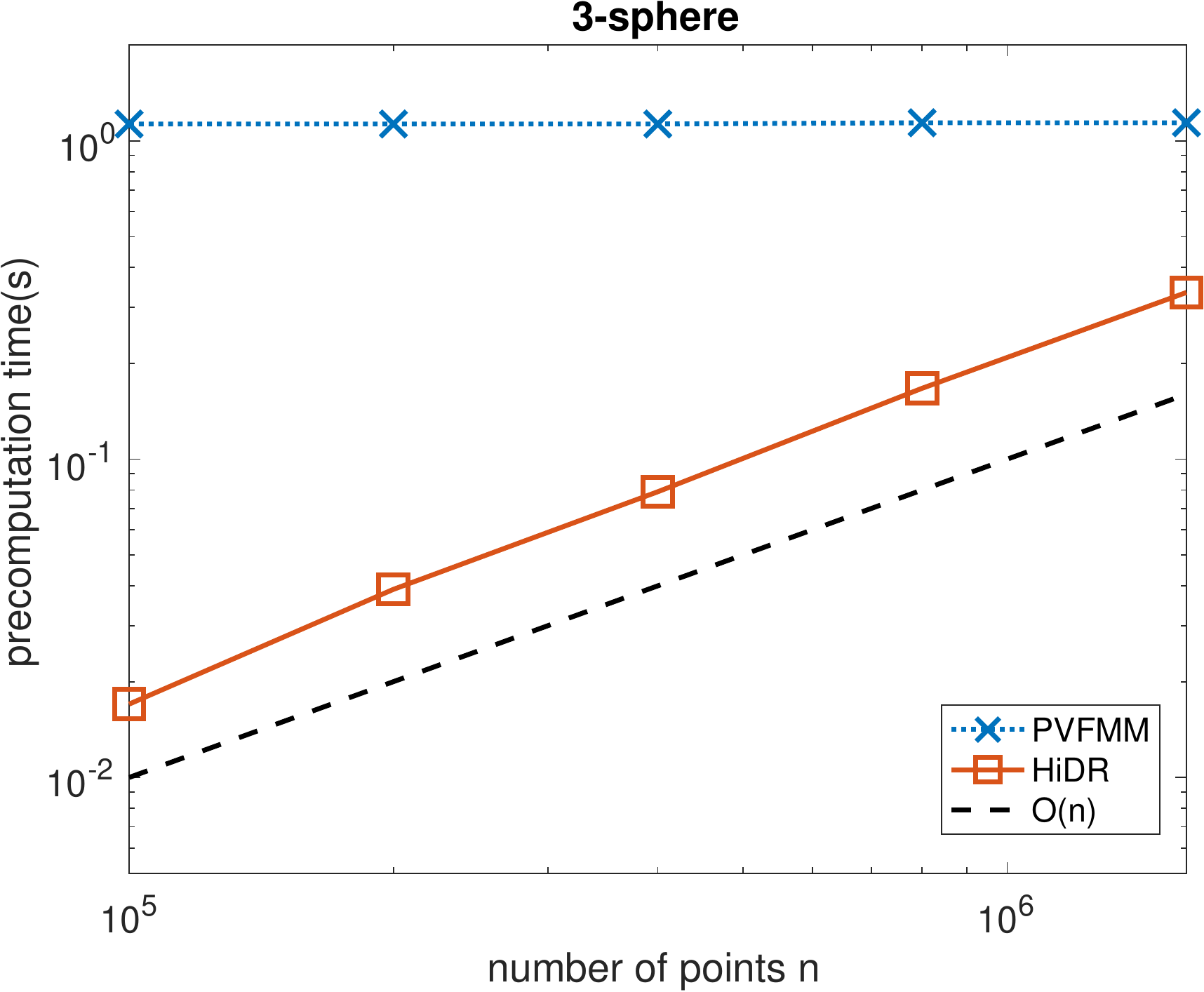}
    \includegraphics[scale=.3]{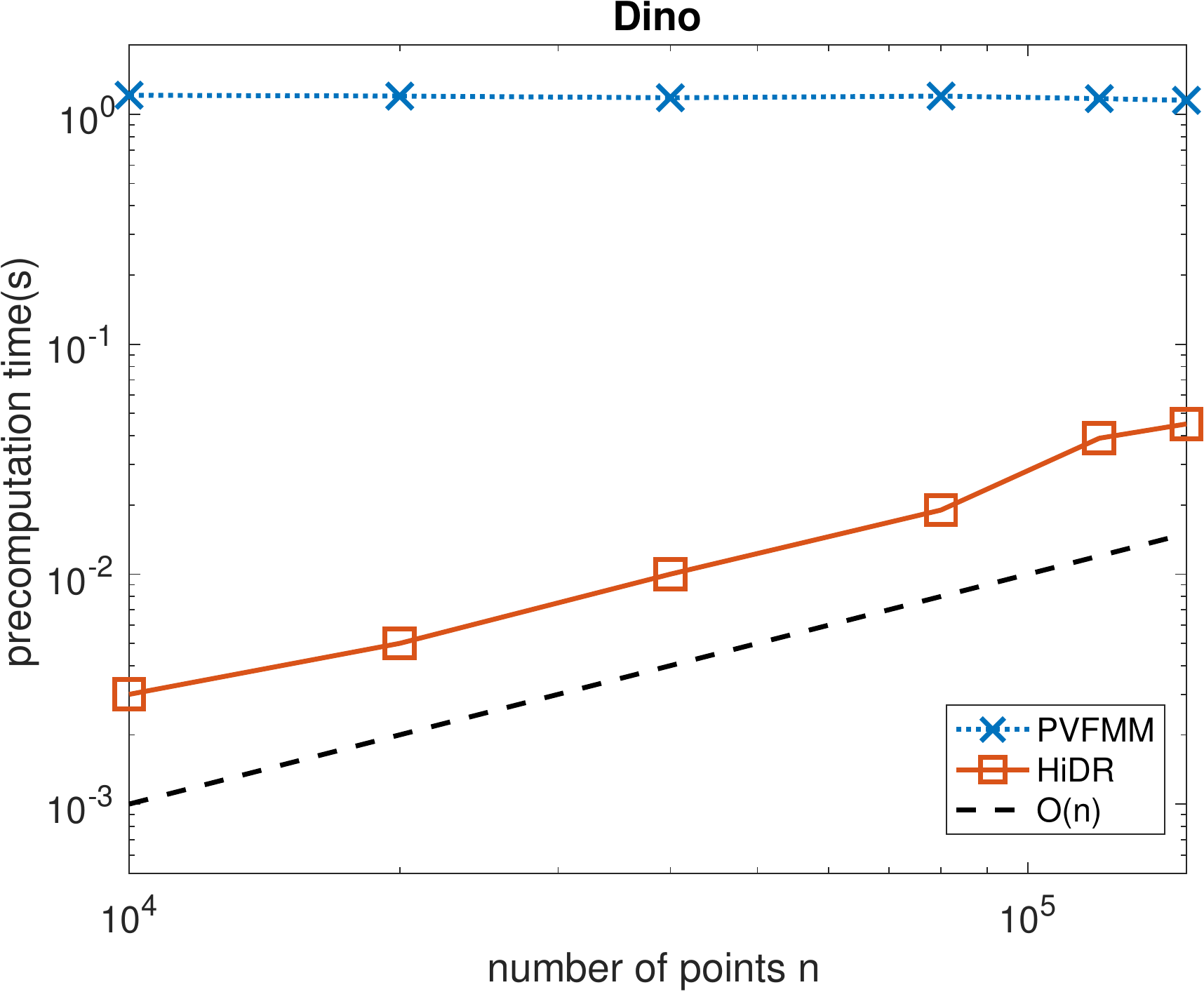}
    \caption{Section \ref{sub:Comparison special} experiment: Precomputation time of PVFMM and DD for approximating $n$-by-$n$ Coulomb kernel matrices with datasets Cube, 3-sphere, Dino}
    \label{fig:Geo-precomp}
\end{figure}

\begin{figure}[htbp]
    \centering
    \includegraphics[scale=.3]{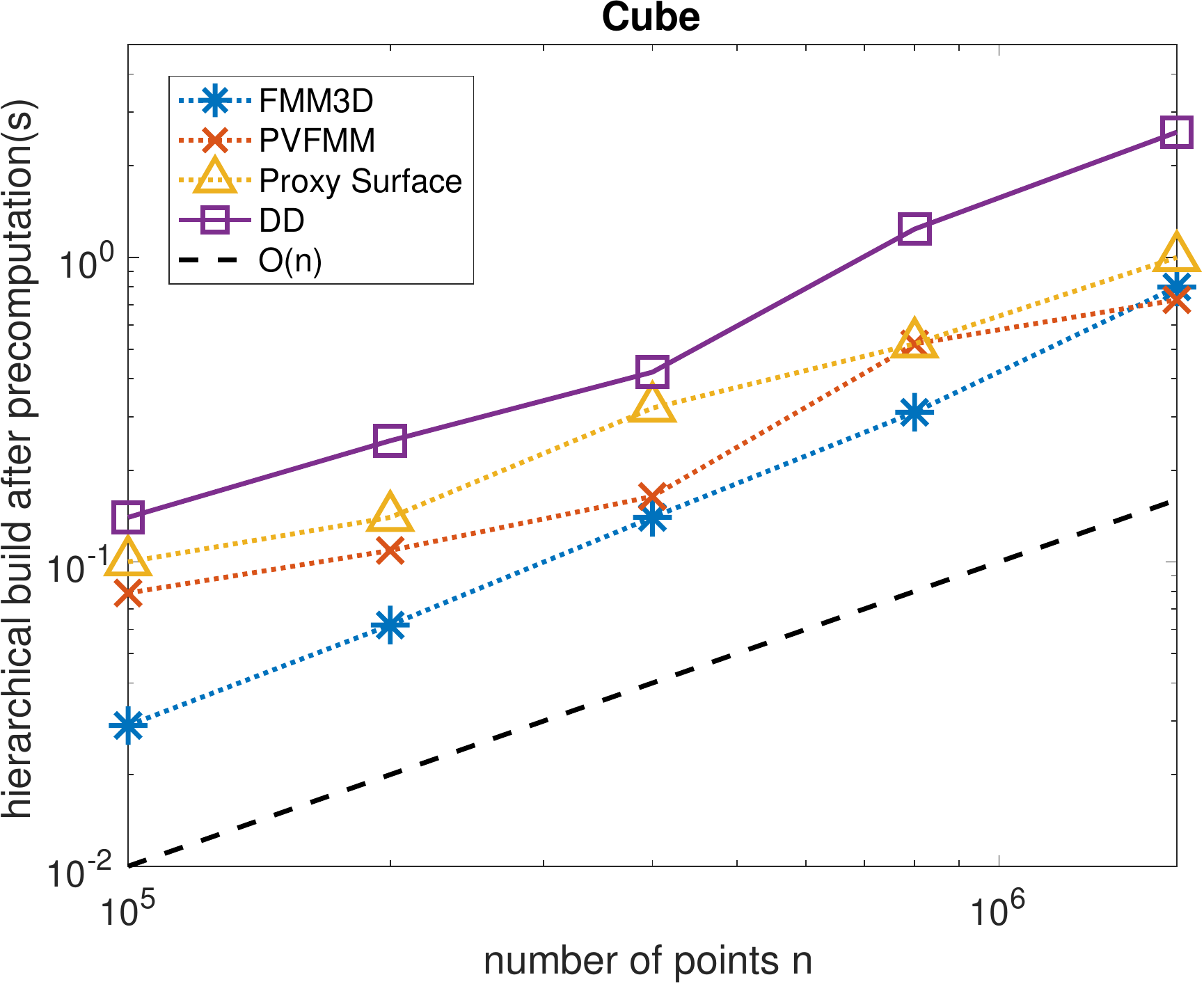}
    \includegraphics[scale=.3]{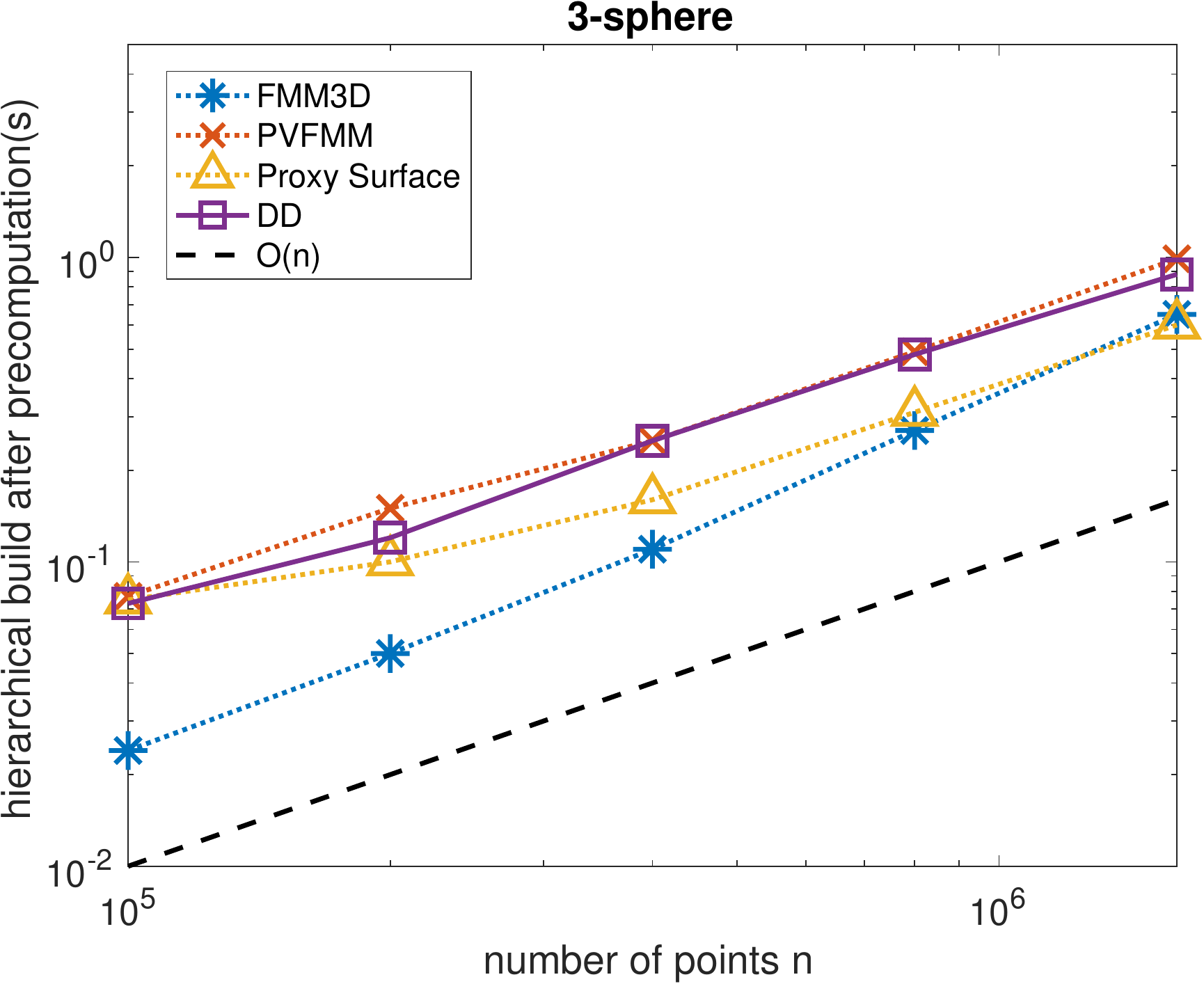}
    \includegraphics[scale=.3]{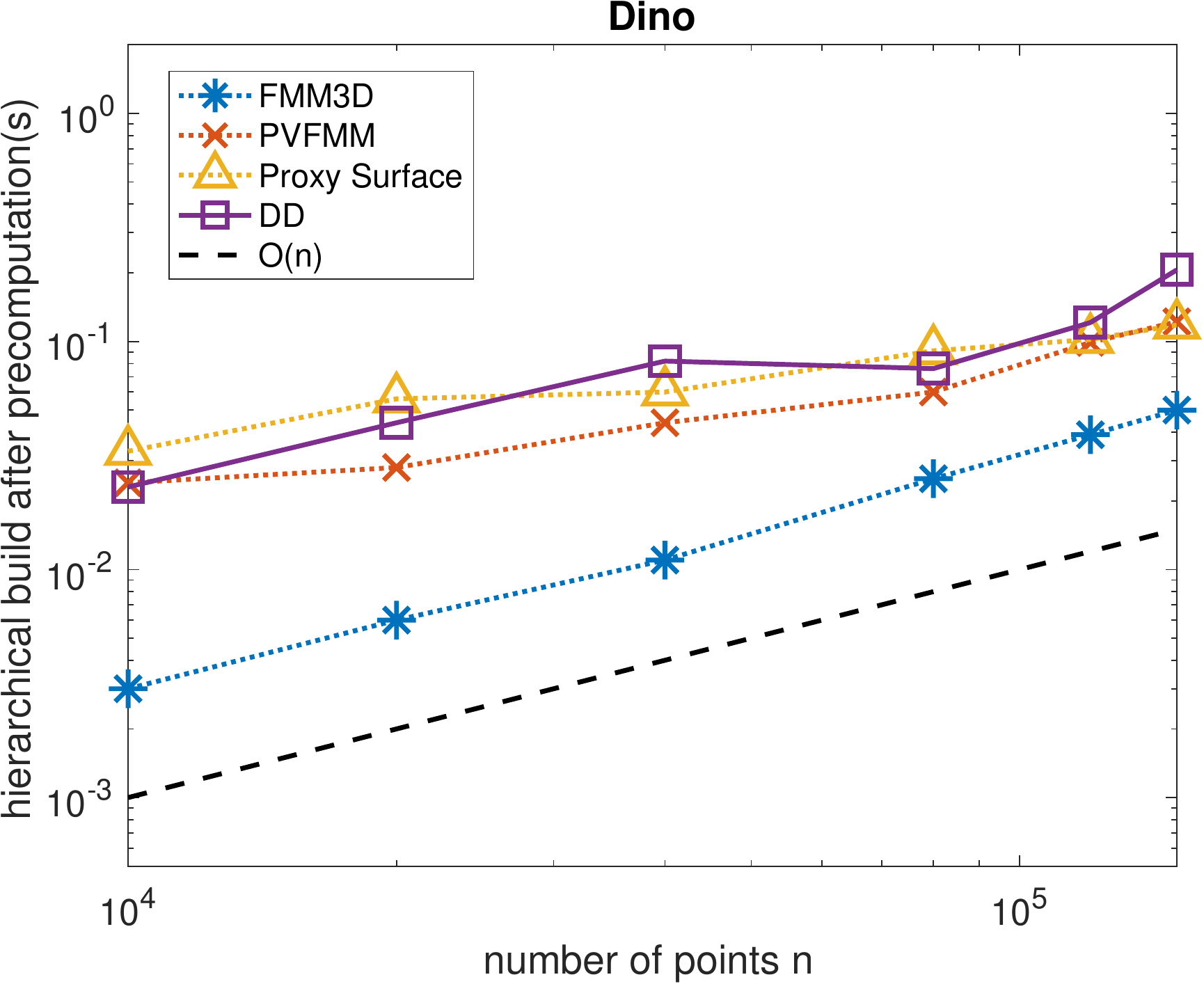}
    \caption{Section \ref{sub:Comparison special}: Hierarchical construction time (after precomputation) of FMM3D, PVFMM, Proxy Surface and DD for approximating $n$-by-$n$ Coulomb kernel matrices with datasets Cube, 3-sphere, Dino}
    \label{fig:Geo-build}
\end{figure}

\begin{figure}[htbp]
    \centering
    \includegraphics[scale=.3]{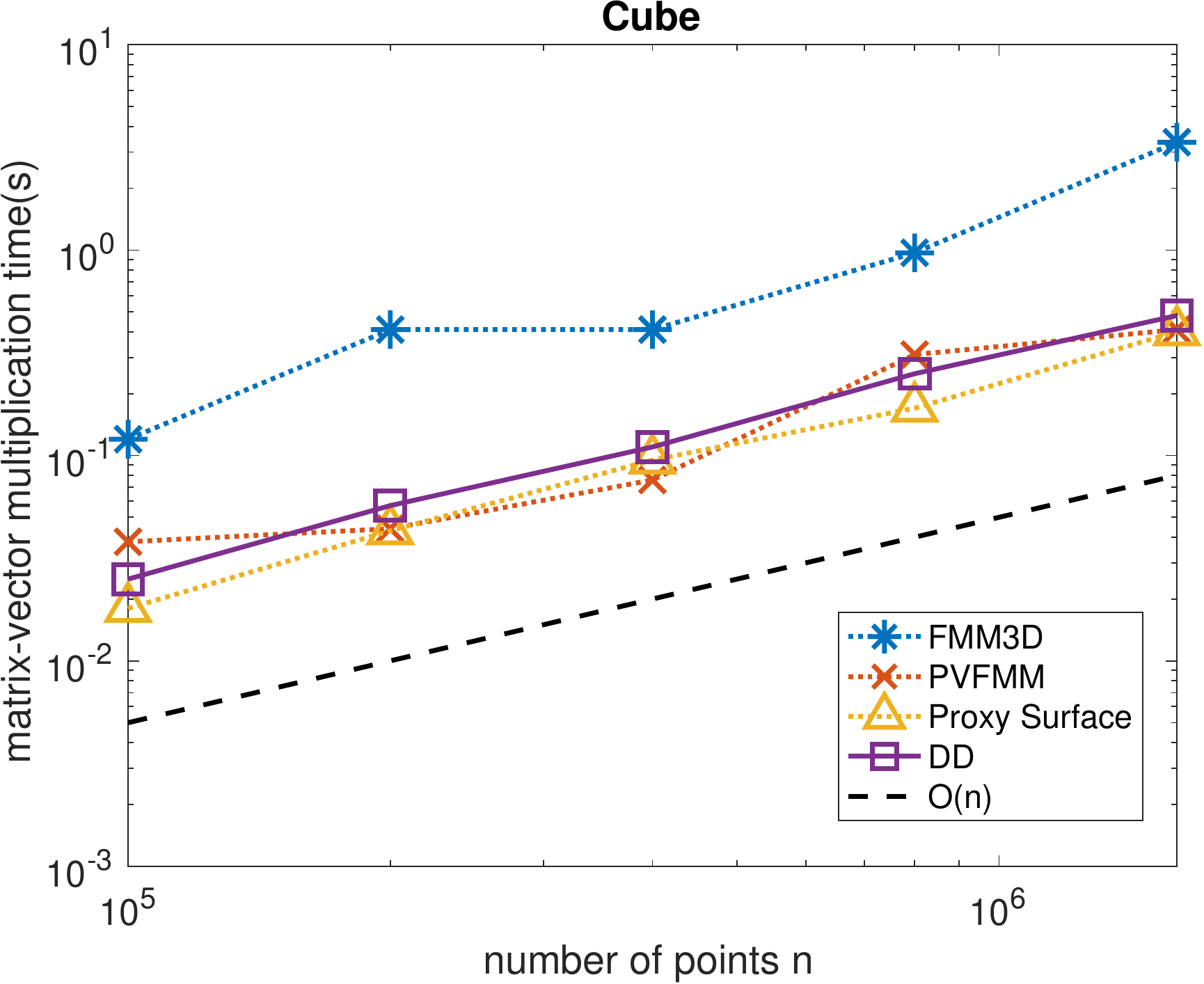}
    \includegraphics[scale=.3]{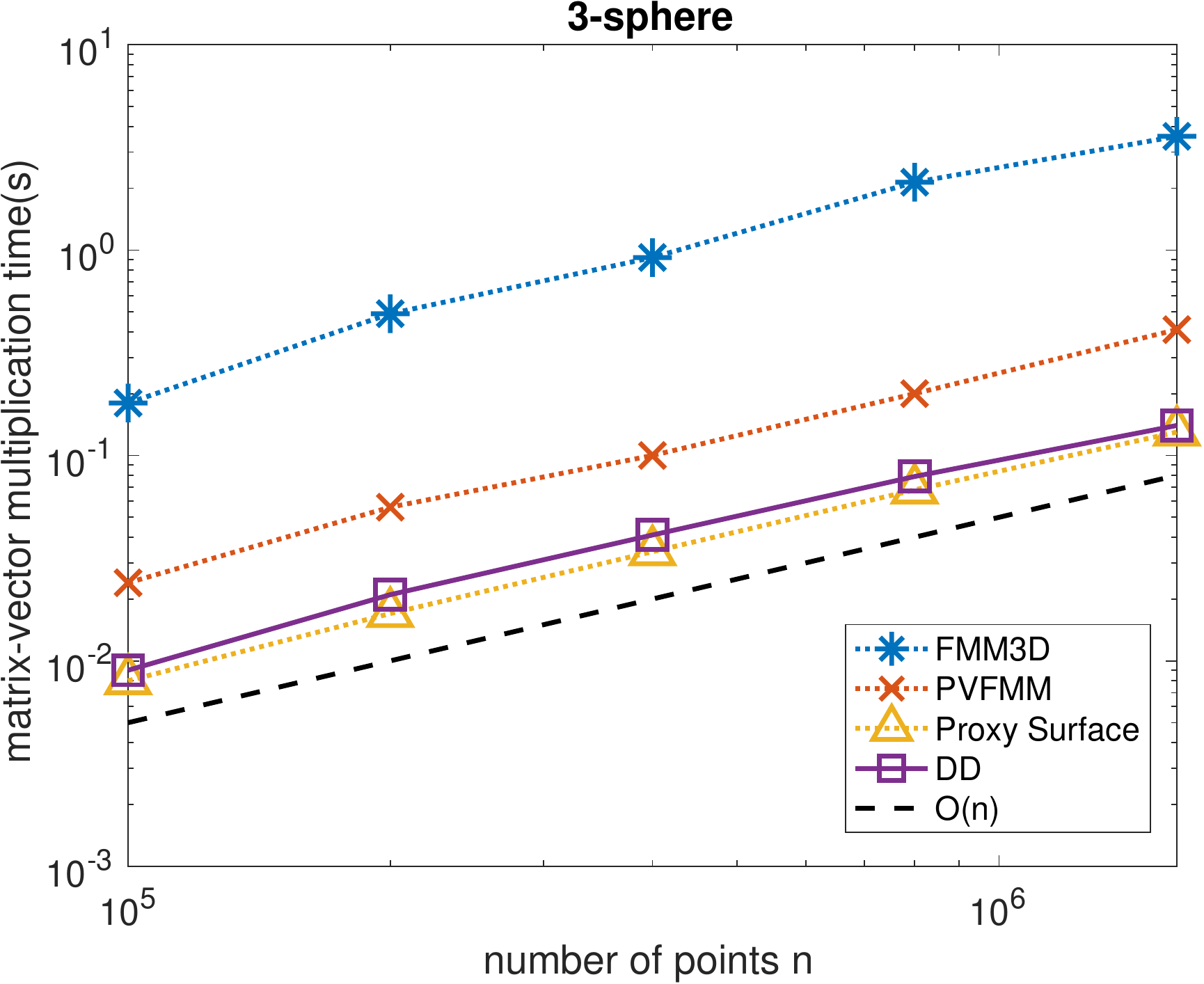}
    \includegraphics[scale=.3]{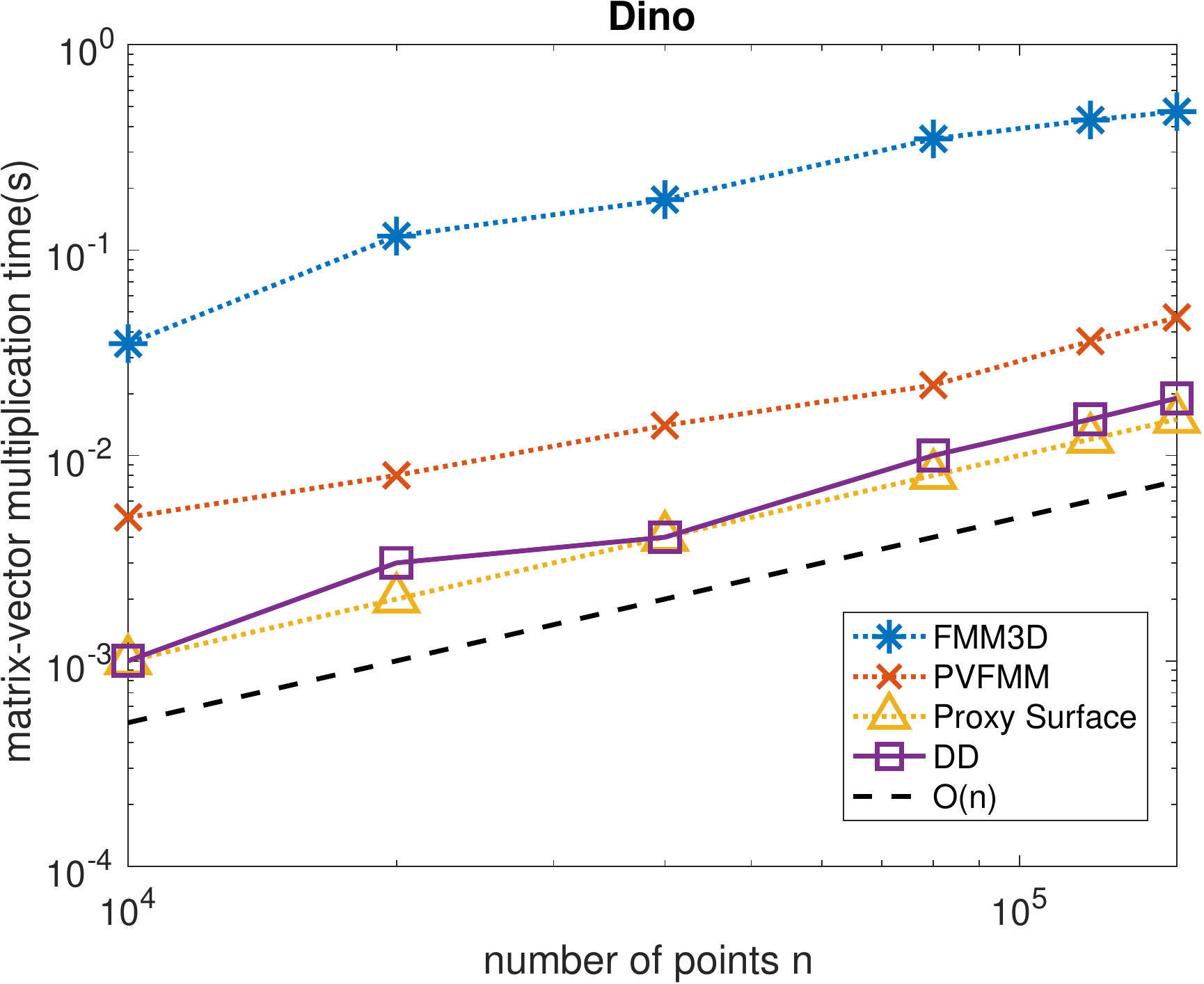}
    \caption{Section \ref{sub:Comparison special}: Matrix-vector multiplication time of FMM3D, PVFMM, Proxy Surface and DD for approximating $n$-by-$n$ Coulomb kernel matrices with datasets Cube, 3-sphere, Dino}
    \label{fig:Geo-mv}
\end{figure}

\begin{figure}[htbp]
    \centering
    \includegraphics[scale=.3]{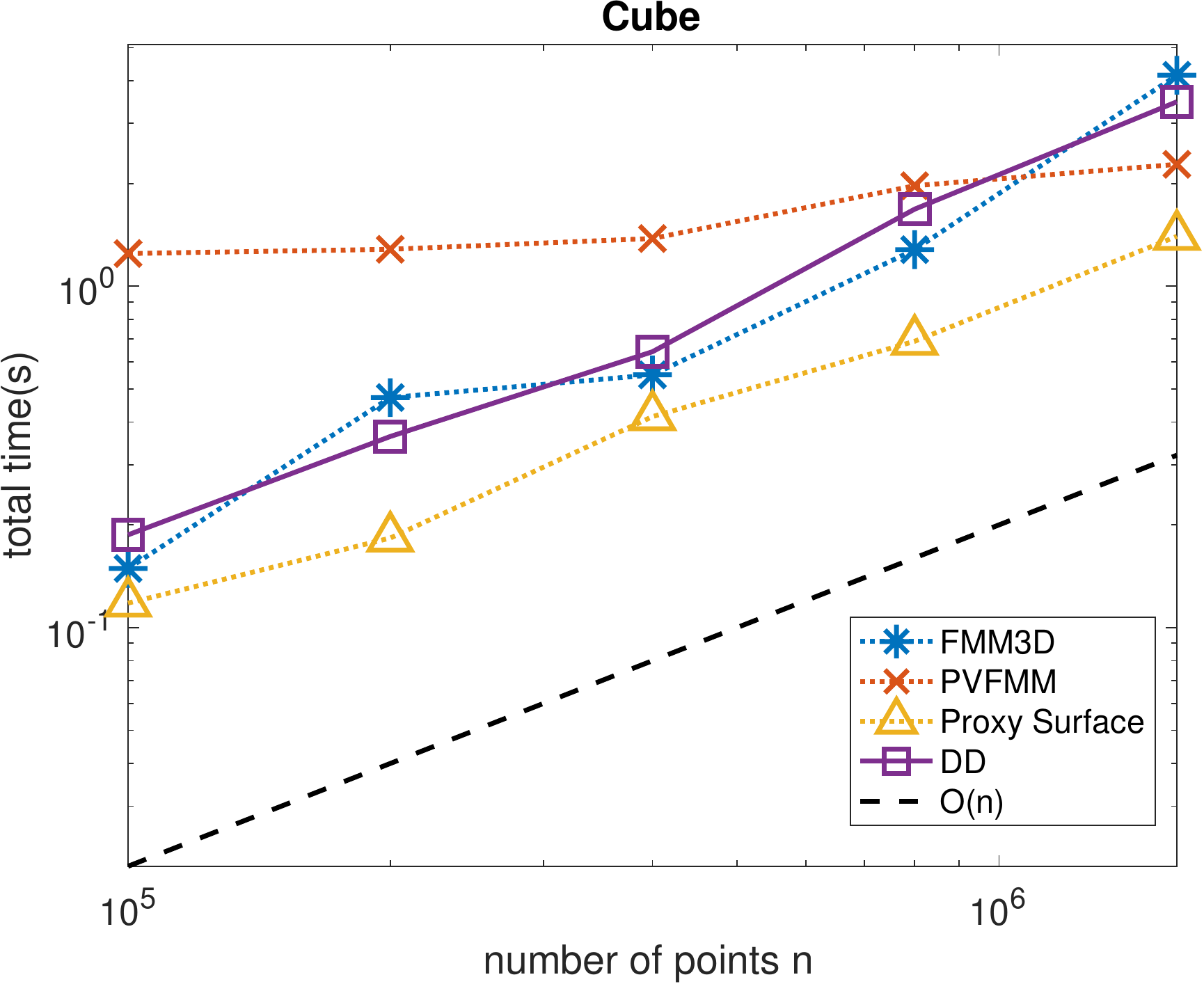}
    \includegraphics[scale=.3]{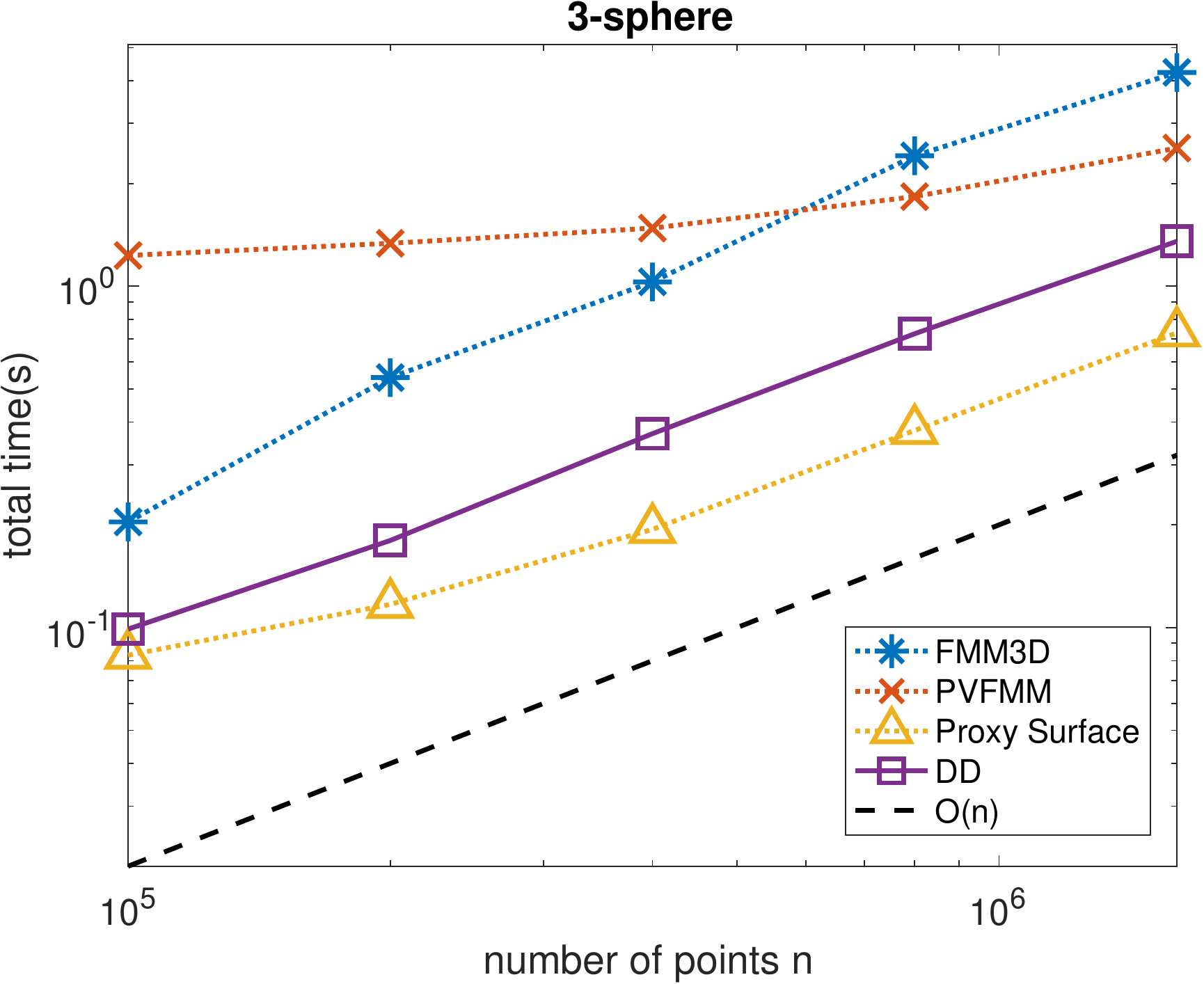}
    \includegraphics[scale=.3]{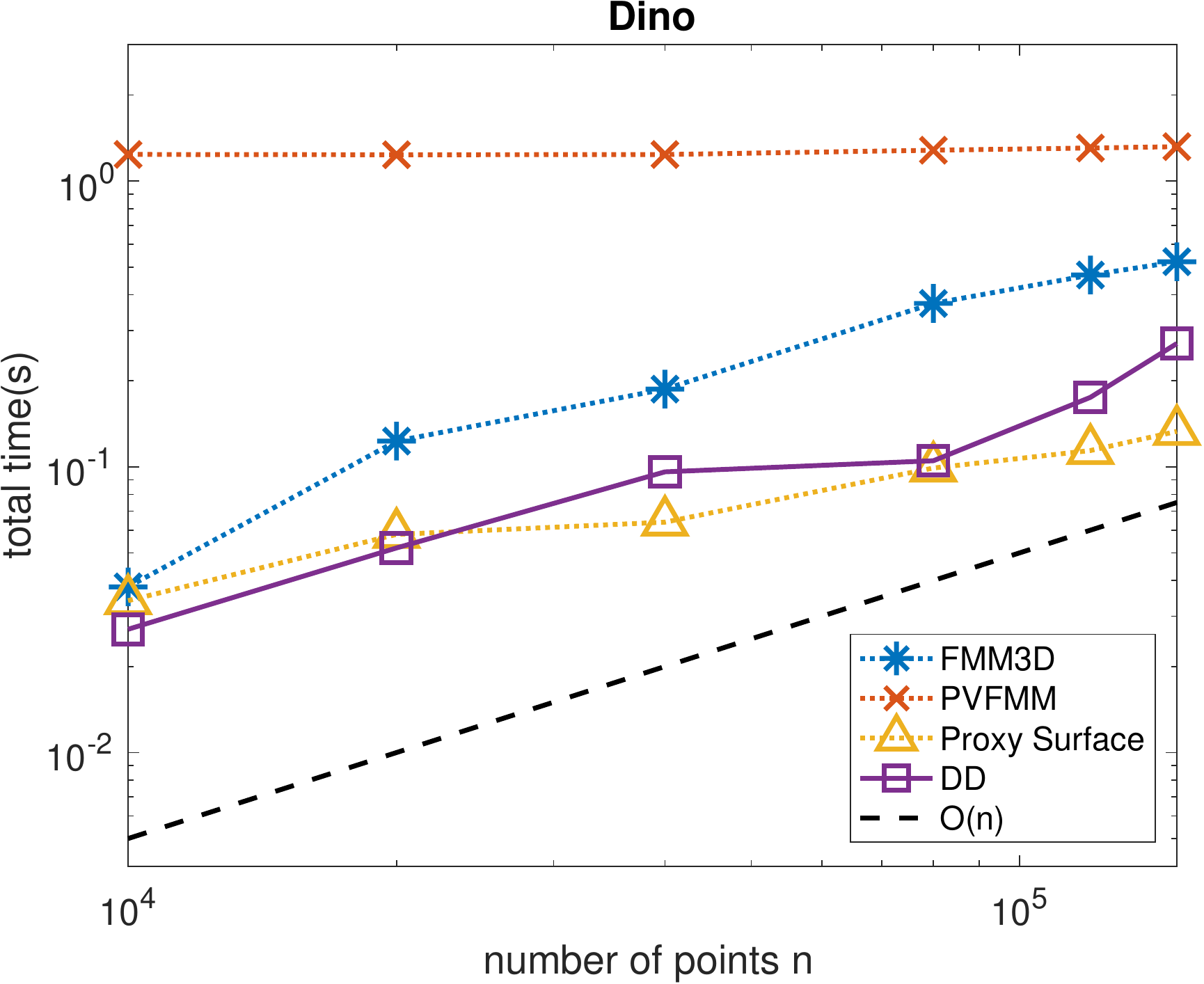}
    \caption{Section \ref{sub:Comparison special}: Total construction time of FMM3D, PVFMM, Proxy Surface and DD for approximating $n$-by-$n$ Coulomb kernel matrices with datasets Cube, 3-sphere, Dino}
    \label{fig:Geo-total}
\end{figure}

\subsection{Memory efficiency}
\label{sub:Comparison general}
%Compare to SMASH interpolation, accuracy-rank and peak memory use for different kernels 
In this section, we illustrate the memory efficiency of the proposed data-driven approach by comparing it to interpolation-based hierarchical matrix construction.
We test the two general-purpose methods for the different kernel functions listed in Table \ref{tab:kernels}. 
The ``3-sphere" dataset is used with $n=20000$ points.

In Figure \ref{fig:mem}, we plot the approximation error  vs the memory use for each method and each kernel function.
The memory use is measured by the cost for storing the hierarchical representation derived by the respective method.
The high memory use of interpolation-based construction is clearly seen from the plots. 
In three dimensions, the number of interpolation nodes is $k^3$ if $k$ interpolation nodes are used in each dimension.
This number may exceed the size of the admissible block to be approximated.
We found that, in practice, to achieve moderate to high approximation accuracy, a large number of interpolation nodes is needed.
The proposed data-driven method, on the other hand, significantly reduces the memory needed to achieve a certain approximation accuracy.
Equivalently, we can also conclude that,
for the same approximation rank and memory requirement, the data-driven method is able to provide a much more accurate hierarchical representation than the one derived from interpolation.
For large scale data, the memory efficiency of data-driven construction would be even more prominent.

\begin{figure}[htbp]
    \centering 
    \includegraphics[scale=.22]{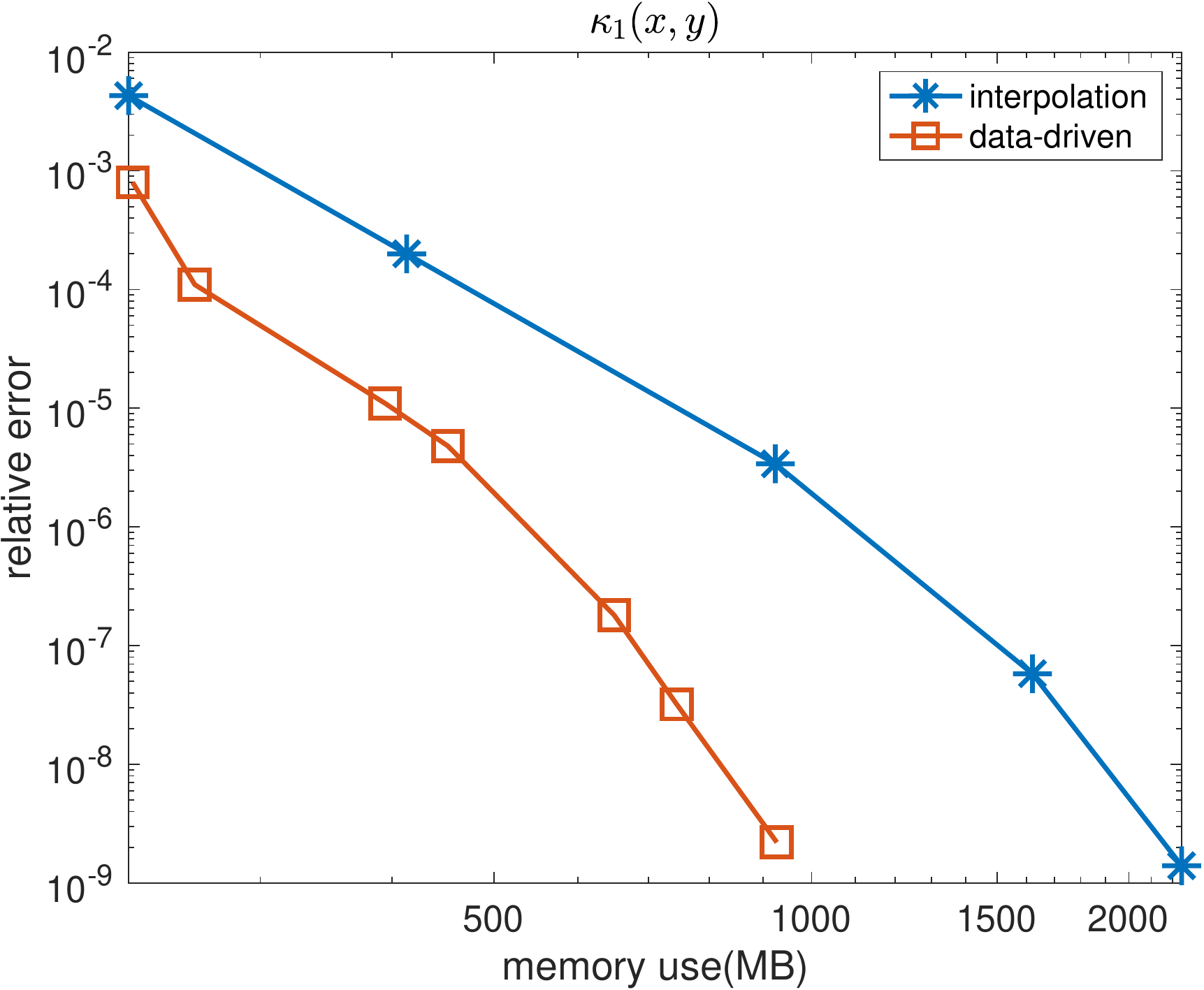} 
    \includegraphics[scale=.22]{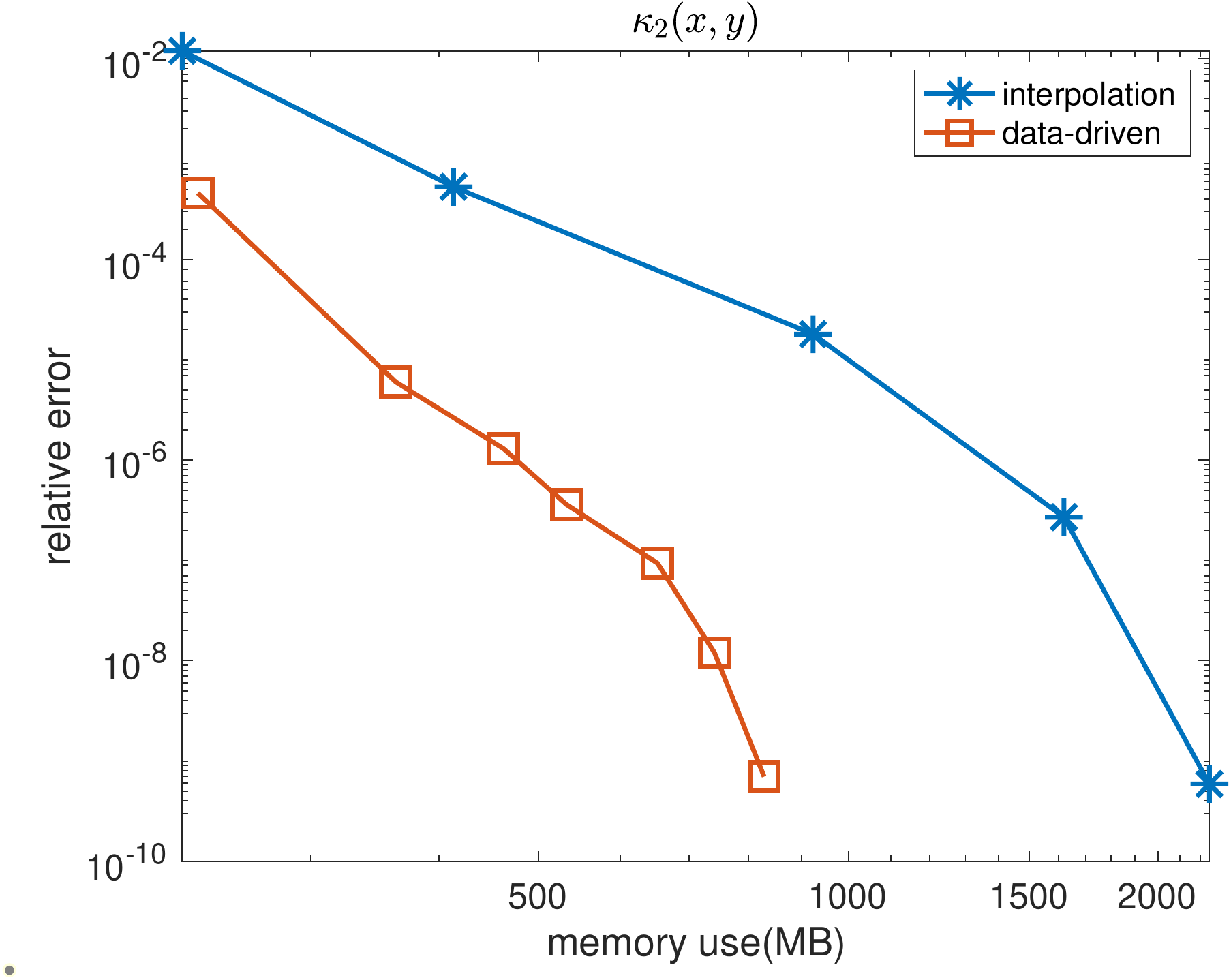} 
    \includegraphics[scale=.22]{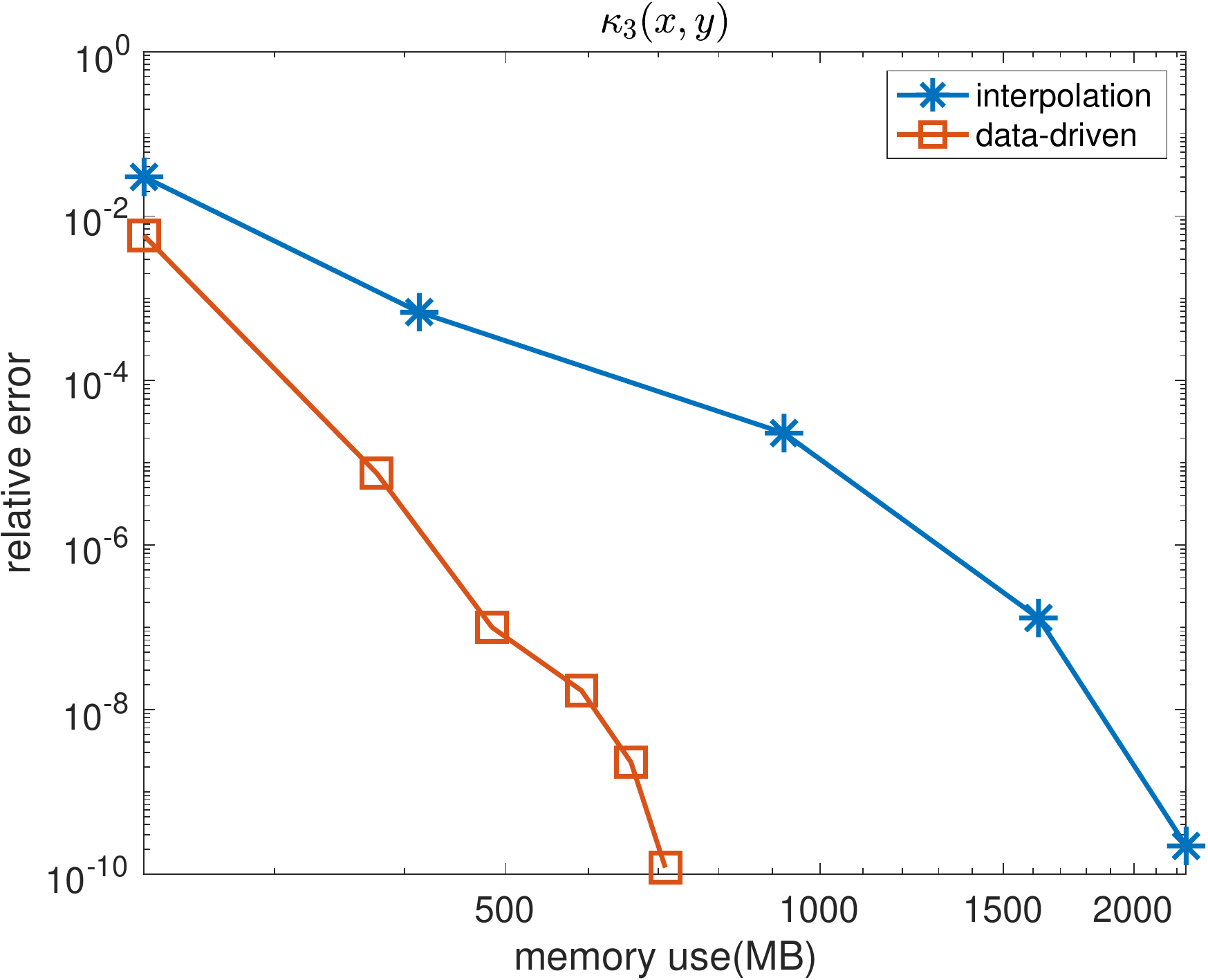} 
    \includegraphics[scale=.22]{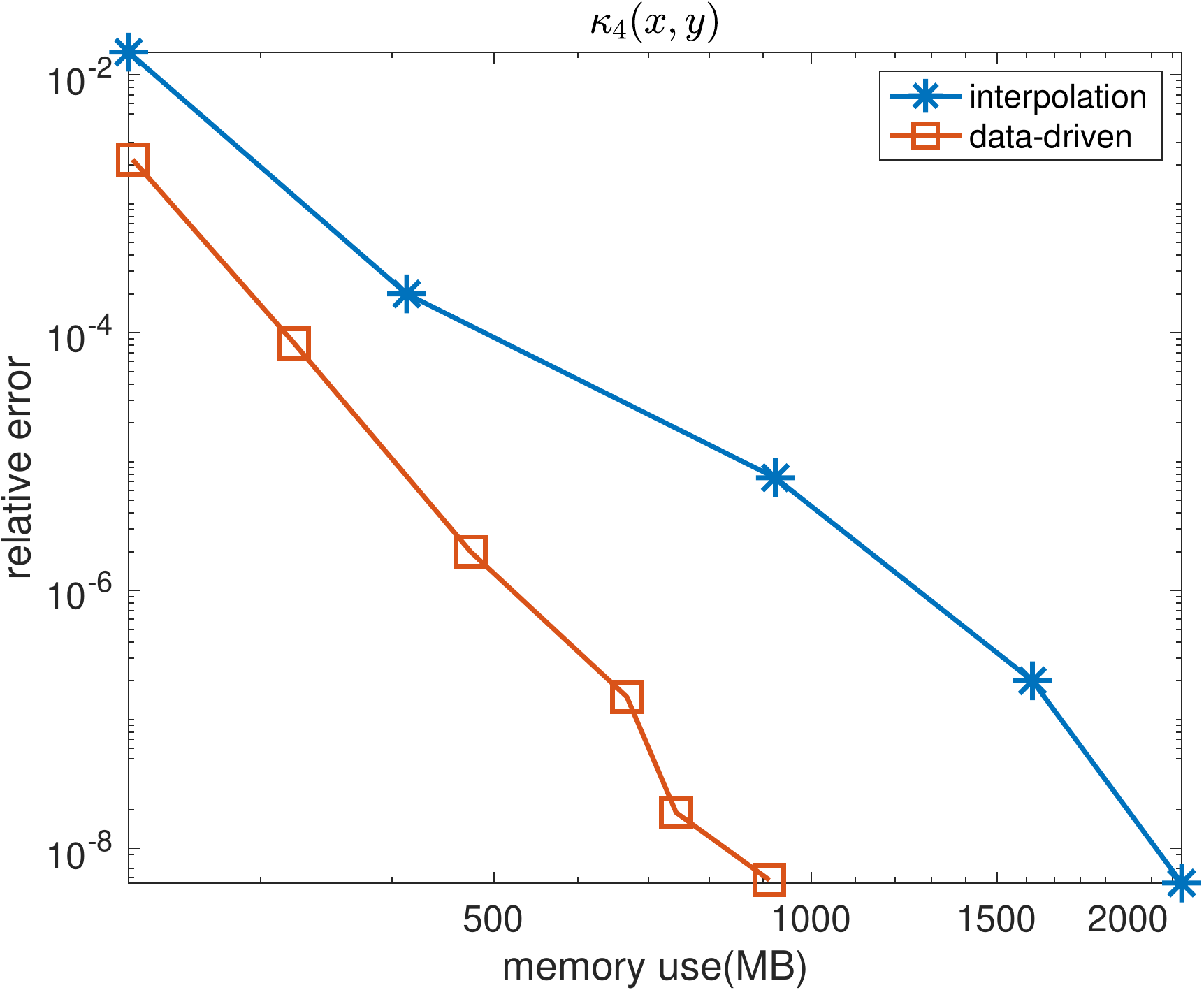} 
    \caption{Section \ref{sub:Comparison general} experiment: error vs memory use of interpolation-based and data-driven constructions for approximating the four kernel matrices (Table \ref{tab:kernels}) with the 3-sphere dataset}
    \label{fig:mem}
\end{figure}

%%%%%%%%%%%% section  (end) %%%%%%%%%%%%%%%%

\section{Conclusion} 
\label{sec:conclusion}
We proposed general-purpose data-driven hierarchical matrix construction accelerated by a novel hierarchical data reduction (HiDR).
The algorithm first computes a reduced data representation following the tree structure and then performs the hierarchical low-rank compression.
Different from all existing methods, HiDR entirely operates on the given dataset, \emph{without} accessing the kernel function or kernel matrix.
The complexity of the whole data-driven construction is linear with respect to the data size.
%HiDR computes reduced representations for each subset as well as its farfield. With reduced representations, the hierarchical matrix representation can be computed rapidly. 
%The HiDR does not require any access to the kernel function or the matrix and is shown to induce much lower cost than the subsequent hierarchical matrix construction.
Compared to general-purpose methods such as interpolation, the new data-driven framework requires less memory for the same matrix approximation accuracy.
For special kernels like the Coulomb kernel, the general data-driven method, as a black-box approach, demonstrates competitive performance when compared to specialized methods optimized for the Coulomb kernel.
The data-driven approach yields low computational cost and is particularly efficient for data sampled from low-dimensional manifolds.
Future work includes extending the data-driven framework to the efficient construction of hierarchical matrices for high dimensional data in machine learning applications.
One appealing feature is that, when constructing hierarchical matrices for different kernel functions associated with the same dataset, HiDR only needs to be performed once, which significantly reduces the total computational cost.
Additionally, the data driven procedure can be optimized towards the special kernel function under consideration. 
In the current presentation, we focus on providing a general approach, which could be useful for general-purpose library for accelerating kernel matrix computations with hierarchical matrix representations.
Possible improvements for special kernel functions will be investigated at a future date.
%%%%%%%%%%%% section  (end) %%%%%%%%%%%%%%%%

\bibliography{cdfeng}
\bibliographystyle{plain}
\end{document}